\documentclass[12pt,reqno]{amsart}
\usepackage[margin=1in]{geometry}
\usepackage{dsfont}
\makeatletter
\def\section{\@startsection{section}{1}%
	\z@{.7\linespacing\@plus\linespacing}{.5\linespacing}%
	{\bfseries
		\centering
}}
\def\@secnumfont{\bfseries}
\makeatother
\usepackage{amsmath,amssymb,amsthm,graphicx,amsxtra, setspace}
\usepackage{relsize}
\usepackage{mathrsfs}
\usepackage{alltt}
\usepackage{relsize}
\usepackage{hyperref}
\usepackage{aliascnt}
\usepackage{tikz}
\usepackage{multicol}
\usepackage{upgreek}
\usepackage{graphicx,type1cm,eso-pic,color}
\allowdisplaybreaks
\usepackage{scalerel,stackengine}
\stackMath
\newcommand\reallywidehat[1]{%
	\savestack{\tmpbox}{\stretchto{%
			\scaleto{%
				\scalerel*[\widthof{\ensuremath{#1}}]{\kern-.6pt\bigwedge\kern-.6pt}%
				{\rule[-\textheight/2]{1ex}{\textheight}}
			}{\textheight}%
		}{0.5ex}}%
	\stackon[1pt]{#1}{\tmpbox}%
}
\usepackage{graphicx,amssymb}
\usepackage{amsmath,amssymb,amsfonts}
\usepackage{hyperref}
\usepackage{amscd}
\usepackage{upgreek}
\usepackage{amsthm}
\usepackage{mathrsfs}
\usepackage{stmaryrd}
\usepackage[toc,page]{appendix}
\usepackage{enumitem}

\numberwithin{equation}{section}
\newtheorem{Thm}{Theorem}[section]
\newtheorem{Def}{Definition}[section]
\newtheorem{Lem}{Lemma}[section]
\newtheorem{Pro}{Proposition}[section]

\newtheorem{Ass}{Assumption}[section]
\newtheorem{Rem}{Remark}[section]

\renewcommand{\d}{\/\mathrm{d}\/}
\def\mR{\mathbb{R}}
\def\mH{\mathbb{H}} 
\def\mN{\mathbb{N}} 
\def\W{\mathrm{W}} 
\def\F{\mathrm{F}} 

\def\mV{\mathbb{V}}
\def\mL{\mathbb{L}}

\def\C{\mathrm{C}}

\def\S{\mathrm{S}}
\def\mE{\mathbb{E}}

\def\mZ{\mathcal{Z}}
\def\Z{\mathrm{Z}}
\def\K{\mathrm{K}}
\def\mO{\mathcal{O}}
\def\Q{\mathrm{Q}}
\def\G{\mathrm{G}}

\def\mD{\mathcal{D}}
\def\mA{\mathscr{A}} 
\def\mcL{\mathscr{L}} 
\def\macL{\mathcal{L}}
 
\def\mcY{\mathcal{Y}}
\def\mB{\mathscr{B}}
\def\mF{\mathscr{F}}
\def\mP{\mathbb{P}}
\def\mU{\mathscr{U}}
\def\A{{\bf A}} 
\def\B{{\mathrm B}}
\def\X{{\mathrm X}}
\def\Y{{\mathrm Y}}
\def\x{{ \bf x}}
\def\H{{ \mathrm{H}}} 

\def\V{ \mathrm{V}}
\def\A{ \mathrm{A}}
\def\k{{ \bf k}}
\def\y{{ \bf y}}
\def\u{{\bf u}}
\def \v{v}
\def \I{\mathrm{I}}

\def\w{w}
\def\p{{\bf p}}
\def\z{{\bf z}}
\def\U{{\mathrm U}}
\def\M{{\mathrm M}}
\def\ze{{\bf \zeta}}
\def\l{\label}
\def\si{\sigma} 
\def\lam{\lambda} 
\def\ga{\uptheta} 

\def\c{\wedge}
\def\t{\tau_N}
\def\ta{\tau_{N}}
\def\var{\varepsilon}
\def\b{\beta}
\def\al{\alpha}
\def\na{\nabla}

\def\vph{\varphi}
\def\th{\theta}
\def\om{\omega}
\def\Om{\Omega}

\def\pa{\partial}
\def\wi{\widetilde } 

\def\la{\left(}
\def\ra{\right)}
\def\iZ{\int_{\mathcal Z}}
\def\iZm{\int_{\mathcal{Z}_m}}
\def\iO{\int_{\mO}}
\def\i0t{\int_0^t}

\def\f2{\frac{1}{2}}
\def\ds{\displaystyle}

\let\originalleft\left
\let\originalright\right
\renewcommand{\left}{\mathopen{}\mathclose\bgroup\originalleft}
\renewcommand{\right}{\aftergroup\egroup\originalright}

\def\be{\begin{equation}}
\def\ee{\end{equation}}
\def\bea{\begin{eqnarray}}
\def\eea{\end{eqnarray}}
\def\bean{\begin{eqnarray*}}
	\def\eean{\end{eqnarray*}}
\def\bit{\begin{itemize}}
	\def\eit{\end{itemize}} 
\def\bi{\bibitem} 

\def\bas{\begin{assumptions}}
	\def\eas{\end{assumptions}}

\def\no{\nonumber}
\def\en{\end{document}}

\newcommand{\Tr}{\mathop{\mathrm{Tr}}}
\renewcommand{\d}{\/\mathrm{d}\/}

\def\w{\textbf{W}^{\varepsilon}_{{\theta}^{\varepsilon}}}

\def\t{t\wedge\tau_N^n}

\def\S{\mathcal{S}}

\def\A{\mathrm{A}}
\def\I{\mathrm{I}}
\def\F{\mathrm{F}}
\def\C{\mathrm{C}}
\def\f{\mathbf{f}}

\def\B{\mathrm{B}}

\def\y{\mathbf{y}}
\def\Y{\mathrm{Y}}
\def\Z{\mathrm{Z}}
\def\E{\mathbb{E}}
\def\X{\mathbb{X}}
\def\x{\mathbf{x}}
\def\z{\mathbf{z}}
\def\v{\mathbf{v}}
\def\V{\mathbb{V}}
\def\w{\mathbf{w}}
\def\W{\mathrm{W}}
\def\G{\mathrm{G}}
\def\Q{\mathrm{Q}}
\def\M{\mathrm{M}}
\def\N{\mathbb{N}}

\def\no{\nonumber}
\def\V{\mathbb{V}}
\def\wi{\widetilde}
\def\Q{\mathrm{Q}}
\def\U{\mathrm{U}}

\def\u{\mathbf{u}}
\def\H{\mathbb{H}}

\renewcommand{\d}{\/\mathrm{d}\/}

\usepackage{graphicx,amssymb}
\usepackage{amsmath,amssymb,amsfonts}
\usepackage{hyperref}
\usepackage{amscd}
\usepackage{amsthm}
\usepackage{mathrsfs}
\usepackage{stmaryrd}
\usepackage[toc,page]{appendix}
\usepackage{enumitem}

\def\mR{\mathbb{R}}
\def\mH{\mathbb{H}} 
\def\mN{\mathbb{N}} 

\def\mV{\mathbb{V}}
\def\mL{\mathbb{L}}

\def\mE{\mathbb{E}}

\def\mZ{\mathcal{Z}}
\def\mO{\mathcal{O}}

\def\mD{\mathcal{D}}
\def\mA{\mathscr{A}} 
\def\mcL{\mathscr{L}} 
\def\macL{\mathcal{L}}
 
\def\mcY{\mathcal{Y}}
\def\mB{\mathcal{B}}
\def\mF{\mathscr{F}}
\def\mP{\mathbb{P}}
\def\mU{\mathcal{U}}
\def\A{{\bf A}} 
\def\B{{\bf B}}
\def\G{{\bf G}}
\def\X{{\bf X}}
\def\Y{{\bf Y}}
\def\Z{{\bf Z}}
\def\N{{\bf N}}
\def\K{{\bf K}}
\def\W{{\bf W}}
\def\x{{ \bf x}}
\def\H{{ \bf h}} 
\def\k{{ \bf k}}
\def\y{{ \bf y}}
\def\u{{\bf u}}
\def\v{v}
\def\w{w}
\def\p{{\bf p}}
\def\z{{\bf z}}
\def\U{{\bf U}}
\def\M{{\bf M}}
\def\ze{{\bf \zeta}}
\def\l{\label}
\def\si{\sigma} 
\def\lam{\lambda} 
\def\ga{\uptheta} 

\def\c{\wedge}
\def\t{\tau_l}
\def\ta{\tau_{\tilde l}}
\def\vf{\frac{\varepsilon}{2}}
\def\va{\varepsilon}
\def\b{\beta}
\def\al{\alpha}
\def\na{\nabla}

\def\vph{\varphi}
\def\th{\uptheta}
\def\de{\delta}
\def\om{\omega}
\def\Om{\Omega}

\def\f2{\frac{1}{2}}

\def\pa{\partial}
\def\wi{\widetilde } 
\newcommand{\tr}{\mathop{\mathrm{Tr}}}

\def\la{\langle}
\def\ra{\rangle}
\def\iZ{\int_{\mathcal Z}}
\def\iO{\int_{\mO}}
\def\i0t{\int_0^t}

\def\ds{\displaystyle}

\def\be{\begin{equation}}
\def\ee{\end{equation}}
\def\bea{\begin{eqnarray}}
\def\eea{\end{eqnarray}}
\def\bean{\begin{eqnarray*}}
\def\eean{\end{eqnarray*}}
\def\bit{\begin{itemize}}
\def\eit{\end{itemize}} 
\def\bi{\bibitem} 

\def\bas{\begin{assumptions}}
\def\eas{\end{assumptions}}

\def\no{\nonumber}
\def\en{\end{document}}

\newcommand{\Addresses}{{
\footnote{
	
	\noindent \textsuperscript{1}Department of Mathematics,
	Indian Institute of Technology,  Roorkee , Uttarakhand - 47667,  INDIA.\par\nopagebreak
	\noindent  \textit{e-mail:} \texttt{manilfma@iitr.ac.in}
	
	\noindent \textsuperscript{2}Department of Mathematics, Indian Institute of Space Science and Technology (IIST),
	Trivandrum- 695 547, INDIA. \par\nopagebreak \noindent
	\textit{e-mail:} \texttt{sakthivel@iist.ac.in}
	
	\noindent \textsuperscript{3} Go.AI, Inc, Beavercreek, OH 45431, U.S.A.
	\par\nopagebreak \noindent
	\textit{e-mail:} \texttt{provostsritharan@gmail.com}
	
}}}

\begin{document}
\title[Stochastic Navier-Stokes Equations with L\'evy Noise]{Dynamic Programming of  the Stochastic 2D-Navier-Stokes Equations Forced by L\'evy Noise \Addresses}
\author[M. T. Mohan, K. Sakthivel and S. S. Sritharan]
{Manil T. Mohan\textsuperscript{1}, K. Sakthivel\textsuperscript{2}   and S. S. Sritharan\textsuperscript{3}}

\maketitle

\begin{abstract}
In this article, we study optimal feedback control synthesis of  stochastic 2D Navier-Stokes equations perturbed  L\'evy type noise with distributed stochastic  control process acting on the state equation. We use  the dynamic programming approach to solve this control problem which involves the study of second order infinite dimensional  Hamilton-Jacobi-Bellman (HJB) equation  consisting of an integro-differential operator with L\'evy measure  associated with the stochastic control problem. Using the regularizing properties of the transition semigroup corresponding to  the stochastic 2D Navier-Stokes equation,  we obtain a smooth solution in weighted function space for the HJB equation and  solve the resultant feedback control problem. 

\end{abstract}


\section{Introduction}
The study of optimal  control theory of fluid mechanics has been one of the active research areas  in applied mathematics
with  several engineering  applications. A significant interesting problem in
this direction is the rigorous study of the feedback synthesis of optimal control problems for the stochastic Navier-Stokes equations (SNSEs) forced by  random noise  using the   infinite-dimensional HJB equation associated with the problem.  The flow control of deterministic models and stochastic viscous flow problems with Gaussian noise have been extensively  analysed over the past decades. For a  rigorous study of optimal control of Navier-Stokes equations, one may look at \cite{Sr0,Fu} and cited references therein.   A systematic and general existence theory of optimal control for a class of deterministic  viscous flow problems with different geometries was done in \cite{FS} and chattering control of this problem was studied in \cite{FS1} . This control problem has been studied  in \cite{Sr1} by showing that the value function, which is the minimum for an objective functional, is the  viscosity solution of the associated HJB equation and  the authors  \cite{G} extended this analysis  for  the optimal control of 2D SNSEs with Gaussian noise.    In \cite{Da1}, the authors considered the optimal control of stochastic Burgers equation with Gaussian noise and they solved the  problem by  using dynamic programming approach joint with Hopf transformation. The  existence of optimal
controls for stochastically forced fluid flow models in two and three dimensions  with  Newtonian and non-Newtonian type constitutive relationships (\cite{SSS}) have been studied by establishing the space-time statistical solutions.   The mild form of the HJB equation associated with controlled stochastic Burgers and Navier-Stokes equations  are considered respectively  in \cite{Da}  and \cite{Da1},  and used the smoothing properties of the transition semigroup to solve the control problem.    For various aspects of the optimal control of deterministic and stochastic fluid dynamic models, one may refer to \cite{Sr1} and  also for other general systems, one can  look at the books \cite{Y,Ba,Fab}.

The onset of turbulence is often related to the randomness of background movement, for instance, structural vibrations, magnetic fields and other environmental disturbances.   The random noise term enters in the classical Navier-Stokes equation as a forcing due to these external effects and that may be incorporated either as a random boundary forcing or as a random distributed forcing acting on the right hand side of the state equation.  In fact, in the recent book \cite{BB},  a phenomenological study of fully developed turbulence  and intermittency is carried out and in which  it is nicely proposed that the experimental observations of these physical characteristics  can be modeled by SNSEs with L\'evy noise.  The existence and uniqueness results are proved  for  2D  Navier-Stokes  equations  \cite{DD,B} and a class of  stochastic partial differential equations (\cite{BZ}) with L\'evy noise.   The  qualitative properties like, ergodicity (\cite{DX,MSS}), nonlinear filtering (\cite{PFS}) and invariant measure (\cite{Do})   for  stochastic Navier-Stokes/Burgers equations with L\'evy noise have  been  studied in the literature. By applying the L\'evy type stochastic forces on the state equation,  the ergodic control of   Navier-Stokes equations  is treated in \cite{MM}  and the optimal control of tidal dynamics model with control on the initial data is discussed in the recent paper \cite{AA}. 

In this paper, we consider the  optimal control problem of minimization of turbulence specified  by a quadratic cost functional subject to the dynamics of the system governed by  2D SNSEs perturbed by  L\'evy type noise with distributed stochastic control force acting on the state equation.  This problem is solved using the dynamic programming approach, as developed in \cite{Da1}, which involves the study of  HJB  equation of  partial integro-differential type involving a compensated integral with L\'evy measure associated with the stochastic equation and smoothing properties of its transition semigroup. A smooth solution to the HJB equation is obtained  by transforming this equation  into a mild form using transition semigroup associated with the SNSEs.  The regularity of  solutions for the  mild form of the HJB equation directly depends on the smoothing properties  of the semigroup and this is achieved via the Bismut-Elworthy-Li type formula (see, \cite{E}) derived for the Navier-Stokes equation with L\'evy noise. However,  boundedness of the derivative of the transition semigroup demands the finiteness of exponential moments of the SNSEs with L\'evy noise. In order to handle this issue, we consider a transformed HJB equation with exponentially large negative weights. It is worth noting that the transformation leads to an auxiliary  HJB equation perturbed by an integral with exponential of jump noise coefficient (see \eqref{g12}-\eqref{ft}) with L\'evy measure.  This has been tackled by deriving regularity of solutions for SNSEs with L\'evy noise with negative  exponential gradients and obtain a smooth solution in weighted functions spaces by applying compactness arguments. This justifies the required smoothness of solutions for the feedback control formula and therefore by standard arguments, we prove the existence of an optimal pair for the control problem.  This work also set the foundation for reinforcement machine learning techniques for fluid mechanics(see, \cite{Be}).

\section{Mathematical Formulation} 
We consider an optimal control problem for incompressible  Navier-Stokes equations subject to  stochastic forces  in a bounded domain   $\mO\subset \mR^2,$  with smooth boundary $\pa\mO$ and Dirichlet boundary conditions:
\begin{equation}\label{2p1} 
\left\{ \begin{aligned}
&\frac{\partial \u}{\partial t}-\nu\Delta\u+(\u\cdot\na)\u+\na p
=\U +{\bf F}  \ \ \   \mbox{in} \ \ \ \mO\times (0,T) , \\[2mm]
&\na\cdot\u=0 \ \ \mbox{in}  \ \ \ \mO\times (0,T), \\[2mm]
&\u(x,t) =0 \ \ \mbox{on} \ \ \pa\mO\times [0,T],  \\ &\u(x,0)=\y(x) \ \ \mbox{in} \ \ \ \mO.  
\end{aligned}
\right.
\end{equation}
The variables  $\u=\u(x,t)$ and $p=p(x,t)$ for $(x,t)\in\mathcal{O}\times(0,T),$ denote the velocity and pressure fields, respectively. The function   $\U=\U(x,t)$  is a stochastic control process taking values in the space of square integrable functions.  It is worth noting that this kind of distributed controls may be realized by a suitable Lorentz force distribution in electrically conducting fluids such as salt water, liquid metals etc. The term ${\bf F}={\bf F}(x,t)$ denote the  external force acting on the fluid flow.  In the sequel, the coefficient of kinematic viscosity $\nu$ will be set to one by suitable scaling.

Since    the level of turbulence in a fluid flow can be characterized by the time averaged enstrophy,   it seems appropriate  to consider the minimization (over all controls from a suitable admissible set) of a  cost functional 
\begin{equation} \label{2p13a}
{\mathcal J}(0,T;\y,\U)= \mE\left[\int_0^T\int_\mO\big(|\text{curl }  \u(x,t)|^2+\frac{1}{2}|\U(x,t)|^2\big)\d x\d t+\int_\mO|\u(T,x)|^2\d x\right],
\end{equation}
where  $\text{curl }\u:=\nabla\times\u$ defines the vorticity of the flow field.  The first term in principle measures the average turbulence in the flow field through the space of square integrable functions of  the vorticity field and other two terms arise due to the mathematical technicalities of the control problem.

\subsection{Function  Spaces}
Let us define  divergence free Hilbert spaces
\begin{equation} \label{2p5}
\mH:=\big\{\mathbf{v}|\in \mL^2(\mO;\mR^2) : \na\cdot \mathbf{v}=0 \ \ \mathbf{v}\cdot {\bf n}|_{\pa\mO}=0\big\},
\end{equation}
with norm $\|\mathbf{v}\|_\mH:=\big(\int_\mO |\mathbf{v}(x)|^2\d x\big)^{1/2}$, which is  denoted as $\|\mathbf{v}\|$, where ${\bf n}$ is the outward normal to $\pa\mO.$ We also use the space
\begin{equation} \label{2p6}
\mV:=\big\{\mathbf{v}\in \mH_0^1(\mO;\mR^2) : \na\cdot \mathbf{v}=0 \big\},
\end{equation}
with norm $\|\mathbf{v}\|_\mV:=\big(\int_\mO |\na\mathbf{v}(x)|^2\d x\big)^{1/2}$ and it will be denoted as $\|\mathbf{v}\|_{\frac{1}{2}}$ to be consistent with the fractional powers  introduced later.  The inner product in the Hilbert space $\mH$ is denoted by $\la\cdot,\cdot\ra$ and the induced duality, for instance between the spaces $\mV$ and its dual $\mV^\prime,$ by $(\!( \cdot,\cdot)\!).$

Let $\mathrm{P}_\mH:\mL^2(\mO)\to \mH$ be the \emph{Helmholtz-Hodge (orthogonal) projection}.   Let us define the \emph{Stokes operator}
\begin{equation}\label{2p7}
\A: \mD(\A)\to \mH   \ \ \ \mbox{with} \ \ \A\mathbf{v}=- \mathrm{P}_\mH\Delta\mathbf{v},
\end{equation}
where $\mD(\A)=\mV\cap\mH^2(\mO)=\big\{\mathbf{v}\in \mH_0^1(\mO)\cap \mH^2(\mO) : \na\cdot \mathbf{v}=0\big\}$ and the nonlinear operator
\begin{equation}\label{2p8}
\B: \mD(\B)\subset\mH\times\mV\to \mH   \ \ \ \mbox{with} \ \ \B(\u,\mathbf{v})=\mathrm{P}_\mH(\u\cdot\na\mathbf{v}).
\end{equation}

According to the \emph{Helmholtz decomposition}, $\mL^2(\mO)$ admits an orthogonal decomposition of a sum of two non-trivial subspaces such that $\mL^2(\mO)=\mH(\mO)\oplus\mH^\perp(\mO),$ where the space $\mH^\perp$ is characterized by $\mH^\perp(\mO)=\big\{\u\in \mL^2(\mO) : \u=\na p, \ p\in \mH^1(\mO)\big\}.$

Note that with the use of the Gelfand triple $\mV\subset \mH\equiv\mH^\prime\subset\mV^\prime,$ we may consider $\A$ as the mapping from $\mV$ into $\mV^\prime.$ Besides, setting $\u=(u_i), \mathbf{v}=(v_i)$ and $\mathbf{w}=(w_i)$ for $i=1,2$, an integration by parts yields 
$$(\!( \A\u,\mathbf{w})\!)=\sum_{i,j=1}^2\iO\pa_iu_j\pa_iw_j \d x=\langle\nabla\u,\nabla\mathbf{w}\rangle=(\!( \u,\A\mathbf{w})\!), \ \ \mbox{where} \ \ \pa_iu_j=\frac{\pa u_j}{\pa x_i}.$$
Note that $\A^{-1}$ is a compact self-adjoint operator and hence by the spectral theorem(see \cite{L}),  there exists a sequence of  orthonormal basis functions $\{e_m\}$ in $\mH$ belonging to $\mathcal{D}(\A)$ and eigenvalues $\{\sigma_m\}$ accumulating at zero so that $\A^{-1}e_m=\sigma_m e_m, \ m=1,2,\ldots.$   Taking $\lambda_m=1/\sigma_m, $ we see that $\A e_m=\lambda_m e_m, \ m=1,2,\ldots.$  and 
$0<\lam_1\leq \lam_2\leq \cdots\leq \lam_m\leq \cdots\to +\infty, \ \ \mbox{as} \ \ m\to +\infty,$  in particular, $\lam_m\sim \lambda_1m$ (see, page 54, \cite{Fo}) . 

We will also use the fractional powers of $\A.$  For $\u\in \mH$ and $\alpha>0,$ let us define
$$\A^\alpha \u=\sum_{m=1}^\infty \lam_m^{\alpha} \langle \u,e_m\rangle  e_m, \ \ \mbox{for}  \ \u\in\mD(\A^\alpha),$$
where 
$$\mD(\A^\alpha)=\Big\{\u\in\mH:\sum_{m=1}^\infty \lam_m^{2\alpha}|\langle\u,e_m\rangle|^2<+\infty\Big\}.$$  
Here  $\mD(\A^\alpha)$ is equipped with the norm 
\begin{align} \label{fn}
\|\u\|_\alpha=\|\A^\alpha \u\|=\left(\sum_{m=1}^\infty \lam_m^{2\alpha}|\langle\u,e_m\rangle|^2\right)^{1/2}.
\end{align} 
In particular, note that $\mD(\A^0)=\mH,$  $\mD(\A^{1/2})=\mV.$    For any $s_1<s_2,$ the embedding $\mD(\A^{s_2})\subset \mD(\A^{s_1})$ is also compact.  In a similar manner, we can define the negative fractional powers of $\A.$ 
Applying H\"older's inequality in the  expression \eqref{fn}, one can get the following interpolation estimate:
\begin{align}\label{II}
\|\u\|_s\leq \|\u\|_{s_1}^\theta\|\u\|_{s_2}^{1-\theta}
\end{align}
for any real numbers  $ s_1\leq s\leq s_2$ and $\theta$ is given by $s=s_1\theta+s_2(1-\theta).$ 

Define trilinear form $b(\cdot,\cdot,\cdot): \mV\times\mV\times\mV\to \mR$ by the relation
\[b(\u,\mathbf{v},\mathbf{w})= \sum_{i,j=1}^2\iO u_i\pa_iv_jw_jdx.\]
Then denote by  $\B(\u,\mathbf{v})=\mathrm{P}_\mH [(\u\cdot \nabla)\mathbf{v}], \  \u,\mathbf{v} \in \mV,$ the linear continuous form from   $\mV\times\mV\to\mV^\prime$ defined by
$$(\!(\B(\u,\mathbf{v}),\mathbf{w})\!)=b(\u,\mathbf{v},\mathbf{w}), \ \  \text{ for all } \ \ \u,\mathbf{v},\mathbf{w} \in \mV.$$
For $\u=\mathbf{v},$ we write it as $\B(\u)=\B(\u,\u).$  Integrating by parts in the previous equality and using the incompressibility condition,  we also get
\begin{align}\label{2p9}
b(\u,\mathbf{v},\mathbf{w})=-b(\u,\mathbf{w},\mathbf{v})  \ \ \ \mbox{and} \ \ b(\u,\mathbf{v},\mathbf{v})=0, \ \ \text{ for all }\ \ \u,\mathbf{v},\mathbf{w} \in \mV.
\end{align}
Moreover, recall that (\cite{T})  $b(\cdot,\cdot,\cdot)$ is a trilinear continuous  form on $\mD(\A^{r_1})\times\mD(\A^{r_2+\f2})\times\mD(\A^{r_3}),$ where $r_i\geq 0,i=1,2,3$ satisfy $r_1+r_2+r_3\geq 1\ \mbox{if} \ r_i\neq 1, i=1,2,3.$   By taking $r_2=0$  and $r_1=r,$ we obtain that there exists a constant $c$ depending on $r$ and $\mO$ such that 
\begin{align} \label{ne}
|b(\u,\mathbf{v},\mathbf{w})|\leq C(\mO,r) \|\u\|_{r}\|\mathbf{v}\|_{\frac{1}{2}}\|\mathbf{w}\|_{\frac{1}{2}-r} \ \mbox{  for } \ r\in(0,1/2)  \  \mbox{and} \  \u,\mathbf{v},\mathbf{w}\in\mV. 
\end{align}

\begin{Def}
Let $(\Omega,\mF,\mP)$ be a probability space equipped with an increasing family of sub-sigma fields $\{\mF_t\}_{0\leq t\leq T}$ of $\mF$ satisfying usual conditions. 

Then  the stochastic process $\{\W(t) : 0\leq t\leq T\}$ is an $\mH$-valued  cylindrical Wiener process on $(\Omega,\mF,\{\mF_t\}_{t\geq 0},\mP)$  if and only if for arbitrary $t,$ the process $\W(t)$ can be expressed as  $\W(t)=\sum\limits_{k=1}^\infty  \beta_k(t)\xi_k,$ where $\beta_k(t), k\in \mathbb{N}$ are independent one dimensional Brownian motions on the space $(\Omega,\mF,\{\mF_t\}_{t\geq 0},\mP).$  
\end{Def}

Next we describe the \emph{Poisson random measure}. Let $(\mZ,|\cdot|)$ be a separable Banach space and $({\bf L}_t)_{t\geq 0}$ be a $\mZ$-valued L\'evy process.  For every $\om\in\Om,$ ${\bf L}_t(\om)$ has at most countable number of jumps in an interval and the jump  $\Delta  {\bf L}_t(\om):[0,T]\to \mZ$ is    defined by $\Delta {\bf L}_t(\om):= {\bf L}_t(\om)-{\bf L}_{t-}(\om)$ at $t\geq 0.$   Then
$$\uppi([0,T],\Gamma)=\#\{t\in [0,T] : \Delta {\bf L}_t(\om)\in \Gamma\}, \  \mbox{where} \  \Gamma \in \mB(\mZ\backslash\{0\}), \ \om\in\Om,$$
is the \emph{Poisson random measure} or \emph{jump measure}     associated with the L{\'e}vy process $({\bf L}_t)_{t\geq 0}.$ Here $\mB(\mZ\backslash\{0\})$ is the Borel  $\si$-field, $\uppi([0,T],\Gamma)$  is the random measure defined on $([0,T]\times (\mZ\backslash\{0\}), \mB([0,T]\times(\mZ\backslash\{0\}))),$ and $\mu(\cdot)=\mE(\uppi(1,\cdot))$ is the intensity measure defined on $((\mZ\backslash\{0\}), \mB(\mZ\backslash\{0\})).$ The intensity measure $\mu(\cdot)$ on $\mZ$ satisfies the conditions $\mu(\{0\})=0$   and 
\begin{align}\label{2p10}
\int_{\mZ} (1\wedge |z|^p)  \mu(\d z)< +\infty,\ \ p\geq 2.
\end{align} 
Then the \emph{compensated Poisson random measure} is defined by $\wi \uppi(\d t,\Gamma)=\uppi(\d t,\Gamma)-\d t\mu(\Gamma),$ where $\d t\mu(\Gamma)$ is the compensator of the L\'evy process $({\bf L}_t)_{t\geq 0}$   and  $\d t$ is the Lebesgue  measure. 
Let  $\G:[0,T]\times\mO\times \mZ \to\mH$ be a measurable and $\mF_t$-adapted process satisfying 
$$\mE\left[\left\|\int_0^T\iZ \G(t,x,z)\wi\uppi(\d t,\d z)\right\|^2\right]<+\infty.$$ Hereafter, we don't write the  explicit dependence of $\G(\cdot,\cdot,\cdot)$ in $x$ variable.   Then it is known that the  integral defined by  $M(t):=\displaystyle\int_0^t\iZ \G(s,z)\wi\uppi(\d s,\d z)$ is an  $\mH$-valued local martingale and there exist  increasing c\'adlag processes so-called quadratic variation process $[M]_t$ and Meyer process $\langle M\rangle_t$ such that $[M]_t-\langle M\rangle_t$ is a local martingale(see, \cite{Sa}).  Indeed, we have $$\mE\|M(t)\|^2=\mE[M]_t=\mE\langle M\rangle_t.$$
Moreover, the following It\^o isometry holds:
$$\mE\left[\left\|\int_0^T\iZ \G(t,z)\wi\uppi(\d t,\d z)\right\|^2\right]=\int_0^T\iZ  \|\G(t,z)\|^2\mu(\d z)\d t .$$ For more details on L\'evy process on may refer to \cite{A} and \cite{P}.

\subsection{The Hamilton-Jacobi-Bellman Equation}

We control the system state through the control process $\U:\Omega\times\mathcal{O}\times [0,T]\to \mH$ which is adapted to the filtration  $\{\mF_t\}_{t\geq 0}.$ For a fixed constant $R>0,$ we define the set of all admissible control  as 
\begin{equation}\label{3.0}
\mU^{0,T}_{R}=\Big\{\U\in \mathrm{L}^2(\Omega, \mathrm{L}^2(0,T;\mH)):\|\U(\cdot, t)\|\leq R,\mP\text{-a.s.,} \ \mbox{and}  \  \U \ \mbox{is adapted to } \{\mF_t\}_{t\geq 0}\Big\}.
\end{equation}
The control $\U$ is also subject to the action of a linear operator $\K\in {\mathcal L}(\mH).$ 

Given an initial time $t\geq 0$ and fixed terminal  time $T> t,$  the abstract controlled stochastic  Navier-Stokes equation  is given by
\begin{equation*}
\left\{
\begin{aligned} 
\d\X(t)&=-[\A\X(t)+\B(\X(t))-\K\U(t)]\d t+{\bf F}_0(t), \ \ \ t\in(0,T]   \\[2mm]
\X(0)&=\x, \ \x\in \mH,
\end{aligned}
\right.
\end{equation*}
where  ${\bf F}_0=\mathrm{P}_\mH{\bf F}.$ In this work, we consider the noise term ${\bf F}_0$ as the sum of Gaussian and L\'evy process as follows
\begin{equation}\label{2p12}
\left\{
\begin{aligned} 
\d\X(t)&=-[\A\X(t)+\B(\X(t))-\K\U(t)]\d t+\A^{-\frac{\va}{2}}\d\W(t)+ \int_{\mZ}\G(t,z) \wi\uppi(\d t,\d z),  \\[2mm]
\X(0)&=\x, \ \x\in \mH.
\end{aligned}
\right.
\end{equation}
The applied random force   $\W(\cdot)$ is  the Hilbert space valued cylindrical Wiener process (or $\frac{\d\W}{\d t}$  is the  space time white noise) and $\int_0^t\int_{\mZ}\G(t,x,z) \wi\uppi(\d t,\d z)$ is the compound Poisson process  defined by the compensated  Poisson random measure $\wi\uppi(\d t,\d z)=\uppi(\d t,\d z)-\d t\mu(\d z)$ with L\'evy measure $\mu(\cdot)$   defined over the probability space $(\Omega,\mF,\{\mF_t\}_ {t\geq 0}, \mP).$  We assume throughout the paper that $\W(\cdot)$ and $\uppi(\cdot,\cdot)$ are independent.

We have following existence and uniqueness theorem for the stochastic Navier-Stokes  equation  \eqref{2p12}.

\begin{Thm}[Theorem 2.6,  \cite{PFS}] \label{eu1} Let $\var>1$, $\G$ satisfies $\int_0^T\int_{\mZ}\|\G(t,z)\|^2\mu(\d z)\d t\leq C$  and  the control   $\U\in \mathrm{L}^2(\Om,\mathrm{L}^2(0,T;\mH)).$ If $\x\in \mH,$ then there exists \emph{a unique  solution} $\X(\cdot)=\X(\cdot;\x,\U)$ of the Navier-Stokes equation \eqref{2p12} with  $\X\in\mathrm{L}^2(\Omega;\mathrm{L}^{\infty}(0,T;\mH)\cap \mathrm{L}^2(0,T;\V)),$  and the $\mathscr{F}_t$-adapted trajectories having paths in $\mathscr{D}([0,T];\mH)\cap \mathrm{L}^2(0,T;\V),$ $\mathbb{P}$-a.s., where $\mathscr{D}([0,T];\mH)$ is the space of all c\`adl\`ag functions from $[0,T]$ to $\mH$. 
\end{Thm}
The cost   functional  associated with  the   optimal  control problem \eqref{2p12}  consists of the minimization over all controls $\U\in {\mathcal U_{R}^{0,T}} $ of the functional
\begin{align} \label{2p13b}
{\mathcal J}(0,T;\x,\U)= \mE\left\{\int_0^T\left(\|\text{curl }  \X(t)\|^2+\frac{1}{2}\|\U(t)\|^2\right)\d t+\|\X(T)\|^2\right\}.
\end{align}
We aim to find a $\wi \U\in{\mathcal U_{R}^{0,T}} $ such that ${\mathcal J}(0,T;\x,\wi \U)=\inf\limits_{\U\in {\mathcal U_{R}^{0,T}}} {\mathcal J}(0,T;\x,\U).$ We make use of dynamic programming approach to  solve this control problem which  involves the study of the \emph{value function}   defined as
$$\v(t,\x):=\inf_{\U\in {\mathcal U_{R}^{0,t}}}\mE\left\{\int_0^t\left(\|\text{curl }  \X(s)\|^2+\frac{1}{2}\|\U(s)\|^2\right)\d s+\|\X(t)\|^2\right\}, $$  
which is formally a  solution   of the following   HJB equation (see, Appendix-A) for $t\in (0,T)$:
\begin{equation} \label{2p14}
\left\{
\begin{aligned}
\mathrm{D}_t\v(t,\x)&=\mcL \v(t,\x) +g(\x)+ \inf_{\U\in\mathfrak{U}}\left(\la \U(t),\K^*\mathrm{D}_\x\v(t,\x)\ra+\frac{1}{2}\|\U(t)\|^2\right),  \\
\v(0,\x)&= \|\x\|^2, \ \x\in \mH, 
\end{aligned} 
\right.
\end{equation}
where $g(\x)=\|\text{curl} \ \x\|^2, \mathfrak{U}=B_\mH(0,R),$ a ball of radius $R$ in $\mH$  and $\mcL \v$ is the integro-differential  operator given by
\begin{align} \label{g1}
\mcL \v(t,\x)&:=\frac{1}{2}\tr(\A^{-\va} \mathrm{D}_\x^2\v(t,\x))-\la\A\x+\B(\x),\mathrm{D}_\x\v(t,\x)\ra \no\\
&\quad+\iZ\big(\v(t,\x+\G(t,z))-\v(t,\x)-\la \G(t,z),\mathrm{D}_\x\v(t,\x)\ra\big)\mu(\d z). 
\end{align}   
Here   the Hamiltonian  function  $$\mathrm{F}(\p):= \inf_{\U\in\mathfrak{U}}\left(\la \U,\p\ra+\frac{1}{2}\|\U\|^2\right)$$ can be evaluated explicitly as 
\begin{align} \label{2p15}
\mathrm{F}(\p)=\left\{\begin{array}{lclclclc}
-\frac{1}{2}\| \p\|^2   &\mbox{for}&  \| \p\|\leq R, \\
-R\| \p\|+\frac{R^2}{2}   &\mbox{for}&  \| \p\|> R.
\end{array}\right.
\end{align} 
The HJB equation \eqref{2p14} can now be written as 
\begin{equation} \label{2p16}
\left\{
\begin{aligned}
\mathrm{D}_t\v(t,\x)&=\mcL \v(t,\x)+ \mathrm{F}(\K^* \mathrm{D}_\x \v(t,\x))+ g(\x), \ \ t\in (0,T), \\[2mm]
\v(0,\x)&=\|\x\|^2, \ \x\in \mH.
\end{aligned}
\right.   
\end{equation}
Moreover, if $\v$ is a smooth solution of the HJB equation \eqref{2p14}, the optimal control is given by   $\wi\U(t)= \mathcal{G}(\K^* \mathrm{D}_\x\v(T-t,\wi \X(t))),$ where 
\begin{align} \label{2p17}
\mathcal{G}(\p)=\mathrm{D}_\p (\mathrm{F}(\p))=\left\{\begin{array}{lclclclc}
-\p   &\mbox{for}&  \| \p\|\leq R, \\
-R\frac{ \p}{\| \p\|}  &\mbox{for}&  \|\p\|> R.
\end{array}\right.
\end{align}
Here $\wi\X(t)$ is the optimal solution of the following uncontrolled Navier-Stokes equation
\begin{equation} \label{2p18}
\left\{
\begin{aligned}
\d\wi\X(t)&=-[\A\wi\X(t)+\B(\wi\X(t))-\mathrm{D}_\p \mathrm{F}(\mathbf{K}^*\mathrm{D}_\x\v(T-t,\wi \X(t)))]\d t+\A^{-\frac{\va}{2}}\d\W(t)\\
&\quad + \int_{\mZ}\G(t,z) \wi\uppi(\d t,\d z), \ \ t\in (0,T),  \\
\wi\X(0)&=\x\in \mH.
\end{aligned}
\right.
\end{equation}
The pair $(\wi\X, \wi\U)$ is the optimal pair of the control problem.   In order to obtain such a smooth solution to \eqref{2p16}, we use the transition semigroup  $(\mathrm{S}_t)_{t\geq 0}$ defined on    $(\Omega,\mathscr{F},\{\mF_t\}_{t\geq 0},\mP)$    of the   uncontrolled stochastic Navier-Stokes equations   associated with \eqref{2p12}. The semigroup is defined by $(\mathrm{S}_tf)(\x)=\mE[f(\Y(t,\x))],$ where  $f:\mH\to\mR$ is a mesurable function and $\Y(t,\x)$ is the solution of the following uncontrolled Navier-Stokes equation:
\begin{align} \label{2p19}
\left\{\begin{array}{lclclc}
\d\Y(t)=-[\A\Y(t)+\B(\Y(t))]\d t+\A^{-\frac{\va}{2}}\d\W(t)
\displaystyle+ \int_{\mZ}\G(t,z) \wi\uppi(\d t,\d z),   t\in (0,T), \\[2mm]
\Y(0)=\x\in \mH.
\end{array}
\right.
\end{align}
Under the transition semigroup, the value function $\v$ solves the following equation  in  mild form 
\begin{align}  \label{2p19a}
\v(t,\x)=\mathrm{S}_tf(\x)+\int_0^t\mathrm{S}_{t-s}\mathrm{F}(\mathbf{K}^*\mathrm{D}_\x\v(s,\x))\d s+\int_0^t\mathrm{S}_{t-s}g(\x)\d s,
\end{align}
with  $f(\x)=\|\x\|^2.$

\begin{Ass}\label{ass2.1} The following are the crucial conditions imposed throughout the paper:
\bit 
\item[ $(H_1)$] For some fixed constant $\th>0$ and for any fixed $p\geq 2,$   the jump noise coefficient $\G$  satisfies  the following exponential bound: 
$$\int_0^T\int_{\mZ}(1+\|\A^{\al_1}\G(t,z)\|)^p\exp\left(2\th\|\G(t,z)\|^2\right)\mu(\d z)\d t\leq C,\ \ \al_1\in [0,1/2].$$
\item[ $(H_2)$] The linear operator $\K$ acting on the control satisfies 
$\|\K^\ast \x\|\leq C_\K\|\A^{-\widetilde{\al}_1}\x\|$   for any  $\x\in \mD(\A^{-\widetilde{\al}_1})$  and  $\al_1<{\widetilde{\al}_1}<\frac{1}{2}.$ 
\eit  
\end{Ass}

\begin{Rem}
Note that if we take $\G$ to be independent of $z$ and $t$ and $\G\in D(A^{\alpha_{1}})$ then we get the special case of $(H_1)$ where the jump measure is a Poisson random measure with intensity measure $\mu$.
\end{Rem}
\begin{Rem}
In order to prove the smoothing property of the transition semigroup, we appeal to Bismut-Elworthy-Li type formula that produces unbounded covariance operator. By properly choosing the value of $\var,$ we are able to show that the semigroup is differentiable in a subspace of $\mH.$  This in turn demands the smoothness  of the control operator  as given in $(H_2).$
\end{Rem}

Next we define the functional spaces on $\mD(\A^\al), \al\geq 0$ having polynomial growth.   If   $\psi:\mH \to \mathbb{R}$ and $ \x,\H\in \mH,$  we set $$\la\mathrm{D}_\x\psi(\x), \H\ra=\lim_{\tau\to 0} \frac{1}{\tau}[\psi(\x+\tau \H)-\psi(\x)], \  \mbox{ if the limit }   \mbox{exists}. $$
The following are the functional spaces frequently used in the rest of the paper.  
\bit
\item The set $\mathrm{C}^0_b(\mH;\mathbb{R})$ denotes the space of all continuous and bounded functions from $\mH$ to $\mathbb{R}$ endowed with the norm $$\|\psi\|_0=\sup_{\x\in\mH}|\psi(\x)| \  \mbox{for any} \ \psi\in \mathrm{C}^0_b(\mH;\mathbb{R}).$$

\item For  $\al=0$ and $k\in\mN,$ the space    $\mathrm{C}^{0, k}(\mH;\mathbb{R})$ denotes the set of all weighted functions  from $\mH$ to $\mathbb{R}$ such that 
$$\|\psi\|_{0,k}=\sup_{\x\in\mH}\frac{|\psi(\x)|}{(1+\|\x\|)^k} <+\infty.$$   

\item  For  $\al=0,\b\in \mN, k\in \mN$ and $\gamma\in [0,1],$ the space    $\mathrm{C}^{0, k,\b+\gamma}(\mH;\mathbb{R})$ denotes the set  of    $\beta$ times differentiable functions such that $\mathrm{D}_{\x}^\b\psi$ is $\gamma$-H\"older continuous satisfying
$$\|\psi\|_{0,k,\b+\gamma}=\|\psi\|_{0,k} + \sup_{r>0}\frac{1}{(1+r)^{k}} [\mathrm{D}_{\x}^\b\psi]_{\gamma} <+\infty,$$
where 
$$[\mathrm{D}_{\x}^\b\psi]_{\gamma}:=\sup_{\substack{\x,\y\in \mathrm{B}_r,\    \x\neq \y}} \frac{|\mathrm{D}_{\x}^\b\psi(\x)-\mathrm{D}_{\y}^\b\psi(\y)|_{{\mathcal L^\b}(\mH^\b;\mR)}}{\|\x-\y\|^\gamma}$$
and $\mathrm{B}_r=\{\x,\y\in\mH; \|\x\|\leq r, \|\y\|\leq r\}.$  

\item For  $\al>0,$   $\mathrm{C}^{\al, k,0}(\mD(\A^\al);\mathbb{R})$ is the space of all functions  from $\mD(\A^\al)$ to $\mathbb{R}$ such that 
$$\|\psi\|_{\al,k,0}=\sup_{\x\in\mD(\A^\al)}\frac{|\psi(\x)|}{(1+\|\x\|_\al)^k} <+\infty.$$

\item Moreover,   $\mathrm{C}^{\al, k,\b+\gamma}(\mD(\A^\al);\mathbb{R})$ is the space of all functions  from $\mD(\A^\al)$ to $\mathbb{R}$ such that  $\psi(\A^{-\al} \cdot)\in \mathrm{C}^{0, k,\b+\gamma}(\mH;\mathbb{R})$ and $$\|\psi\|_{\al,k,\b+\gamma}=\|\psi(\A^{-\al} \cdot)\|_{0,k,\b+\gamma}.$$
In particular note that when $\b=0$ and $\gamma=1,$ we can write
$$\|\psi\|_{\al,k,1}=\|\psi\|_{\al,k,0}+\sup_{\x\in\mD(\A^\al)}\frac{1}{(1+\|\x\|_\al)^k}\sup_{\H\in\mD(\A^\al)}\frac{|\la \mathrm{D}_\x\psi(\x),\H\ra|}{\|\H\|_\al},$$
and when $\b=1$  and $\gamma=1$,
$$\|\psi\|_{\al,k,2}=\|\psi\|_{\al,k,0}+\sup_{\x\in\mD(\A^\al)}\frac{1}{(1+\|\x\|_\al)^k}\sup_{\H\in\mD(\A^\al)}\frac{|\mathrm{D}_\x^2\psi(\x)\la \H,\H\ra|}{\|\H\|_\al^2}.$$
\eit 
We use the following interpolation result in several places (see, \cite{Da1}). 
\begin{Lem}\label{lem2.1}
Let $k,k_1,k_2\in\mN.$ Further let $ \b_i\geq 0,$  $\gamma_i\geq 0,i=1,2$ with $(k,\b+\gamma)=\lambda( k_1,\b_1+\gamma_1)+ (1-\lambda)( k_2,\b_2+\gamma_2)$ for any   $\lambda\in[0,1].$  If $\psi\in \mathrm{C}^{\al, k_1,\b_1+\gamma_1}\cap \mathrm{C}^{\al, k_2,\b_2+\gamma_2}$ then $\psi\in \mathrm{C}^{\al, k,\b+\gamma}$ and there exists a constant $c>0$ depending on $\lambda,\b_i,\gamma_i,i=1,2$  satisfying  
\begin{align} \label{I}
\|\psi\|_{\al,k,\b+\gamma} \leq C \|\psi\|^\lambda_{\al,k_1,\b_1+\gamma_1} \|\psi\|^{ (1-\lambda)}_{\al,k_2,\b_2+\gamma_2}. 
\end{align}
\end{Lem}
The following are the two main results proved in this paper. The first one concerns with the existence of smooth solutions of the HJB equation \eqref{2p16} in the weighted function spaces.
\begin{Thm}\label{main1}
Let the assumptions ($H_1$)-($H_2$) hold true. For any $\al_1<\al<\wi \al_1$ and $d\geq 0$, there exists a function $\v\in\C([0,T];\C^{\alpha,d,2})$ such that $\v$ satisfies the mild form \eqref{2p19a} of the HJB equation \eqref{2p16} in a ball in $\mH$.
\end{Thm}

The existence of smooth solution of the HJB equation \eqref{2p16} justifies the feedback control formula  given by $\wi\U(t)= \mathcal{G}(\K^* \mathrm{D}_\x\v(T-t,\wi \X(t)))$.  Using this, we prove the existence of optimal control for the control problem \eqref{2p12}-\eqref{2p13b}, which justifies the dynamic programming approach. Moreover, we establish an identity satisfied by the cost functional in terms of the solution of the HJB equation \eqref{2p16}. 

\begin{Thm}\label{thm3.3}
Suppose the conditions given in Theorem \ref{main1} are satisfied. Then for any control $\U\in\mathcal{U}_{R}^{0,T}$, the following identity holds:
\begin{align}\label{3.122a}
&	\mathcal{J}(0,T;\x,\U)\\&=\v(T,\x)+\frac{1}{2}\mE\left[\int_0^T\|\U(t)+\K^*\mathrm{D}_{\x} \v(T-t,\X(t))\|^2-\chi(\|\K^*\mathrm{D}_{\x} \v(T-t,\X(t))\|-R)\d t\right],\nonumber
\end{align}
where $\mathcal{J}(0,T;\x,\U)$ is defined in (\ref{2p13b}), the function $\chi$ satisfies  $\chi(a)=0$ for $a\leq 0$ and $\chi(a)=a^2$ for $a\geq 0$ and $\X(\cdot)$ is the solution of (\ref{2p12}). Moreover, the closed loop equation (\ref{2p18}) has an  optimal pair  $(\wi\X,\wi\U)$ with $\wi \U(t)= {\mathcal G}(\K^* \mathrm{D}_{\x}\v(T-t,\wi \X(t)))$.
\end{Thm}

These two theorems are proved in sections \ref{sec5} and \ref{sec6}, respectively. 

\section{Finite Dimensional  Approximation}  
We introduce an approximation for the controlled Navier-Stokes systems \eqref{2p12}.  Let $\{e_1,e_2,\ldots,e_m\}$ be the first $m$ eigenvectors of $\A$ and $\mathrm{P}_m$ be the projector of  $\mH$ onto the space spanned by these $m$ eigenvectors. Then $\X_m(t)=\sum\limits_{i=1}^mc_i(t)e_i$ solves the following finite dimensional control system:
\begin{equation}\label{2F1}
\left\{
\begin{aligned} 
\d\X_m(t)&=-[\A\X_m(t)+\B_m(\X_m(t))-\K_m\U_m(t)]\d t+\A^{-\frac{\va}{2}}\d\W_m(t)\\
&\quad +\displaystyle \int_{\mZ_m}\G_m(t,z) \wi\uppi(\d t,\d z),  \ t\in [0,T), \\
\X_m(0)&=\mathrm{P}_m\x, \ \x\in \mH,
\end{aligned}
\right.
\end{equation}
where $\B_m(\X_m(t))=\mathrm{P}_m\B(\mathrm{P}_m\X_m),\K_m=\mathrm{P}_m\K , \W_m=\mathrm{P}_m\W $, $\G_m=\mathrm{P}_m\G$ and $\mZ_m=\mathrm{P}_m\mZ$.  The approximated cost functional is the minimization over  all $\U_m\in {\mathcal U_{R}^{0,T}} \cap \mathrm{L}^2(\Omega\times [0,T];\mathrm{P}_m\mH) $ of the functional
\begin{align} \label{2F2}
{\mathcal J_m}(0,T;\mathrm{P}_m\x,\U_m)= \mE\left\{\int_0^T\left(\|\text{curl }  \X_m(t)\|^2+\frac{1}{2}\|\U_m(t)\|^2\right)\d t+\|\X_m(T)\|^2\right\}.
\end{align}
The finite dimensional  equation associated with \eqref{2p19} is given by 
\begin{equation}\label{2F3}
\left\{
\begin{aligned} 
\d\Y_m(t)&=-[\A\Y_m(t)+\B_m(\Y_m(t))]\d t+\A^{-\frac{\va}{2}}\d\W_m(t) \\
&\quad+ \int_{\mZ}\G_m(t,z) \wi\uppi(\d t,\d z), \   t\in (0,T),  \\
\Y_m(0)&=\mathrm{P}_m\x, \ \x\in \mH.
\end{aligned}
\right.\end{equation}
Then  by the semigroup $(\mathrm{S}^m_tf)(\x)=\mE[f(\Y_m(t,\x))],$  we have the following approximated HJB equation:
\begin{equation}\label{3.4}
\left\{
\begin{aligned}
\mathrm{D}_t\v_m(t,\x)&=\mcL_m \v_m(t,\x)+ \mathrm{F}_m(\K^* \mathrm{D}_\x \v_m(t,\x))+ g_m(\x), \ \ t\in (0,T), \\[2mm]
\v(0,\x)&=\|\x\|^2, \ \x\in \mathrm{P}_m\mH,
\end{aligned}
\right.   
\end{equation}
and the corresponding mild  form:
\begin{align}  \label{2p20}
\v_m(t,\x)=(\mathrm{S}^m_tf)(\x)+\int_0^t\mathrm{S}^m_{t-s}\mathrm{F}_m(\K^*\mathrm{D}_\x\v_m(s,\x))\d s+\int_0^t\mathrm{S}^m_{t-s}g(\x)\d s.
\end{align} 
As we are looking for a smoothness of the Feynman-Kac semigroup, we appeal to Bismut-Elworthy-Li type formula for the c\`adl\`ag process which in turn requires the differential of $\Y_m$ of the system \eqref{2F3} with respect to initial data. 

The first differential of $\Y_m$ in the direction of $\H \in \mathrm{P}_m \mH$   defined by the form  $\eta_m^\H(t)=\la \mathrm{D}_\x\Y_m(t,\x),\H\ra $ satisfies 
\begin{equation}\label{d1}
\left\{
\begin{aligned} 
\mathrm{D}_t\eta_m^\H(t)&=-\A\eta_m^\H(t)-\B_m(\eta_m^\H(t),\Y_m(t))-\B_m(\Y_m(t),\eta_m^\H(t)), \ t\in(0,T), \\
\eta_m^\H(0)&=\H.
\end{aligned}
\right.
\end{equation}
The second  differential of $\Y_m$ in the direction of $\H,\k \in \mathrm{P}_m \mH$    defined by the form $\ze_m^{\H,\k}(t)= \mathrm{D}^2_\x\Y_m(t,\x)\la\H,\k\ra $ satisfies 
\begin{equation}\label{d2}
\left\{
\begin{aligned} 
\mathrm{D}_t\ze_m^{\H,\k}(t)&=-\A\ze_m^{\H,\k}(t)-\B_m(\ze_m^{\H,\k}(t),\Y_m(t))-\B_m(\Y_m(t),\ze_m^{\H,\k}(t))\\
&\quad-2\B_m(\eta_m^\H(t),\eta_m^\H(t)), \ t\in(0,T), \\
\ze_m^{\H,\k}(0)&=0.
\end{aligned}
\right.
\end{equation}
For simplicity, we will only use  $\eta_m,\ze_m$ for   $\eta_m^\H$ and $\ze_m^{\H,\k}$ respectively.  Recall that  the Bismut-Elworthy-Li type formula  gives the derivative of the semigroup $S_t$ when the system \eqref{2F1} is forced by L\'evy noise is as follows:
$$\la \mathrm{D}_\x (\S^m_tf)(\x),\H\ra=\frac{1}{t}\mE\left[f(\Y_m(t,\x))\int_0^t\la\A^{\frac{\va}{2}}\eta_m(s),\d\W_m(s)\ra\right], \ \  \H\in \mathrm{P}_m\mH.$$

In order to get finite expectation on the right hand side, we will effectively use the  transformation (see, \cite{Da1})  $\w_m(t,\x)=e^{-\ga \|\x\|^2}\v_m(t,\x)$ for sufficiently large $\ga>0$ as chosen later to arrive at 
\begin{eqnarray}\label{2p21}
\left\{
\begin{aligned} 
\mathrm{D}_t\w_m(t,\x)&=\mcL_m \w_m(t,\x)-2\ga\|\x\|_{\frac{1}{2}}^2\w_m(t,\x)+\widetilde{\mathrm{F}}_m(\x,\w_m(t,\x),\mathrm{D}_\x \w_m(t,\x))+ \wi g(\x),  \\
\w_m(0,\x)&=\wi f(\x), \ \x\in \mathrm{P}_m\mH,
\end{aligned}
\right.
\end{eqnarray}
for $t\in(0,T)$. Here, $\wi g(\x)=e^{-\ga\|\x\|^2}\|\text{curl} \ \x\|^2, \ \wi f(\x)=e^{-\ga\|\x\|^2}\|\x\|^2,$
\begin{align} \label{g12}
\mcL_m \w_m(t,\x)&=\frac{1}{2}\tr(\A^{-\va} \mathrm{D}_\x^2\w_m(t,\x))-\la\A\x+\B(\x),\mathrm{D}_\x\w_m(t,\x)\ra \\
&\quad+\int_{\mathcal{Z}_m}\big[\w_m(t,\x+\G_m(t,z))-\w_m(t,\x)-\la \G_m(t,z),\mathrm{D}_\x\w_m(t,\x)\ra\big]\mu(\d z)\no
\end{align}  
and 
\begin{align}\label{ft}
&\widetilde{\mathrm{F}}_m(\x,\w_m,\mathrm{D}_\x \w_m)\no\\&=2\th\la\A^{-\va}\x,\mathrm{D}_\x\w_m\ra+(4
\ga^2\|\A^{-\frac{\va}{2}}\x\|^2+2\ga\tr(\A^{-\va}))\w_m+e^{-\ga\|\x\|^2}\mathrm{F}(e^{\ga\|\x\|^2}(\mathbf{K}^*\mathrm{D}_\x\w_m+2\ga\mathbf{K}^*\w_m\x))\no\\
&\quad +\iZm\big[(e^{\ga \|\x+\G_m\|^2-\ga\|\x\|^2}-1)\w_m(t,\x+\G_m(t,z))-2\th\la \G_m(t,z),\w_m\x\ra\big]\mu(\d z).
\end{align}
Further, consider the Feynman-Kac semigroup defined by  
$$(\mathrm{T}^m_tf)(\x)=\mE\big[e^{-2\ga\mathcal{Y}_m(t)}f(\Y_m(t,\x))\big] \ \mbox{for} \   f\in \mathrm{C}_b(\mathrm{P}_m\mH;\mR), $$  where  $\mathcal{Y}_m(t)=\int_0^t\|\Y_m(r)\|_{\frac{1}{2}}^2\d r$ and  $\Y_m$ is a solution of the stochastic equation \eqref{2F3}. Then by the Feynman-Kac formula,  it is clear that $\varphi_m(t,\x)= (\mathrm{T}^m_t\wi f)(\x)$ formally solves the following auxiliary Kolmogorov equation 
\begin{align}  \label{2p22}
\left\{
\begin{array}{lclclclc}
\mathrm{D}_t\varphi_m(t,\x)=\mcL_m\varphi_m(t,\x)-2\ga\|\x\|_{\frac{1}{2}}^2\varphi_m(t,\x),  \ \ t\in (0,T),  \\
\varphi_m(0,\x)=   \wi f(\x), \ \x\in \mathrm{P}_m\mH.
\end{array}
\right.   
\end{align}
Besides, the mild form of \eqref{2p21} can now be written as 
\begin{align}  \label{2p23}
\w_m(t,\x)=\mathrm{T}^m_t\wi f(\x)+\int_0^t\mathrm{T}^m_{t-s}\widetilde{\mathrm{F}}_m(\x,\w_m,\mathrm{D}_\x\w_m(s,\x))\d s+\int_0^t\mathrm{T}^m_{t-s}\wi g(\x)\d s.
\end{align}
Note that the semigroup $(\mathrm{T}_t)_{t\geq 0}$ behaves smoother than the semigroup $(\mathrm{S}_t)_{t\geq 0}$ due to the  exponential inside the expectation with large enough negative weights.  

Now we state the Bismut-Elworthy-Li type formula  for the Feynman-Kac semigroup  $\mathrm{T}^m_t$ involving the   It\^o -L\'evy  process $\Y_m$ of the system \eqref{2F3}.   

\begin{Pro}[see, Appendix C, \cite{MSS}]\label{p2.1}
For each $f\in \mathrm{C}_b(\mathrm{P}_m \mH),$  the semigroup $(\mathrm{T}^m_t)_{t\geq0}$ is G\^ateaux differentiable  and  its directional derivative in any direction $\H\in \mathrm{P}_m \mH$ is given by the following: 
\begin{align} \label{BE}
\la \mathrm{D}_\x (\mathrm{T}^m_tf)(\x),\H\ra&=\mE\left[e^{-2\ga\mathcal{Y}_m(t)}f(\Y_m(t,\x))\left(\frac{1}{t}\int_0^t\la\A^{\frac{\va}{2}}\eta_m(s),\d\W_m(s)\ra\right.\right.\no\\
&\quad\left.\left.+4\ga\int_0^t\left(1-\frac{s}{t}\right)\la\A^{\f2}\eta_m(s),\A^{\f2}\Y_m(s)\ra \d s\right)\right],  
\end{align}
for $\x \in \mathrm{P}_m\mH,  \ t\in[0,T],$  where  $\mathcal{Y}_m(t)=\int_0^t\|\Y_m(r)\|_{\frac{1}{2}}^2\d r.$
\end{Pro}
\begin{Rem}
Differentiating \eqref{BE} further with respect to the initial data, we arrive at the following second differential of the semigroup $\mathrm{T}_t^m$ for $\H\in \mathrm{P}_m\mH$:
\begin{align} \label{BE6}
&\mathrm{D}^2_\x (\mathrm{T}^m_tf)(\x)\la\H,\H\ra\no\\&=\mE\left[e^{-2\ga\mathcal{Y}_m(t)}f(\Y_m(t)\left\{\frac{1}{t}\int_0^t\la \A^{\frac{\va}{2}}\zeta_m(s),\d\W_m(s)\ra\right.\right.\no\\
&\quad\left.\left.+4\ga\int_0^t\left(1-\frac{s}{t}\right)  \left[  \la\A^{\f2}\zeta_m(s),\A^{\f2}\Y_m(s)\ra + \la\A^{\f2}\eta_m(s),\A^{\f2}\eta_m(s)\ra\right] \d s\right\}\right] \no \\
&\quad+\mE\left[e^{-2\ga\mathcal{Y}_m(t)}f(\Y_m(t))\left\{-4\ga \int_0^t \la\A^{\f2}\eta_m(s),\A^{\f2}\Y_m(s)\ra f(\Y_m(t))+\la \mathrm{D}_\x f(\Y_m(t)), \eta_m(t)\ra\right\}\right.\no\\&\qquad\times \left.\left\{\frac{1}{t}\int_0^t\la \A^{\frac{\va}{2}}\eta_m(s),\d\W_m(s)\ra+4\ga\int_0^t\big(1-\frac{s}{t}\big)\la\A^{\f2}\eta_m(s),\A^{\f2}\Y_m(s)\ra \d s\right\}\right].
\end{align}
\end{Rem}

\begin{Pro}\label{main0}
Let the assumptions ($H_1$)-($H_2$) hold true. For any $\al_1<\al<\wi \al_1$ and $d\geq 0$, there exists a function $\w\in\C([0,T];\C^{\alpha,d,2})$ such that $\w$ satisfies the following mild form 
\begin{align}  \label{349}
\w(t,\x)=\mathrm{T}_t\wi f(\x)+\int_0^t\mathrm{T}_{t-s}\widetilde{\mathrm{F}}(\x,\w,\mathrm{D}_\x\w(s,\x))\d s+\int_0^t\mathrm{T}_{t-s}\wi g(\x)\d s.
\end{align}
of the  infinite dimensional HJB equation
\begin{eqnarray}\label{315}
\left\{
\begin{aligned} 
\mathrm{D}_t\w(t,\x)&=\mcL \w(t,\x)-2\ga\|\x\|_{\frac{1}{2}}^2\w(t,\x)+\widetilde{\mathrm{F}}(\x,\w(t,\x),\mathrm{D}_\x \w(t,\x))+ \wi g(\x), \ t\in[0,T],  \\
\w(0,\x)&=\wi f(\x),  \mbox{  with $\x$ restricted to a ball in $\mH$}.
\end{aligned}
\right.
\end{eqnarray}
\end{Pro}
We prove this proposition in section \ref{sec5}. 
\subsection{ A Priori Estimates}
In order to estimate \eqref{BE}  and \eqref{BE6}, we are in need of the following estimates.  The proof can be completed by the similar arguments as in \cite{Sa}.

\begin{Lem} \label{L2}
Suppose the assumption $(H_1)$ is satisfied. Then for any  $p\geq 2$ and $\va>1,$ there exists a constant $C>0$  such that for any $t\in[0,T]$ and $\x\in \mH$: 
\begin{align} \label{2p31}
\mE\left[\sup_{t\in[0,T]}\|\Y_m(t)\|^{p} + \int_0^T  \|\Y_m(t)\|^{p-2}\|\Y_m(t)\|_{\frac{1}{2}}^2 \d t  \right]
\leq C(p,T)(1+\|\x\|^{p}).  
\end{align}
\end{Lem}

\begin{Rem}
Note that the proof of Lemma \ref{L2} requires that $\tr(\A^{-\va})<+\infty$. Indeed, for the 2D NSE, using Corollary 2.2, \cite{AAI}, we have $$\tr(\A^{-\va})=\sum_{k=1}^\infty\lambda_k^{-\va}\leq \left(\frac{|\mathcal{O}|}{2\pi}\right)^{\var} \sum_{k=1}^\infty k^{-\va}<+\infty, \ \mbox{ for any} \ \  \va>1,$$  where $|\mathcal{O}|$ is the $2$-dimensional Lebesgue measure of $\mathcal{O}$.  
\end{Rem}

We also need the following estimate in the sequel. 
\begin{Lem} \label{L3}
Suppose the assumption $(H_1)$ is satisfied and let $\va>1.$   For any $k\geq 2$,  there exists a constant $C>0$  such that for any $t\in[0,T]$ and $\x\in \mH$:
\begin{align} \label{2p311}
\mE\left(\sup_{t\in[0,T]}\|\Y_m(t)\|^{2} + \int_0^T  \|\Y_m(t)\|_{\frac{1}{2}}^2 \d t \right)^k 
\leq C(k,T)(1+\|\x\|^{2})^k,
\end{align}
where $C$ is independent of $m$. 
\end{Lem}
\begin{proof} By applying the finite dimensional It\^o formula to the process $\|\Y_m(t)\|^2,$ we obtain 
\begin{align}\label{2.37}
&\|\Y_m(t)\|^2+2\int_0^{t}\|\Y_m(s)\|^2_{\frac{1}{2}} \d s\nonumber\\&=\|\x\|^2+\tr(\A^{-\varepsilon})T+\int_0^t\la \Y_m(s), \A^{-\frac{\va}{2}}\d\W_m(s)\ra+2\int_0^t \iZm \la \G_m(s-,z),\Y_m(s-)\ra\wi\uppi(\d s,\d z)\nonumber\\&\quad +\int_0^t \iZm \| \G_m(s,z)\|^2\uppi(\d s,\d z)\nonumber\\& =\|\x\|^2+\tr(\A^{-\varepsilon})T+\int_0^t \iZm \| \G_m(s,z)\|^2\mu(\d z)\d s+\int_0^t\la \Y_m(s), \A^{-\frac{\va}{2}}\d\W_m(s)\ra\nonumber\\&\quad +\int_0^t \iZm \left[2\la \G_m(s-,z),\Y_m(s-)\ra+\| \G_m(s-,z)\|^2\right]\wi\uppi(\d s,\d z).
\end{align}	
Let us take supremum over $[0,T]$ in \eqref{2.37}, raise to the power $k\geq 2$ and then take expectation to get 
\begin{align} \label{2p322}
\mE\left(\sup_{t\in[0, T]}\|\Y_m(T)\|^2+2\int_0^{T}\|\Y_m(t)\|^2_{\frac{1}{2}} \d t\right)^k 
\leq C(k,T,\G)\left[\left(1+\|\x\|^2\right)^k+\sum_{i=1}^3\mE[ K_i^k]\right], 
\end{align}
where
\begin{align*}
K_1&=2\sup_{t\in[0, T]}\left|\int_0^t\la \Y_m(s), \A^{-\frac{\va}{2}}\d\W_m(s)\ra\right|,  \\
K_2&=\sup_{t\in[0, T]}\left|\int_0^t \iZm \left[2\la \G_m(s-,z),\Y_m(s-)\ra+\|\G_m(s-,z)\|^2\right]\wi\uppi(\d s,\d z)\right|. \\
\end{align*}
Applying Burkholder-Davis-Gundy inequality for $k$-th moment   (see, \cite{P}, page 37) and Lemma \ref{L3},  we obtain
\begin{align}
\mE[K_1^k]&\leq C(k) \mE\left(\int_0^T \|\Y_m(s)\|^2\tr(\A^{-\va}) \d s\right)^{k/2} \no\\&\leq C(k)\mE\left[\tr(\A^{-\va})T \sup_{t\in[0,T]}\|\Y_m(t)\|^2\right]^{k/2}\nonumber \\
&\leq   C(k,\G,T)(1+\|\x\|^k) .
\end{align}
By the compensated Poisson measure  $\wi \uppi(\d t,\d z) = \uppi(\d t,\d z)-\mu(\d z)\d t$ and Kunita's inequality (see \cite{A}),  we write integral $K_2$  as 
\begin{align}\label{kunita}
\mE[K_2^k]&\leq \sup_{t\in[0, T]}\left|\int_0^t \iZm \left[2\la \G_m(s-,z),\Y_m(s-)\ra+\|\G(s-,z)\|^2\right]\wi\uppi(\d s,\d z)\right|^k\nonumber\\&\leq C(k)\left\{\mE\left[\left(\int_0^T\iZm\left|2\la \G_m(s,z),\Y_m(s)\ra+\|\G(s,z)\|^2\right|^2\mu(\d z)\d s\right)^{k/2}\right]\right.\nonumber\\&\quad+ \left.\mE\left[\left(\int_0^T\iZm\left|2\la \G_m(s,z),\Y_m(s)\ra+\|\G(s,z)\|^2\right|^k\mu(\d z)\d s\right)\right]\right\}\nonumber\\&=:K_3+K_4. 
\end{align}
By Cauchy-Schwarz inequality, we estimate $K_3$ as 
\begin{align}
K_3&\leq C(k)\mE\left[\left(\int_0^T\iZm[\| \G_m(s,z)\|^2\|\Y_m(s)\|^2+\|\G(s,z)\|^4]\mu(\d z)\d s\right)^{k/2}\right]\no\\&\leq C(k)\mE\left[\sup_{s\in[0,T]}\|\Y_m(s)\|^k\right]\left(\int_0^T\iZ\|\G(s,z)\|^2\mu(\d z)\d s\right)^{k/2}\no\\&\quad +C(k)\left(\int_0^T\iZ\|\G(s,z)\|^4\mu(\d z)\d s\right)^{k/2}\no\\&\leq C(k,T,\G)(1+\|\x\|^k).
\end{align}
Similarly,  we have 
\begin{align}\label{2.41}
K_4\leq  C(k,T,\G)(1+\|\x\|^k).
\end{align}
Combining \eqref{2p322}-\eqref{2.41}, we finally obtain \eqref{2p311}.
\end{proof}
Next, we prove some regularity of solutions.  Consider the \emph{Ornstein-Uhlenbeck} stochastic differential equation 
\begin{align}\label{N1}\d\N_m(t)+\A\N_m(t)\d t=\A^{-\frac{\va}{2}}\d\W_m(t), \ \  \
\N_m(0)=0.
\end{align}
For any $0<\al<\f2$ and  $\va>1+\al,$   the problem \eqref{N1} has a unique solution  $\N\in \mathrm{C}([0,T]; \mD(\A^\al))$    represented by the variation of constant formula
$$\N_m(t)=\int_0^te^{-\A(t-s)}\A^{-\frac{\va}{2}}\d\W_m(s)$$   By factorization method (see \cite{DZ}, Section 5.3), one can show that
\begin{align} \label{N2}
\mE\left[\sup_{t\in[0,T]}\|\N_m(t)\|_\al^{2k}\right]\leq c(k,\al,T)\ \ \ \mbox{for any} \ \ k> 1. 
\end{align}
Indeed, taking the following integral identity into account, 
$$\int_r^t(t-s)^{\de-1}(s-r)^{-\de}\d s=\frac{\pi}{\sin\de\pi}, \ \ \ 0\leq r\leq t, \ \ \ \de\in(0,1),$$
one can re-write the integral $\N_m(t),$  by applying the stochastic Fubini's theorem, as 
$$\N_m(t)=\frac{\sin\de\pi}{\pi}\int_0^t (t-s)^{\de-1}e^{-\A(t-s)} \wi \N_m(s) \d s ,$$
where $$\wi \N_m(s)=\int_0^s(s-r)^{-\de} e^{-\A(s-r)}  \A^{-\frac{\va}{2}}\d\W_m(r), \ \mbox{ provided} \  \ \mE\left[\int_0^T\|\wi\N_m(s)\|^{2k}_\al\d s\right] <+\infty.$$  More precisely for any $p>0,$   we have (see Theorem 4.36, \cite{DZ})
\begin{align*}
\mE\left[ \|\wi\N_m(s)\|^p_\al \right]&\leq 
C(p)\left(\int_0^s\|(s-r)^{-\de} \A^\al e^{-\A(s-r)}  \A^{-\frac{\va}{2}}\|^2_{\mcL^2(\mH)} \d r\right)^{p/2} \\
&\leq C(p) \left(\int_0^s(s-r)^{-2\de} \|\A^{\frac{\al}{2}} e^{-\A(s-r)}\|^2_{\mcL(\mH)}  \|\A^{(\frac{\al}{2}-\frac{\va}{2})}\|^2_{\mcL^2(\mH)} \d r\right)^{p/2}. 
\end{align*}
By the smoothing properties  of the semigroup (see, \cite{Sh}), we have 
$$\|\A^{\frac{\al}{2}} e^{-\A t}\|_{\mcL(\mH)} \leq  t^{-\frac{\al}{2}} \ \ \mbox{and} \ \  \|\A^{(\frac{\al}{2}-\frac{\va}{2})}\|^2_{\mcL^2(\mH)} = \tr(\A^{\al-\va}) < +\infty \ \mbox{for any} \ \va>1+\al.$$
Moreover, if $\de < \f2-\frac{\al}{2},$ one can validate 
\begin{align*}
\mE \left[\sup_{t\in[0, T]}\|\wi\N_m(t)\|^p_\al  \right]
\leq C(p) \tr(\A^{\al-\va})^{p/2}\left(  \int_0^T\tau^{-2\de-\al}  \d\tau\right)^{p/2} <+\infty. 
\end{align*}
On the other hand,  we note that
\begin{align*}
\mE\left[\sup_{t\in[0,T]} \|\N_m(t)\|_\al ^{2k}\right]&\leq  \mE\left[\sup_{t\in[0,T]} \left(\frac{\sin\de\pi}{\pi} \int_0^t (t-s)^{\de-1}\|e^{-\A(t-s)}\|_{\mathscr{L}(\mH)}\|\wi\N_m(s)\|_\al \d s\right)^{2k}\right] \\ &\leq C\mE\left[\sup_{t\in[0,T]}\|\wi\N_m(t)\|_\al^{2k}\right]\left(\int_0^T \tau^{(\de-1)}\d \tau\right).
\end{align*}
Since for any  $\de > 0,$   we see that  $\int_0^T \tau^{(\de-1)} \d\tau<+\infty$ and hence the required estimate follows.     
Using this result, we prove the following regularity result. 
\begin{Rem}\label{rem3.3}
One can easily seen that the estimate \eqref{N2} holds true for $\alpha\in[0,\frac{1}{2}]$ as well. 
\end{Rem}
\begin{Lem}\label{lem2.4}
Suppose the assumption $(H_1)$ is satisfied.  Let  $\va>1+\al, \  \al\in(0,\f2).$  Then  there exist constants $C>0$  and $\ga_1>0$ such that for any $t\in[0,T]$ and $\x\in \mD(\A^\al)$, the following estimate holds:
\begin{align} \label{2p24}
\mE\left[e^{-\ga_1(\al)\mathcal{Y}_m(t)}\|\A^\al \Y_m(t)\|^{2} \right]\leq C(T)(1+\|\A^\al \x\|^{2}).
\end{align}
where $\mathcal{Y}_m(t)$ is the integral defined in Proposition \ref{p2.1}.   Moreover, for any $k\geq 1,$ we have
\begin{align} \label{2p241}
\mE\left[e^{-\ga_1(\al,k)\mathcal{Y}_m(t)}\|\A^\al \Y_m(t)\|^{2k} \right]\leq C(k,T)(1+\|\A^\al \x\|)^{2k}.
\end{align}
\end{Lem}
\begin{proof}
By the change of variable $\Z_m=\Y_m-\N_m,$   the process $\Z_m$ solves the  equation
\begin{equation}\label{T1}
\left\{
\begin{aligned}
\d\Z_m(t)&=-[\A\Z_m(t)+\B_m(\Y_m(t))]\d t+ \iZm\G_m(t,z) \wi\uppi(\d t,\d z), \   t\in (0,T),  \\
\Z_m(0)&=\mathrm{P}_m\x.
\end{aligned}
\right.
\end{equation}
For a given $m\in \mN$ and for all $\ l>0,$ define a stopping time $$\t:=\inf\left\{t\geq 0: \|\Z_m(t)\|_\al\geq  l\right\}\c T.$$  We can argue that  $\t\to T, \mP$-a.s., as $ l\to\infty.$ 
Applying  the  finite dimensional It\^o formula
to $[e^{-\wi C\mathcal{Y}_m(t)}\|\Z(t)\|_\al^2],$ where $\wi C >0$ is a constant to be chosen later, to get
\begin{align} \label{2p25}
&e^{-\wi C \mathcal{Y}_m(t\c\t)}\|\Z_m(t\c\t)\|_\al^2 +2\int_0^{t\c\t} e^{-\wi C \mathcal{Y}_m(s)} \|\Z_m(s)\|^2_{\al+\frac{1}{2}}\d s\no \\&= \|\x\|^2_\al 
-\wi C  \int_0^{t\c\t} e^{-\wi C \mathcal{Y}_m(s)} \|\Z_m(s)\|^2_\al \|\Y_m(s)\|_{\frac{1}{2}}^2 \d s\no \\
&\quad-2\int_0^{t\c\t} e^{-\wi C \mathcal{Y}_m(s)} \la\B_m(\Y_m(s)), \A^{2\al}\Z_m(s)\ra \d s + 
\int_0^{t\c\t} \iZm e^{-\wi C \mathcal{Y}_m(s)}\|\G_m(s,z)\|_\al^2\uppi(\d z,\d s)  \no\\
&\quad+2\int_0^{t\c\t}  \iZm e^{-\wi C \mathcal{Y}_m(s)}\la\A^\al \G_m(s,z),\A^\al \Z_m(s-)\ra\wi\uppi(\d s,\d z) 
\no\\&:= \|\x\|^2_\al +\sum_{i=1}^4 I_i({t\c\t}). 
\end{align}
For any $\al\in(0,1/2),$ we get from  the inequality \eqref{ne} with $r=\alpha$ that 
\begin{align}\label{nne}
|\la\B_m(\Y_m), \A^{2\al}\Z_m\ra| \leq C(\al) \|\Y_m\|_\al\|\Y_m\|_{\frac{1}{2}}\|\Z_m\|_{\al+\frac{1}{2}}.
\end{align}
By  the Cauchy-Schwartz inequality, we  obtain
\begin{align} \label{2p26}
|I_2({t\c\t})|
&\leq \int_0^{t\c\t} e^{-\wi C \mathcal{Y}_m(s)} \|\Z_m(s)\|^2_{\al+\frac{1}{2}} \d s + \frac{C(\al)}{4}  \int_0^{t\c\t} e^{-\wi C \mathcal{Y}_m(s)} \|\Y_m(s)\|^2_\al \|\Y_m(s)\|_{\frac{1}{2}}^2 \d s \no\\ 
&\leq  \int_0^{t\c\t} e^{-\wi C \mathcal{Y}_m(s)} \|\Z_m(s)\|^2_{\al+\frac{1}{2}} \d s + \frac{C(\al) }{2} \int_0^{t\c\t} e^{-\wi C \mathcal{Y}_m(s)} \|\Z_m(s)\|^2_\al \|\Y_m(s)\|_{\frac{1}{2}}^2 \d s \no\\ &\quad+\frac{C(\al)}{2}  \int_0^{t\c\t} e^{-\wi C \mathcal{Y}_m(s)} \|\N_m(s)\|^2_\al \|\Y_m(s)\|_{\frac{1}{2}}^2 \d s. 
\end{align}
Note that the first integral can be absorbed in the left side of \eqref{2p25} and when $\wi C = \frac{C(\al)}{2},$ the second integral cancels with the integral $I_2.$ Then, we have 
\begin{align} \label{2.50}
&\hspace*{-0.7cm}e^{-\wi C \mathcal{Y}_m(t\c\t)}\|\Z_m(t\c\t)\|_\al^2 +\int_0^{t\c\t} e^{-\wi C \mathcal{Y}_m(s)} \|\Z_m(s)\|^2_{\al+\frac{1}{2}}\d s\no \\&\leq  \|\x\|^2_\al 
+\frac{C(\al)}{2}  \int_0^{t\c\t} e^{-\wi C \mathcal{Y}_m(s)} \|\N_m(s)\|^2_\al \|\Y_m(s)\|_{\frac{1}{2}}^2 \d s\no\\&\quad +
\int_0^{t\c\t} \iZm e^{-\wi C \mathcal{Y}_m(s)}\|\G_m(s,z)\|_\al^2\uppi(\d z,\d s)  \no\\
&\quad+2\int_0^{t\c\t}  \iZm e^{-\wi C \mathcal{Y}_m(s)}\la\A^\al \G_m(s,z),\A^\al \Z_m(s-)\ra\wi\uppi(\d s,\d z).
\end{align}
For any $\va>1+\al,$  applying \eqref{N2} and Lemma \ref{L3},  we have the following estimate:
\begin{align}\label{2.51}
&\mE\left[\int_0^{t\c\t} e^{-\wi C \mathcal{Y}_m(s)} \|\N_m(s)\|^2_\al \|\Y_m(s)\|_\f2^2 \d s\right]\no\\
&\qquad \leq  \left[\mE\left(\sup_{s\in[0,t\c\t]}\|\N_m(s)\|^2_\al\right)^2\right]^{1/2}  \left[\mE\left(\int_0^{t\c\t}  \|\Y_m(s)\|_{\frac{1}{2}}^2 \d s\right)^2\right]^{1/2}\no\\&\qquad \leq C(T,\G)(1+\|\x\|_\al^2).
\end{align}
Using  $(H_1),$ we get 
\begin{align} \label{2p29}
\mE[I_3({t\c\t})]=  \mE  \left[\int_0^{t\c\t} e^{-\wi C \mathcal{Y}_m(s)}\left(\iZm \|\G_m(s,z)\|^2_\al\mu(\d z)\right)\d s \right]
\leq  C(\G,T).  
\end{align}
Note that the integral  $I_4$ is a local martingale.  Indeed, we have
$$ \mE  | I_4(t\c\t)|^2 \leq  4\mE \left[\int_0^{t\c\t}  \|\Z_m(s)\|_\al^{2 } \left(\iZm \| \G_m(s,z)\|^2_\al \mu(\d z)\right) \d s\right] <+\infty,$$
whence $\mE[I_4(t)]=0,$  for all $t\in[0,T].$
Consequently,  the identity \eqref{2p25} reduces to the following:
\begin{align} 
\mE\left[e^{-\wi C \mathcal{Y}_m(t\c\t)}\|\Z_m(t\c\t)\|_\al^2\right] +\mE\left[\int_0^{t\c\t} e^{-\wi C \mathcal{Y}_m(s)} \|\Z_m(s)\|^2_{\al+\frac{1}{2}}\d s \right]
\leq C(\G,T)(1+ \|\x\|^2_\al).
\end{align}
Recalling  that $\Y_m=\Z_m+\N_m,$ we obtain   $\|\Y_m(t)\|_\al \leq  \|\Z_m(t)\|_\al +\|\N_m(t)\|_\al.$  Since the constant $C>0$ is independent of $l>0,$  passing the limit $l\to\infty$ and applying  Fatou's lemma, one can arrive at \eqref{2p24}.   

Note that from (\ref{2.50}), we have 
\begin{align} 
&\mE\left[e^{-C(\alpha,k) \mathcal{Y}_m(t\c\t)}\|\Z_m(t\c\t)\|_\al^{2k}\right] \no \\&\leq  C(\alpha,k)\left[1+\|\x\|^{2k}_\al 
+\mE\left(
\int_0^{t\c\t} \iZm e^{-\wi C \mathcal{Y}_m(s)}\|\G_m(s,z)\|_\al^2\mu(\d z)\d s \right)^k \right]\no\\&\quad + \mE\left(\int_0^{t\c\t} e^{-\wi C \mathcal{Y}_m(s)} \|\N_m(s)\|^2_\al \|\Y_m(s)\|_{\frac{1}{2}}^2 \d s\right)^k \no\\
&\quad+\mE\left(\int_0^{t\c\t}  \iZm e^{-\wi C \mathcal{Y}_m(s)}\left[2\la\A^\al \G_m(s,z),\A^\al \Z_m(s-)\ra+\|\G_m(s-,z)\|_\al^2\right]\wi\uppi(\d s,\d z)\right)^k.
\end{align}
The estimates similar to (\ref{2.51}) and   (\ref{kunita}) respectively for the last two integrals completes the proof of \eqref{2p241}.
\end{proof} 

\begin{Pro}\label{lem3.2} Let the Assumption  \ref{ass2.1} hold. Then for any $x\in\mH$, $\ga>0$ and  $t\in [0,T],$ the following holds:
	\begin{align}\label{e00}
		\mE\left[\exp\left(\ga \|\Y_m(t)\|^2+\ga \int _0^t \|\Y_m(s)\|^2_1 \d s\right)\right] 
		\leq C(\ga,\Q,\G,T)  \exp(\ga \|x\|^2). 
	\end{align}
\end{Pro}
The proof follows easily for the case $x$ restricted to a ball of radius $R_{0}$ in $\mH$ by an application of the Ito formula to exponential functions. This exponential estimate also holds up to a stopping time without restricting the initial data  $x$ to a ball in $\mH$.

The following lemmas give  estimations related to  the systems \eqref{d1}  and \eqref{d2} whose proofs are similar to that of  Lemmas 3.1 and 3.2 of \cite{Da1}.  
\begin{Lem}
Let $\x,\H\in\mH$ and $t\in[0,T].$  Then for any $m\in\mN,$ we have the following estimates:
\begin{align} \label{d3}
e^{-4\mathcal{Y}_m(t)}\|\eta_m(t)\|^2+\int_0^t e^{-4\mathcal{Y}_m(s)}\|\eta_m(s)\|_{\frac{1}{2}}^2\d s &\leq \|\H\|^2, \\
e^{-8\mathcal{Y}_m(t)}\|\zeta_m(t)\|^2+\int_0^t e^{-8\mathcal{Y}_m(s)}\|\zeta_m(s)\|_{\frac{1}{2}}^2\d s &\leq 32 \|\H\|^4,\label{d4}
\end{align}
where the integral $\mathcal{Y}_m(t)=\ds\int_0^t \|\Y_m(s)\|^2_{\frac{1}{2}}\d s.$
\end{Lem} 
\begin{Lem}\label{lem2.6}
Let $\x,\H\in\mD(\A^\al), \al\in(0,1/2)$ and $t\in[0,T].$  Then for any  $m,k\in\mN,$ there exist  constants $\ga_2,c_3$ and $\ga_3,c_4$  such that the following  hold:
\begin{align}\label{d6}
\mE\left(e^{-\ga_2(\al)\mathcal{Y}_m(t)}\|\A^\al\eta_m(t)\|^2+\int_0^t e^{-\ga_2(\al)\mathcal{Y}_m(s)}\|\A^{\al+\frac{1}{2}}\eta_m(s)\|^2\d s\right)^k \no\\
\leq C(\al,k,T)(1+\|\A^\al\x\|^2)^k \|\A^\al\H\|^{2k}
\end{align}
and 
\begin{align} \label{d7}
\mE\left(e^{-\ga_3(\al)\mathcal{Y}_m(t)}\|\A^\al\zeta_m(t)\|^2+\int_0^t e^{-\ga_3(\al)\mathcal{Y}_m(s)}\|\A^{\al+\frac{1}{2}}\zeta_m(s)\|^2\d s\right)^k \no\\
\leq  C(\al,k,T)(1+\|\A^\al\x\|^2) ^k\|\A^\al\H\|^{4k}.
\end{align}
\end{Lem} 

\section{A Priori Estimates for Finite Dimensional  HJB Equation}
The following smoothness  result of the Feynman-Kac semigroup  is crucial in establishing smoothness of  solutions to the mild form of the HJB equation \eqref{2p23} and hence for the HJB equation  \eqref{2p20} as well.  The derivative of the semigroup $\mathrm{T}_t^m$ can be estimated  as follows. 
\begin{align}
|\la \mathrm{D}_\x (\mathrm{T}^m_tf)(\x),\H\ra| &\leq\mE\left[e^{-\ga\mathcal{Y}_m(t)}|f(\Y_m(t))|   \left|e^{-\ga\mathcal{Y}_m(t)}\frac{1}{t}\int_0^t\la\A^{\frac{\va}{2}}\eta_m(s),\d\W_m(s)\ra\right|\right]\no\\
&\quad+\mE\left[e^{-\ga\mathcal{Y}_m(t)}|f(\Y_m(t))|  \left|e^{-\ga\mathcal{Y}_m(t)} 4\ga\int_0^t\left(1-\frac{s}{t}\right)\la\A^{\f2}\eta_m(s),\A^{\f2}\Y_m(s)\ra \d s\right|\right]\no\\
&:=B_1+B_2.  \label{se1}
\end{align}
Let $0<\al_3<\f2$ and $k\geq 1.$  Let $\x,\mathbf{h}\in\mathcal{D}(\A^{\alpha_3})$. By the Cauchy-Schwarz inequality, we  estimate $B_1$ as 
\begin{align*}
B_1\leq  \|f\|_{\al_3,k,0}\Big[\mE e^{-2\ga\mathcal{Y}_m(t)}(1+\|\Y_m(t)\|_{\al_3})^{2k}   \Big]^{1/2}\Big[\mE\Big(e^{-\ga\mathcal{Y}_m(t)}\frac{1}{t}\int_0^t\la\A^{\frac{\va}{2}}\eta_m(s),\d\W_m(s)\ra \Big)^2\Big]^{1/2}.
\end{align*}
Note that for any $\uptheta$  large enough
\begin{align*}
\mE \left(e^{-2\ga\mathcal{Y}_m(t)}(1+\|\Y_m(t)\|_{\al_3})^{2k}\right)   \leq C(k,\G,T)(1+\|\x\|_{\al_3})^{2k}.
\end{align*}
Let $Z_m(t)=e^{-\ga\mathcal{Y}_m(t)}\int\limits_0^t\la\A^{\frac{\va}{2}}\eta_m(s),\d\W_m(s)\ra.$ By It\^o's formula, one can see that 
\begin{align*}
Z_m(t)^2&=2\int_0^tZ_m(s)\d Z_m(s)+[Z_m,Z_m](t)\\
&=-2\uptheta \int_0^t Z_m^2(s) \|\Y_m(s)\|^2_{\f2}\d s+2 \int_0^tZ_m(s)e^{-\ga\mathcal{Y}_m(s)}\la\A^{\frac{\va}{2}}\eta_m(s),\d\W_m(s)\ra\\
&\quad +\int_0^te^{-2\ga\mathcal{Y}_m(s)}\|\eta_m(s)\|_{\vf}^2 \d s.
\end{align*}
Since the first integral is non-positive and second integral is a martingale, we have
\begin{align*}
\mE\left[Z_m(t)^2\right]\leq \mE\left(\int_0^te^{-2\ga\mathcal{Y}_m(s)}\|\eta_m(s)\|_{\vf}^2 \d s\right).
\end{align*}
Observe that  for any $0<\wi\al<\al_3$ and  $\va<1+2\wi\al,$ one can get
\begin{align*}
\|\eta_m\|_{\vf}&= \left(\sum_{i=1}^\infty\lambda_i^\va|(\eta_m,e_i)|^2\right)^{1/2}
= \left(\sum_{i=1}^\infty\lambda_i^{\va-1-2\wi\al}\lam_i^{1+2\wi\al}|(\eta_m,e_i)|^2\right)^{1/2}\\
&\leq C\left(\sum_{i=1}^\infty\lam_i^{1+2\wi\al}|(\eta_m,e_i)|^2\right)^{1/2}=C\|\eta_m\|_{\f2+\wi\al}.
\end{align*}
Applying the interpolation inequality \eqref{II} with $s=\f2+\wi\al,s_1=\al_3$ and $s_2=\f2+\al_3$ so that $s_1\leq s\leq s_2$ and $\th=2(\al_3-\wi\al)$ to get
$$
\|\eta_m\|_{\f2+\wi\al} \leq \|\eta_m\|_{\al_3}^{2(\al_3-\wi\al)}\|\eta_m\|_{\f2+\al_3}^{1-2(\al_3-\wi\al)}.
$$
By applying Young's inequality and choosing $\uptheta$ sufficiently large so that Lemma \ref{lem2.6}  leads to
\begin{align*}
\mE\left[Z_m(t)^2\right]&\leq C\left[\mE\left(\int_0^te^{-\ga\mathcal{Y}_m(s)}\|\eta_m(s)\|_{\al_3}^{2}\right) \d s\right]^{{2(\al_3-\wi\al)}}\left[\mE\left(\int_0^te^{-\ga\mathcal{Y}_m(s)}\|\eta_m(s)\|_{\f2+\al_3}^{2} \d s\right)\right]^{{1-2(\al_3-\wi\al)}}\\
&\leq C(\al_3,\G,T)t^{2(\al_3-\wi\al)} (1+\|\x\|^2_{\al_3})\|\H\|^2_{\al_3}.
\end{align*}
It thus leads to
$$
B_1\leq  C(\al_3,k,\G,T) \ t^{-(1-(\al_3-\wi\al))} (1+\|\x\|_{\al_3})^{k+1}\|\H\|_{\al_3} \ \|f\|_{\al_3,k,0}.
$$
Similarly, we also obtain that
$$
B_2\leq C(k,\G,T,\uptheta)(1+\|\x\|_{\al_3})^{k}\|f\|_{\al_3,k,0}  \left[\mE\left( \int_0^te^{-2\ga\mathcal{Y}_m(s)}\left(1-\frac{s}{t}\right)|\la\A^{\f2}\eta_m(s),\A^{\f2}\Y_m(s)\ra| \d s\right)^2\right]^{1/2}.
$$
Noting $s/t\leq 1,$ we obtain from Lemma \ref{lem2.6} that
\begin{align*}
\mE\left( \int_0^te^{-2\ga\mathcal{Y}_m(s)}|\la\A^{\f2}\eta_m(s),\A^{\f2}\Y_m(s)\ra| \d s\right)^2 &\leq \|\H\|^2\mE\left(\int_0^te^{-2\ga\mathcal{Y}_m(s)}\|\Y_m(s)\|_\f2^2\d s\right)\\
&\leq C(\G,T)(1+\|\x\|^2)\|\H\|^2.
\end{align*}
Since $\|\x\|\leq C\|\x\|_{\al_3},  \al_3>0,$ one can arrive at
\[B_2\leq C(k,\G,T,\uptheta)(1+\|\x\|_{\al_3})^{k+1} \|\H\|_{\al_3}\|f\|_{\al_3,k,0}.
\]
In view of $B_1$ and $B_2,$ the estimate \eqref{se1}   reads as
$$
\|\mathrm{T}_t^mf\|_{\al_3,k+1,1}\leq C(\al_3, k, \G,T)t^{-(1-(\al_3-\wi\al))}\|f\|_{\al_3,k,0}.
$$
Besides, following the arguments in Proposition 3.3, \cite{Da1} for the estimate on the second derivative of the semigroup $\mathrm{T}_t^m$, we have the  following general regularity result:
\begin{Pro}     Let $\ga$ be the constant such that $\ga \geq \ga_4= \max\{8,\ga_1,\ga_2,\ga_3\}.$  Let $\beta\in [0,2],\gamma\in[0,1]$  such that $\b+\gamma\leq 2.$ Then for any $0<\wi\al<\al_3$ with $\va<1+2\wi\al$ and $k\in\mN,$  there exists a constant $C>0$ such that for any $f\in C_b(\mathrm{P}_m\mH), t\in[0,T]$ the following estimate holds when $\x$ restricted to a ball in $\mH$:
\begin{align} \label{S1}
\|\mathrm{T}_t^mf\|_{\al_3,k+2\b,\b+\gamma} \leq C(\al_3, k, \G,T) t^{-\b(1-(\al_3-\wi\al))}\|f\|_{\al_3,k,\gamma}.
\end{align}
\end{Pro}
Next we estimate the terms  in  \eqref{2p23}  one by one. 
\begin{Lem}
Let $\ga$ be the constant such that $\ga\geq \ga_4.$  For any $m\in \mN$   and $\beta\in [0,2],\gamma\in[0,1]$  such that $\b+\gamma\leq 2,$    there exists a constant $C>0$ satisfying
\begin{align} \label{S}
\|\mathrm{T}_t^m\wi f\|_{0,0,\b+\gamma}\leq C(\ga),  \  \text{ for all }\ t\in[0,T] \mbox{ and $\x$ restriced to a ball in } \mH.
\end{align}
\end{Lem}
\begin{proof}
Taking $\b=\gamma=0$ and recalling the definition of the semigroup $\mathrm{T}_t^m\wi f,$   we  have
\begin{align}  \label{S2}
\mathrm{ T}_t^m\wi f(\x)&= \mE\left[e^{-2\ga \mathcal{Y}_m(t) - \ga \|\Y_m(t)\|^2} \|\Y_m(t)\|^2 \right] \no\\
& \leq \mE\left[(\mathcal{Y}_m(t)+\|\Y_m(t)\|^2) e^{-\ga \{\mathcal{Y}_m(t) + \|\Y_m(t)\|^2\}}  \right] \leq C(\ga),
\end{align}
since $ye^{-\ga y}<+\infty$ for any large $\ga\geq \ga_4.$ Thus \eqref{S} follows in this case.     Next, we take $\b=1$ and $\gamma=0.$   
For any $\H\in \mathrm{P}_m\mH$, we have 
\begin{align} \label{S3}
&\la \mathrm{D}_\x (\mathrm{T}_t^m \wi f)(\x),\H\ra\no\\&=-\mE\left[ e^{-2\ga \mathcal{Y}_m(t) - \ga \|\Y_m(t)\|^2} \left\{ 4\ga \int_0^t\la \A^{1/2}\Y_m(s),\A^{1/2}\eta_m(s)\ra \d s +2\ga \la \Y_m(t),\eta_m(t)\ra \right\} \|\Y_m(t)\|^2\right]\no\\
&\quad +2 \mE \left[ e^{-2\ga \mathcal{Y}_m(t) - \ga \|\Y_m(t)\|^2} \la \Y_m(t),\eta_m(t)\ra\right] =: I_1+I_2+I_3. 
\end{align} Again for $\ga\geq \ga_4, $ we get from \eqref{d3} that
\begin{align*}
I_1&\leq  C(\ga) \mE\left[  (\sqrt{\mathcal{Y}_m(t)} e^{-\ga \mathcal{Y}_m(t)}) e^{-\ga \mathcal{Y}_m(t)}
\left(\int_0^t\|\eta_m(s)\|_{\frac{1}{2}}^2 \d s\right)^{1/2}\right] \\
&\leq C(\ga)\left[\mE\left( \int_0^t e^{-2\ga \mathcal{Y}_m(s)}  \|\eta_m(s)\|_{\frac{1}{2}}^2 \d s\right) \right]^{1/2} \leq C(\ga) \|\H\|.  
\end{align*}
Using \eqref{d3} again
\begin{align*}
I_2 &\leq C(\ga) \mE\left[ e^{-2\ga \mathcal{Y}_m(t) - \ga \|\Y_m(t)\|^2} \|\Y_m(t)\|^3\|\eta_m(t)\|\right] \leq C(\ga) \mE\left[ e^{-2\ga \mathcal{Y}_m(t)}\|\eta_m(t)\|\right] \\
&\leq C(\ga) \left[\mE\left(e^{-4\ga \mathcal{Y}_m(t)}\|\eta_m(t)\|^2\right)\right]^{1/2} \leq C(\ga) \|\H\|,
\end{align*}
since $y^\frac{3}{2} e^{-\ga y} \leq e^{\frac{3}{2} y}e^{-\ga y} <+\infty,$ for $\ga \geq \ga_4.$  Similarly, we can show that $I_3\leq C(\ga) \|\H\|.$ 
Hence, there exists a constant $  C(\ga)$ such that
\begin{align} \label{S4}
\|\mathrm{T}_t^m \wi f\|_{0,0,1} \leq  C(\ga). 
\end{align}
Next, we consider the extreme case $\b=\gamma=1.$ The differential of \eqref{S2} is given by
\begin{align} \label{sd1}
&\hspace*{-0.7cm}\mathrm{D}_\x^2\mathrm{T}_t^m \wi f(\x)\la\H,\H\ra\no\\
=&\mE\left[ e^{-2\ga \mathcal{Y}_m(t) - \ga \|\Y_m(t)\|^2} \left( 4\ga \int_0^t\la \A^{1/2}\Y_m(s),\A^{1/2}\eta_m(s)\ra \d s
+2\ga \la \Y_m(t),\eta_m(t)\ra \right)^2 \|\Y_m(t)\|^2\right]\no\\
&-4\ga\mE\left[ e^{-2\ga \mathcal{Y}_m(t) - \ga \|\Y_m(t)\|^2} \|\Y_m(t)\|^2\int_0^t\left( \la \A^{1/2}\Y_m(s),\A^{1/2}\zeta_m(s)\ra + \|\eta_m(s)\|^2_{\frac{1}{2}}\right)  \d s \right] \no \\
&-2\ga \mE\Big[ e^{-2\ga \mathcal{Y}_m(t) - \ga \|\Y_m(t)\|^2}  \|\Y_m(t)\|^2\big(\la \Y_m(t),\zeta_m(t)\ra +\|\eta_m(t)\|^2 \big) \Big]  \no\\
&-4\mE\left[ e^{-2\ga \mathcal{Y}_m(t) - \ga \|\Y_m(t)\|^2} \left\{ 4\ga \int_0^t\la \A^{1/2}\Y_m(s),\A^{1/2}\eta_m(s)\ra \d s +2\ga \la \Y_m(t),\eta_m(t)\ra\right\}\la \Y_m(t),\eta_m(t)\ra \right] \no\\
&+ 2 \mE \left[ e^{-2\ga \mathcal{Y}_m(t) - \ga \|\Y_m(t)\|^2} \left(\la \Y_m(t),\zeta_m(t)\ra+\|\eta_m(t)\|^2\right)\right]
=:\sum_{i=1}^5 J_i. 
\end{align}  
Applying \eqref{d3},  we obtain that 
\begin{align*}
J_1&\leq   C(\ga) \mE\left[ e^{-2\ga \mathcal{Y}_m(t) - \ga \|\Y_m(t)\|^2}  \left\{\|\Y_m\|^2  \mathcal{Y}_m(t) \int_0^t\|\eta_m(s)\|_{\frac{1}{2}}^2 \d s  + \|\Y_m(t)\|^4  \|\eta_m(t)\|^2\right\}\right] \\
&\leq  C(\ga)\mE\left[ \left(e^{-\ga \mathcal{Y}_m(t) }\mathcal{Y}_m(t)\right) \left(e^{-\ga \|\Y_m(t)\|^2} \|\Y_m(t)\|^2\right)  \int_0^t e^{-\ga \mathcal{Y}_m(s) }\|\eta_m(s)\|_{\frac{1}{2}}^2 \d s\right]  \\
&\quad + C(\ga)\mE\left[ \left(e^{-\ga \|\Y_m(t)\|^2} \|\Y_m(t)\|^4\right)  e^{-2\ga \mathcal{Y}_m(t) } \|\eta_m(t)\|^2\right] \\
&\leq C(\ga)\|\H\|^2, \ \mbox{provided}  \  \ga\geq \ga_4. 
\end{align*}
Using \eqref{d3} and \eqref{d4}, we get
\begin{align*}
|J_2|&\leq  C(\ga) \mE\left[ e^{-2\ga \mathcal{Y}_m(t)} \sqrt{\mcY_m(t)} \left(\int_0^t\|\zeta_m(s)\|_{\frac{1}{2}}^2 \d s\right)^{1/2} \right] + c(\ga) \mE\left[  \int_0^te^{-2\ga \mathcal{Y}_m(s)} \|\eta_m(s)\|^2_{\frac{1}{2}}\d s\right] \\
&\leq C(\ga) \left[\mE\left(\int_0^t e^{-2\ga \mathcal{Y}_m(s)}\|\zeta_m(s)\|_{\frac{1}{2}}^2\d s\right)\right]^{1/2}  +  C(\ga) \|\H\|^2  \leq    C(\ga) \|\H\|^2.
\end{align*}
Similarly, we can estimate the integrals $J_3,J_4$ and $J_5$ in terms of $\|\H\|^2$ as well. Therefore, there exists a constant $ C(\ga)>0$ such that 
$|\mathrm{D}_\x^2(\mathrm{T}_t^m \wi f)(\x)\la\H,\H\ra| \leq  C(\ga)\|\H\|^2, $ whence 
\begin{align} \label{S5}
\|\mathrm{T}_t^m \wi f\|_{0,0,2} \leq  C(\ga).
\end{align}
The rest of the cases follow from the interpolation estimate \eqref{I}.  Indeed, for $\beta+\gamma\in(0,2)$, then from \eqref{I} we have
$$ \|\mathrm{T}_t^m \wi f\|_{0,0,\beta+\gamma} \leq C \|\mathrm{T}_t^m \wi f\|^{(\beta+\gamma)/2}_{0,0,2} \|\mathrm{T}_t^m \wi f\|^{1-(\beta+\gamma)/2}_{0,0,0}.$$   
The required result follows from \eqref{S2} and \eqref{S4}. This completes the proof. 
\end{proof}

Next, we prove the  estimate for the second term in \eqref{2p23}.  
\begin{Lem}
Let $\al_1<\al<\wi \al_1$ and $k\in\mN.$  Let $\b>0$ be such that $(1+\b)(1-\upepsilon)<1,$ where $\upepsilon=(\al-\al_1).$  Then for any $t\in[0,T]$ and $\w_m\in C^{\al,k,1+\b},$ there exists a constant $C>0$ such that the  following holds when $\x$ restricted to a ball in $\mH$:
\begin{align} \label{C}
\|\mathrm{T}_t^m\widetilde{\mathrm{F}}_m\|_{\al,k+2+2(1+\b),1+\b+\gamma}  \leq C(\al, k, \b, \G,T) t^{-(1+\b)(1-\upepsilon)}(1+\|\w_m\|_{\al,k,1+\gamma}).
\end{align}
\end{Lem}
\begin{proof}
In view of the Proposition 3.1,   the Feynman-Kac semigroup $\mathrm{T}_t^m$ satisfies the estimate
\begin{align} \label{C1}
\|\mathrm{T}_t^m\widetilde{\mathrm{F}}_m\|_{\al,k+2+2(1+\b),1+\b+\gamma}  \leq C(\al, k, \b, \G,T) t^{-(1+\b)(1-\upepsilon)}(1+\|\widetilde{\mathrm{F}}_m\|_{\al,k+2,\gamma}).
\end{align}
The proof can be completed by estimating $\|\widetilde{\mathrm{F}}_m(\x,\w_m,\mathrm{D}_\x \w_m)\|_{\al,k+2,\gamma}.$ Let us first consider the case of $\gamma=0$. We have 
\begin{align}\label{3.11}
&\hspace*{-0.7cm}|\widetilde{\mathrm{F}}_m(\x,\w_m,\mathrm{D}_\x \w_m)|\nonumber\\&\leq 2\th|\la\A^{-\va}\x,\mathrm{D}_\x\w_m\ra|+4
\ga^2\|\A^{-\frac{\va}{2}}\x\|^2|\w_m|+2\ga\tr(\A^{-\va})|\w_m|\nonumber\\&\quad+e^{-\ga\|\x\|^2}|\mathrm{F}(e^{\ga\|\x\|^2}(\mathbf{K}^*\mathrm{D}_\x\w_m+2\ga\mathbf{K}^*\w_m\x))|\nonumber\\
&\quad +\iZm\big|(e^{\ga \|\x+\G_m\|^2-\ga\|\x\|^2}-1)\w_m(t,\x+\G_m(t,z))-2\th\la \G_m(t,z),\w_m\x\ra\big|\mu(\d z)\nonumber\\&=:\sum_{i=1}^5I_i.
\end{align}
We estimate $I_1$ as 
\begin{align}\label{i1}
I_1&= 2\uptheta| \la\A^{-\va}\A^{\alpha}\x,\A^{-\alpha}\mathrm{D}_\x\w_m\ra|\leq \|\A^{-\va}\A^{\alpha}\x\| \|\A^{-\alpha}\mathrm{D}_\x\w_m\|\nonumber\\&\leq 2\uptheta \|\A^{-\varepsilon}\|_{\mathcal{L}(\mH)}\|\x\|_{\alpha}\sup_{\|\H\|_{\alpha}=1}|\langle \A^{-\alpha}\mathrm{D}_{\x}\w_m,\H\rangle|\nonumber\\&\leq 2\uptheta(\tr(\A^{-2\va}))^{1/2}\frac{\|\x\|_{\alpha}}{(1+\|\x\|_{\alpha})^{k+2}}(1+\|\x\|_{\alpha})^{k+2}\sup_{\|\H\|_{\alpha}=1}|\langle \A^{-\alpha}\mathrm{D}_{\x}\w_m,\H\rangle|\nonumber\\&\leq C(\uptheta)(1+\|\x\|_{\alpha})^{k+2} \sup_{x\in\mathrm{D}(\A^{\alpha})}\frac{1}{(1+\|\x\|_{\alpha})^{k}}\sup_{\|\H\|_{\alpha}=1}|\langle \A^{-\alpha}\mathrm{D}_{\x}\w_m,\H\rangle|\nonumber\\&=C(\uptheta)(1+\|\x\|_{\alpha})^{k+2}\|\w_m\|_{\alpha,k,1}.
\end{align}
Similarly, we estimate $I_2$ as 
\begin{align}
I_2&=4
\ga^2\|\A^{-\frac{\va}{2}}\x\|^2|\w_m|\leq 4\uptheta^2 \|\A^{-\frac{\va}{2}}\|_{\mathcal{L}(\mH)}^2\|\x\|^2\frac{1}{(1+\|\x\|_{\alpha})^{k+2}}(1+\|\x\|_{\alpha})^{k+2}|\w_m|\nonumber\\&\leq C(\uptheta)(1+\|\x\|_{\alpha})^{k+2}\|\w_m\|_{\alpha,k,0}.
\end{align}
The term $I_3$ can be bounded by $C(\uptheta)(1+\|\x\|_{\alpha})^{k+2}\|\w_m\|_{\alpha,k,0}$. From the definition of $\mathrm{F}(\cdot)$ (see \eqref{2p16}) and Assumption \ref{ass2.1} $(H_2)$,  we obtain that 
\begin{align}
I_4&\leq C(R)\|\mathbf{K}^*\mathrm{D}_\x\w_m+2\ga \mathbf{K}^*\w_m\x \|\leq C(R,\uptheta)\left(\|\A^{-\widetilde{\alpha}_1}\mathrm{D}_{\x}\w_m\|+|\w_m|\|\A^{-\widetilde{\alpha}_1}\x\|\right)\nonumber\\&\leq C(R,\uptheta)\left(\|\A^{-\alpha}\mathrm{D}_{\x}\w_m\|+|\w_m|\|\x\|_{\alpha}\right) \leq C(R,\uptheta)(1+\|\x\|_{\alpha})^{k+2}\|\w_m\|_{\alpha,k,1}.
\end{align}
The term $I_5$ can be estimated as 
\begin{align*}
I_5&\leq \iZm\big|(e^{\ga \|\x+\G_m\|^2-\ga\|\x\|^2}-1)\w_m(t,\x+\G_m(t,z))\big|\mu(\d z)+2\th \iZm\big|\la \G_m(t,z),\w_m\x\ra\big|\mu(\d z)\nonumber\\&=: I_6+I_7.
\end{align*}
Using the Assumption \ref{ass2.1}$-(H_1)$, we note that 
\begin{align}\label{3.15}
I_6&\leq 2\iZm e^{\ga \|\x+\G_m\|^2-\uptheta\|\x\|^2}\big|\w_m(t,\x+\G_m(t,z))\big|\mu(\d z)\nonumber\\&\leq 2\iZm e^{\ga \|\G_m\|^2}e^{2\uptheta\|\G_m\|\|\x\|}(1+\|\x\|_{\alpha}+\|\G_m\|_{\alpha})^{k+2}\nonumber\\
&\quad\quad\times\frac{1}{(1+\|\x\|_{\alpha}+\|\G_m\|_{\alpha})^{k+2}}\big|\w_m(t,\x+\G_m(t,z))\big|\mu(\d z)\nonumber\\&\leq C(\uptheta)\iZm e^{2\ga \|\G_m\|^2}(1+\|\x\|_{\alpha}+\|\G_m\|_{\alpha})^{k+2}\frac{1}{(1+\|\x+\G_m\|_{\alpha})^{k}}\big|\w_m(t,\x+\G_m(t,z))\big|\mu(\d z)\nonumber\\&\leq C(\uptheta)\sup_{\mathbf{y}\in\mathrm{D}(\A^{\alpha})}\sup_{z\in\mathcal{Z}}\left(\frac{1}{(1+\|\mathbf{y}(z)\|_{\alpha})^{k}}\big|\w_m(t,\mathbf{y}(z))\big|\right)\nonumber\\
&\quad\quad\times\int_{\mathcal{Z}}e^{2\uptheta\|\G_m\|^2}(1+\|\x\|_{\alpha})^{k+2}(1+\|\G_m\|_{\alpha})^{k+2}\mu(\d z)\nonumber\\&\leq C(\uptheta)(1+\|\x\|_{\alpha})^{k+2}\|\w_m\|_{\alpha,k,0}\int_{\mathcal{Z}}e^{2\uptheta\|\G\|^2}(1+\|\G\|_{\alpha})^{k+2}\mu(\d z).
\end{align}
Now, we estimate $I_7$ as 
\begin{align}\label{i7}
I_7&\leq 2\uptheta\int_{\mathcal{Z}_m}\|\G_m(t,z)\|\|\x\||\w_m|\mu(\d z)\leq C(\uptheta) \|\w_m\|_{\alpha,k,0}(1+\|\x\|_{\alpha})^{k+2}\int_{\mathcal{Z}}\|\G(t,z)\|\mu(\d z).
\end{align}
Combining (\ref{i1})-(\ref{i7}), we obtain 
\begin{align}\label{417}
&\|\widetilde{\mathrm{F}}_m(\x,\w_m,\mathrm{D}_\x \w_m)\|_{\alpha,k+2,0}\nonumber\\&\leq C(R,\uptheta)\left(1+\int_{\mathcal{Z}}e^{2\uptheta\|\G\|^2}(1+\|\G\|_{\alpha})^{k+2}\mu(\d z)+\int_{\mathcal{Z}}\|\G(t,z)\|\mu(\d z)\right)\|\w_m\|_{\alpha,k,1}.
\end{align}
We now consider the case of $\gamma=1$. We know that 
\begin{align}\label{418}
\la \mathrm{D}_{\x}\widetilde{\mathrm{F}}_m,\H\ra &=2\uptheta \la\A^{-\va}\H,\mathrm{D}_\x\w_m\ra+2\uptheta \mathrm{D}_{\x}^2\w_m\la \A^{-\va}\x,\H\ra+8
\ga^2\la \A^{-\va}\x, \H\ra\w_m\nonumber\\&\quad+4
\ga^2\|\A^{-\frac{\va}{2}}\x\|^2\la \mathrm{D}_{\x}\w_m,\H\ra+ 2\ga\tr(\A^{-\va})\la \mathrm{D}_{\x}\w_m,\H\ra\nonumber\\&\quad +\la\mathrm{D}_{\x}e^{-\ga\|\x\|^2}\mathrm{F}(e^{\ga\|\x\|^2}(\mathbf{K}^*\mathrm{D}_\x\w_m+2\ga\mathbf{K}^*\w_m\x)),\H\ra\nonumber\\
&\quad +\iZm\left<\mathrm{D}_{\x}\left((e^{\ga \|\x+\G_m\|^2-\ga\|\x\|^2}-1)\w_m(t,\x+\G_m(t,z))-2\th\la \G_m(t,z),\w_m\x\ra\right),\H\right>\mu(\d z)\nonumber\\&=\sum_{i=8}^{14}I_i.
\end{align}
We estimate $I_8$ as 
\begin{align}\label{i8}
|I_8|&\leq 2\uptheta|\la \A^{-\va+\alpha}\H,\A^{-\alpha}\mathrm{D}_\x\w_m\ra|\leq 2\uptheta\|\A^{-\va+\alpha}\H\|\|\A^{-\alpha}\mathrm{D}_{\x}\w_m\|\nonumber\\&\leq 2\uptheta\|\H\|_{\alpha}\|\A^{-\va}\|_{\mathcal{L}(\mH)}\sup_{\|\H\|_{\alpha}=1}\left|\la\A^{-\alpha}\mathrm{D}_{\x}\w_m,\H\ra\right|\leq C(\uptheta)\|\H\|_{\alpha}(1+\|\x\|_{\alpha})^{k+2}\|\w_m\|_{\alpha,k,1}.
\end{align}
Similarly, we have 
\begin{align}
|I_9|&\leq 2\uptheta |\mathrm{D}_{\x}^2\w_m\la \A^{-\va}\x,\H\ra|\leq 2\uptheta\|\A^{-\alpha}\mathrm{D}_{\x}^2\w_m\A^{-\va}\x\|\|\H\|_{\alpha}\nonumber\\&\leq 2\uptheta\|\A^{-\alpha}\mathrm{D}_{\x}^2\w_m\|_{\mathcal{L}(\mH^2;\mathbb{R})}\|\A^{-\va}\x\|\|\H\|_{\alpha}\nonumber\\&\leq 2\uptheta\sup_{\|\H\|_{\alpha}=1}|\la\A^{-\alpha}\mathrm{D}_{\x}^2\w_m\H,\H\ra|\|\A^{-\va}\|_{\mathcal{L}(\mH)}\|\x\|_{\alpha}\|\H\|_{\alpha}\nonumber\\&\leq C(\uptheta)\|\H\|_{\alpha}(1+\|\x\|_{\alpha})^{k+2}\|\w_m\|_{\alpha,k,2}.
\end{align}
Using Cauchy-Schwarz inequality, we estimate $I_{10}$ as 
\begin{align}
|I_{10}|\leq 8\uptheta^2 \|\A^{-\va}\|_{\mathcal{L}(\mH)}\|\x\|_{\alpha}\|\H\|_{\alpha}|\w_m|\leq C(\uptheta)\|\H\|_{\alpha}\|\w_m\|_{\alpha,k,0}.
\end{align}
One can estimate $I_{11}$ as 
\begin{align}
|I_{11}|&\leq 4
\ga^2\|\A^{-\frac{\va}{2}}\x\|^2\|\A^{-\alpha}\mathrm{D}_{\x}\w_m\|\|\H\|_{\alpha}\leq 4
\ga^2 \|\A^{-\frac{\va}{2}}\|_{\mathcal{L}(\mH)}^2\|\x\|^2_{\alpha}\|\H\|_{\alpha}\sup_{\|\H\|_{\alpha}=1}\left|\la\A^{-\alpha}\mathrm{D}_{\x}\w_m,\H\ra\right|\nonumber\\&\leq C(\uptheta)\|\H\|_{\alpha}(1+\|\x\|_{\alpha})^{k+2}\|\w_m\|_{\alpha,k,1}.
\end{align}
Similarly $I_{12}$ can be estimated as 
\begin{align}
|I_{12}|\leq 2\ga\tr(\A^{-\va})\|\A^{-\alpha} \mathrm{D}_{\x}\w_m\|\| \H\|_{\alpha}\leq C(\uptheta)\|\H\|_{\alpha}(1+\|\x\|_{\alpha})^{k+2}\|\w_m\|_{\alpha,k,1}.
\end{align}
One can write $I_{13}$ as 
\begin{align}
I_{13}&=e^{-\uptheta\|\x\|^2}\la\mathrm{D}_{\p}\mathrm{F}(\p),2\uptheta \la\x,\H\ra\p+\mathbf{K}^*\mathrm{D}_{\x}^2\w_m\H+2\uptheta\mathbf{K}^*\la\mathrm{D}_{\x}\w_m,\H\ra\x+2\uptheta\mathbf{K}^*\w_m\H\ra\nonumber\\&\quad-2\uptheta e^{-\uptheta\|\x\|^2}\la\x,\H\ra\mathrm{F}(\p),
\end{align}
where $\p=e^{\ga\|\x\|^2}(\mathbf{K}^*\mathrm{D}_\x\w_m+2\ga\mathbf{K}^*\w_m\x)$, and $\mathrm{F}(\p)$ and $\mathrm{D}_{\p}(\mathrm{F}(\p))$ are  defined in (\ref{2p15}) and (\ref{2p17}) respectively. Now we estimates the terms of $I_{13}$ separately. We estimate the term  $| e^{-\uptheta\|\x\|^2}\la\mathrm{D}_{\p}\mathrm{F}(\p),2\uptheta \la\x,\H\ra\p\ra|$ as 
\begin{align}\label{324}
|e^{-\uptheta\|\x\|^2}\la\mathrm{D}_{\p}\mathrm{F}(\p),2\uptheta \la\x,\H\ra\p|&\leq 2\uptheta e^{-\uptheta\|\x\|^2}\|\mathrm{D}_{\p}\mathrm{F}(\p)\||\la\x,\H\ra|\|\p\|\nonumber\\&\leq C(R,\uptheta)\|\x\|\|\H\|\|\mathbf{K}^*\mathrm{D}_\x\w_m+2\ga\mathbf{K}^*\w_m\x\|\nonumber\\&\leq C(R,\uptheta)\|\x\|_{\alpha}\|\H\|_{\alpha}\left(\|\A^{-\widetilde{\alpha}_1}\mathrm{D}_{\x}\w_m\|+|\w_m|\|\A^{-\widetilde{\alpha}_1}\x\|\right)\nonumber\\&\leq C(R,\uptheta)(1+\|\x\|_{\alpha})^{k+2}\|\w_m\|_{\alpha,k,1}\|\H\|_{\alpha}.
\end{align}
Similarly using Assumption \ref{ass2.1} $(H_2)$, we estimate $|e^{-\uptheta\|\x\|^2}\la\mathrm{D}_{\p}\mathrm{F}(\p),\mathbf{K}^*\mathrm{D}_{\x}^2\w_m\H\ra|$ as 
\begin{align}
|e^{-\uptheta\|\x\|^2}\la\mathrm{D}_{\p}\mathrm{F}(\p),\mathbf{K}^*\mathrm{D}_{\x}^2\w_m\H\ra|&\leq e^{-\uptheta\|\x\|^2}\|\mathrm{D}_{\p}\mathrm{F}(\p)\|\|\mathbf{K}^*\mathrm{D}_{\x}^2\w_m\H\|\nonumber\\&\leq C(R) \|\A^{-\widetilde{\alpha}_1}\mathrm{D}_{\x}^2\w_m\H\|\leq C(R) \|\A^{-\alpha}\mathrm{D}_{\x}^2\w_m\|_{\mathcal{L}(\mH^2;\mR)}\|\H\|\nonumber\\&\leq C(R) \sup_{\|\H\|_{\alpha}=1}|\la\A^{-\alpha}\mathrm{D}_{\x}^2\w_m\H,\H\ra|\|\H\|_{\alpha}\nonumber\\&\leq C(R) (1+\|\x\|_{\alpha})^{k+2}\|\w_m\|_{\alpha,k,2}\|\H\|_{\alpha}.
\end{align}
Finally, we estimate $|2\uptheta e^{-\uptheta\|\x\|^2}\la\mathrm{D}_{\p}\mathrm{F}(\p),\mathbf{K}^*\la\mathrm{D}_{\x}\w_m,\H\ra\x+\mathbf{K}^*\w_m\H\ra|$ as 
\begin{align}\label{326}
&|2\uptheta e^{-\uptheta\|\x\|^2}\la\mathrm{D}_{\p}\mathrm{F}(\p),\mathbf{K}^*\la\mathrm{D}_{\x}\w_m,\H\ra\x+\mathbf{K}^*\w_m\H\ra|\nonumber\\&\leq 2\uptheta\|\mathrm{D}_{\p}\mathrm{F}(\p)\|\|\mathbf{K}^*\la\mathrm{D}_{\x}\w_m,\H\ra\x+\mathbf{K}^*\w_m\H\|\nonumber\\&\leq 2\uptheta\left(\|\mathbf{K}^*\x\||\la\mathrm{D}_{\x}\w_m,\H\ra|+|\w_m|\|\mathbf{K}^*\H\|\right)\nonumber\\&\leq 2\uptheta \left(\|\A^{-\widetilde{\alpha}_1}\x\|\|\A^{-\alpha}\mathrm{D}_{\x}\w_m\|\|\H\|_{\alpha}+|\w_m|\|\A^{-\widetilde{\alpha}_1}\H\|\right)\nonumber\\&\leq 2\uptheta\left(\|\x\|_{\alpha}\sup_{\|\H\|_{\alpha}=1}|\la\A^{-\alpha}\mathrm{D}_{\x}\w_m,\H\ra|+|\w_m|\right)\|\H\|_{\alpha}\nonumber\\&\leq 2\uptheta(1+\|\x\|_{\alpha})^{k+2}\|\w_m\|_{\alpha,k,1}\|\H\|_{\alpha}.
\end{align}
Combining (\ref{324})-(\ref{326}), we obtain 
\begin{align}
|I_{13}|\leq C(\uptheta)(1+\|\x\|_{\alpha})^{k+2}\|\w_m\|_{\alpha,k,2}\|\H\|_{\alpha}.
\end{align}
We can write $I_{14}$ as 
\begin{align}\label{i14}
I_{14}&=\iZm\left(e^{\ga \|\G_m\|^2+2\uptheta\la \G_m,\x\ra}-1\right)\la \mathrm{D}_{\x}\w_m(t,\x+\G_m),\H\ra\mu(\d z)\nonumber\\&\quad+\iZm e^{\ga \|\G_m\|^2+2\uptheta\la \G_m,\x\ra}2\uptheta\la\G_m,\H\ra\w_m(t,\x+\G_m)\mu(\d z)\nonumber\\&\quad-2\uptheta\iZm\left[\la\G_m,\x\ra\la\mathrm{D}_{\x}\w_m,\H\ra+\la\G_m,\w_m\H\ra\right]\mu(\d z)\nonumber\\&=:I_{15}+I_{16}+I_{17}.
\end{align}
A calculation similar to (\ref{3.15}) yields 
\begin{align}\label{i15}
|I_{15}|&\leq 2\iZm e^{\ga \|\G_m\|^2+2\uptheta|\la \G_m,\x\ra|}|\la \mathrm{D}_{\x}\w_m(t,\x+\G_m),\H\ra|\mu(\d z)\nonumber\\&\leq 2\iZm e^{\ga \|\G_m\|^2+2\uptheta\| \G_m\| \|\x\|}\|\A^{-\alpha} \mathrm{D}_{\x}\w_m(t,\x+\G_m)\|\mu(\d z)\|\H\|_{\alpha}\nonumber\\&\leq C(R,\uptheta)\iZm e^{2\ga \|\G_m\|^2}\frac{1}{(1+\|\y(z)\|_{\alpha})^k}\nonumber\\
&\quad\quad\times\sup_{\|\H\|_{\alpha}=1}|\la\A^{-\alpha}\mathrm{D}_{\x}\w_m(\y(z)),\H\ra|(1+\|\x\|_{\alpha}+\|\G\|_{\alpha})^{k+2}\mu(\d z)\|\H\|_{\alpha}.
\nonumber\\&\leq C(R,\uptheta)\|\w_m\|_{\alpha,k,1}(1+\|\x\|_{\alpha})^{k+2}\|\H\|_{\alpha}\int_{\mathcal{Z}}e^{2\uptheta\|\G\|^2}(1+\|\G\|_{\alpha})^{k+2}\mu(\d z).
\end{align}
The term $|I_{16}|$ can be estimated similarly as 
\begin{align}\label{i16}
|I_{16}|&\leq C(R,\uptheta)\|\w_m\|_{\alpha,k,0}(1+\|\x\|_{\alpha})^{k+2}\|\H\|_{\alpha}\int_{\mathcal{Z}}e^{2\uptheta\|\G\|^2}(1+\|\G\|_{\alpha})^{k+3}\mu(\d z).
\end{align}
We estimate $|I_{17}|$ as 
\begin{align}\label{i17}
|I_{17}|&\leq 2\uptheta\iZm\left(|\la\G_m,\x\ra| |\la\mathrm{D}_{\x}\w_m,\H\ra|+\|\G_m\||\w_m|\|\H\|\right)\mu(\d z)\nonumber\\&\leq 2\uptheta\iZm\left(\|\G_m\|\|\x\| \|\A^{-\alpha}\mathrm{D}_{\x}\w_m\|+\|\G_m\||\w_m|\right)\mu(\d z)\|\H\|_{\alpha}\nonumber\\&\leq C(\uptheta)(1+\|\x\|_{\alpha})^{k+2}\|\w\|_{\alpha,k,1}\|\H\|_{\alpha} \int_{\mathcal{Z}}\|\G\|\mu(\d z).
\end{align}
Combining \eqref{i15}-\eqref{i17} and substituting it in \eqref{i14}, we get 
\begin{align}\label{i14a}
|I_{14}|\leq C(R,\uptheta)\|\w_m\|_{\alpha,k,1}(1+\|\x\|_{\alpha})^{k+2}\|\H\|_{\alpha}\int_{\mathcal{Z}}\left(e^{2\uptheta\|\G\|^2}(1+\|\G\|_{\alpha})^{k+3}+\|\G\|\right)\mu(\d z).
\end{align}
Using the estimates of $I_8$-$I_{14}$ given in \eqref{i8}-\eqref{i14a} in \eqref{418}, we further have 
\begin{align}\label{434}
&\|\widetilde{\mathrm{F}}_m(\x,\w_m,\mathrm{D}_\x \w_m)\|_{\alpha,k+2,1}\nonumber\\&\leq C(R,\uptheta)\left(1+\int_{\mathcal{Z}}e^{2\uptheta\|\G\|^2}(1+\|\G\|_{\alpha})^{k+3}\mu(\d z)+\int_{\mathcal{Z}}\|\G(t,z)\|\mu(\d z)\right)\|\w_m\|_{\alpha,k,2}.
\end{align}
An application of interpolation inequality given in \eqref{I} for \eqref{417} and  \eqref{434} yields the bound for \eqref{C1} and hence the required result \eqref{C}. 
\end{proof}

Finally, we estimate the  last term in \eqref{2p23} as follows. 
\begin{Lem}
Let $\ga$ be the constant such that $\ga\geq \ga_4.$  For any $m\in \mN$   and $\beta\in [0,2],\gamma\in[0,1]$  such that $\b+\gamma\leq 2,$    there exists a constant $c_{12}>0$ satisfying
\begin{align} \label{SS}
\left\|\int_0^t \mathrm{T}_{t-s}^m\wi g(\x) \d s\right\|_{0,0,\b+\gamma}\leq C(\ga),  \  \text{ for all }\ t\in[0,T], \mbox{ and $\x$ restricted to a ball in } \mH.
\end{align}
\end{Lem}
\begin{proof}
Since $\| \text{curl} \ \u\|=\|\A^\f2\u\|=\|\u\|_\f2$ for any $\u\in \mV,$   we obtain for   $\b=\gamma=0$   and  $\ga\geq \ga_4$   as follows
\begin{align} \label{G}
\left|\int_0^t \mathrm{T}_{t-s}^m\wi g(\x)\d s\right| &=  \mE\left[\int_0^t e^{-2\ga \mathcal{Y}_m(t-s) - \ga \|\Y_m(t-s)\|^2} \| \text{curl } \Y_m(t-s)\|^2 \d s \right] \no\\ &\leq   \mE\left[\int_0^t  e^{-2\ga \mathcal{Y}_m(r)}\|\Y_m(r)\|_{\frac{1}{2}}^2 \d r \right]   \no \\
&=\mE\left[-\frac{1}{2\uptheta}\int_0^t\d(e^{-2\ga \mathcal{Y}_m(r)})\right]=\frac{1}{2\uptheta}\mE\big(1-e^{-2\ga \mathcal{Y}_m(t)}\big)\leq C(\uptheta).
\end{align}
Next, we consider $\b=1$ and $\gamma=0.$ The first differential for  any $\H\in \mathrm{P}_m\mH$ is
\begin{align} \label{G1}
&\left< \mathrm{D}_\x \int_0^t (\mathrm{T}_{t-s}^m \wi g)(\x)\d s,\H\right>\no\\&=-\mE\left[\int_0^t\left( e^{-2\ga \mathcal{Y}_m(s) - \ga \|\Y_m(s)\|^2} \left\{ 4\ga \int_0^{s}\la \A^{1/2}\Y_m(r),\A^{1/2}\eta_m(r)\ra \d r
+2\ga \la \Y_m,\eta_m\ra \right\} \|\Y_m(s)\|^2_\f2\right)\d s\right]\no\\&\quad +2 \mE \left[\int_0^te^{-2\ga \mathcal{Y}_m(s) - \ga \|\Y_m(s)\|^2} \la \A^\f2\Y_m(s),\A^\f2\eta_m(s)\ra \d s\right] =: I_1+I_2+I_3. 
\end{align} 
Invoking \eqref{d3} for any $\uptheta$ sufficiently large and arguments similar to \eqref{G},  we get 
\begin{align*}
I_1&\leq  4\th \mE\left[ \int_0^t  \left(e^{-\frac{\ga}{2} \mathcal{Y}_m(s)}\|\Y_m(s)\|^2_\f2\right)\left(\sqrt{\mathcal{Y}_m(s)} e^{-\frac{\ga}{2} \mathcal{Y}_m(s)}\right)
\left(\int_0^se^{-\ga \mathcal{Y}_m(r)}\|\eta_m(r)\|_{\frac{1}{2}}^2 \d r\right)^{1/2}\d s\right] \\
&\leq C(\th) \|\H\| \mE \left[\int_0^t  e^{-\frac{\ga}{2} \mathcal{Y}_m(s)}\|\Y_m(s)\|^2_\f2 \d s\right] \leq C(\ga) \|\H\|  
\end{align*}
and
\begin{align*}
I_2&\leq 2\th\mE\left(\sup_{s\in[0,t]}\left[e^{-\ga \mathcal{Y}_m(s)}\|\eta_m(s)\|\right]\int_0^t  e^{-\th\mathcal{Y}_m(s)-\th\|\Y_m(s)\|^2}\|\Y_m(s)\|\|\Y_m(s)\|^2_\f2 \d s\right)\\
&\leq C(\th)\|\H\|\left[\mE\left(\sup_{s\in[0,t]}\left[e^{-2\ga \|\Y_m(s)\|^2}\|\Y_m(s)\|^2\right]\right)\right]^{1/2}\left[\mE\left(\int_0^t  e^{-\th\mathcal{Y}_m(s)}\|\Y_m(s)\|^2_\f2 \d s\right)^2\right]^{1/2}\\
&\leq C(\ga) \|\H\|  .
\end{align*}
Similar to \eqref{G}, we obtain
\begin{align*}
I_3&\leq 2 \mE\left[ \left(\int_0^t  e^{-2\th \mathcal{Y}_m(s)}\|\Y_m(s)\|^2_\f2\d s\right)^{1/2}
\left(\int_0^te^{-2\ga \mathcal{Y}_m(s)}\|\eta_m(s)\|_{\frac{1}{2}}^2 \d s\right)^{1/2}\right] 
\leq C(\th)\|\H\|.
\end{align*}
Estimations of $I_1,I_2$ and $I_3$ lead to 
$\left\|\int\limits_0^t(\mathrm{T}_{t-s}^m\wi g)(\x)\d s\right\|_{0,0,1}\leq C(\th).$
Writing the second differential  of \eqref{G1}, as we did in \eqref{sd1}, and by similar inequalities as above, one can prove the theorem for $\b=\gamma=1.$  Other cases can be completed by applying the interpolation inequalities \eqref{I}. This completes the proof. 
\end{proof}

\begin{Pro}\label{prop3.2}
Let $\al_1<\al<\wi \al_1$, there exists $d(\alpha_1,\alpha)$ and a constant $C(\alpha_1,\alpha)$ such that for any $m\in\mathbb{N}$ and $\x$ restricted to a ball in $\mH$ we have:
\begin{align} \l{3p35}
\sup_{t\in[0, T]}\|\w_m(t,\cdot)\|_{\alpha,d,2}\leq C.
\end{align}
\end{Pro}
\begin{proof}
Since the value function $\v_m(\cdot,\cdot)$ is the solution of the approximated control problem \eqref{2F2}, by the energy estimate (see Lemma \ref{L2}), we have 
\begin{align*}
0\leq \v_m(T,\x)=\inf_{\U_m\in {\mathcal U_{R}^{0,T}} \cap \mathrm{L}^2(\Omega\times [0,T];\mathrm{P}_m\mH)}\mathrm{J}_m(\U_m)\leq \mathrm{J}_m(0)\leq C(1+\|\x\|^2).
\end{align*}
The above estimate is also true for any $t\in[0,T]$, that is,
$
\v_m(t,\x)\leq C(1+\|\x\|^2).
$
Since $\v_m(t,\x)=e^{\uptheta\|\x\|^2}\w_m(t,\x)$, we have 
\begin{align*}
|\w_m(t,\x)|\leq C e^{-\uptheta\|\x\|^2}(1+\|\x\|^2)\leq C(\uptheta),
\end{align*}
for any $t\in[0,T]$ and $\x\in\mathrm{P}_m\mH$. It leads to 
$
\|\w_m\|_{\alpha,0,0}\leq C(\uptheta).
$  Now we estimate the equation \eqref{2p23} as follows: 
\begin{align}\label{336}
&\|\w_m(t,\cdot)\|_{\al,k+2+2(1+\b),1+\b}\nonumber\\&\leq \|\mathrm{T}^m_t\wi f\|_{\al,k+2+2(1+\b),1+\b}+\int_0^t\|\mathrm{T}^m_{t-s}\widetilde{\mathrm{F}}_m\|_{\al,k+2+2(1+\b),1+\b}\d s+\int_0^t\|\mathrm{T}^m_{t-s}\wi g\|_{\al,k+2+2(1+\b),1+\b}\d s\nonumber\\&\leq \|\mathrm{T}^m_t\wi f\|_{0,0,1+\beta}+C\int_0^t(t-s)^{-(1+\b)(1-\upepsilon)}(1+\|\w_m(t-s,\cdot)\|_{\al,k,1})\d s+\int_0^t\|\mathrm{T}^m_{t-s}\wi g\|_{0,0,1+\b}\d s\nonumber\\&\leq C(\uptheta)\left[1+T+\frac{T^{1-(1+\b)(1-\upepsilon)}}{1-(1+\b)(1-\upepsilon)}\left(1+\sup_{t\in[0, T]}\|\w_m(t,\cdot)\|_{\al,k,1}\right)\right]\nonumber\\&\leq C(\uptheta,T,\beta,\upepsilon)\left(1+\sup_{t\in[0, T]}\|\w_m(t,\cdot)\|_{\al,k,1}\right).
\end{align}
Let us choose $k=\frac{1}{\beta}(4+2\beta)$ and apply the interpolation inequality given in Lemma \ref{lem2.1} to obtain 
\begin{align}\label{337}
\|\w_m(t,\cdot)\|_{\alpha,\frac{1}{\beta}(4+2\beta),1}\leq C\|\w_m(t,\cdot)\|_{\alpha,0,0}^{\frac{1}{1+\beta}}\|\w_m(t,\cdot)\|_{\alpha,\frac{1}{\beta}(4+2\beta)+4+2\beta,1+\beta}^{\frac{\beta}{1+\beta}}.
\end{align} 
Substituting (\ref{337}) in (\ref{336}) and using the fact that $\|\w_m(t,\cdot)\|_{\alpha,0,0}$ is uniformly bounded, we arrive at 
\begin{align}
\sup_{t\in[0, T]}\|\w_m(t,\cdot)\|_{\alpha,\frac{1}{\beta}(4+2\beta)+4+2\beta,1+\beta}\leq C,
\end{align}
which completes the proof. 
\end{proof}

In order to prove the existence of mild solution for \eqref{2p22}, we need the following convergence results for the solution of the SNSE \eqref{2p19}. 	For any $m\in\mathbb{N}$, let us set 
\begin{align*}
\mathrm{K}_m(t)=\int_0^te^{-(t-s)\A}\A^{-\frac{\var}{2}}\d\W_m(t)+\int_0^t \int_{\mZ_m}e^{-(t-s)\A}\G_m(t,z)\widetilde{\uppi}(\d s,\d z),
\end{align*}
which is the unique solution of 
\begin{equation}
\left\{
\begin{aligned}
\d \mathrm{K}_m(t)&=-\A \mathrm{K}_m(t)+\A^{-\frac{\var}{2}}\d\W_m(t)+\int_{\mZ_m}\G_m(s,z)\widetilde{\uppi}(\d s,\d z), \ t\in(0,T),\\
\mathrm{K}_m(0)&=0,
\end{aligned}
\right.
\end{equation}
where $\W_m=\mathrm{P}_m\W$,  $\G_m=\mathrm{P}_m\G$ and $\mZ_m=\mathrm{P}_m\mZ$. By the convolution estimate for the Gaussian noise  \eqref{N2}, Remark \ref{rem3.3} and a convolution estimate similar to the jump noise integral proved in Lemma 2.15, \cite{MSS}, we have 
\begin{align}\label{4.40}
\sup_{t\in[0,T]}\|\mathrm{K}_m(t)\|_{\frac{1}{2}}<+\infty, \ \mathbb{P}\text{-a.s.}
\end{align} It shows that $\mathrm{K}_m$ converges almost surely in $\mathrm{L}^{\infty}(0,T;\mV)$ to the unique solution 
\begin{align*}
\mathrm{K}(t)=\int_0^te^{-(t-s)\A}\A^{-\frac{\var}{2}}\d\W(t)+\int_0^t \int_{\mZ}e^{-(t-s)\A}\G(t,z)\widetilde{\uppi}(\d s,\d z),
\end{align*}
of the equation 
\begin{equation}
\left\{
\begin{aligned}
\d \mathrm{K}(t)&=-\A \mathrm{K}(t)+\A^{-\frac{\var}{2}}\d\W(t)+\int_{\mZ}\G(s,z)\widetilde{\uppi}(\d s,\d z), \ t\in(0,T),\\
\mathrm{K}(0)&=0.
\end{aligned}
\right.
\end{equation}

\begin{Pro}\label{prop3.3}
Suppose that the Assumption \ref{ass2.1} is satisfied. Let $\{x_m\}_{m=1}^{\infty}$ be such that $x_m\to x$ strongly in $\mH.$ Then $\Y_m(\cdot),$ the solution of  \eqref{2F3} converges to the unique solution  $\Y(\cdot)$ of \eqref{2p19}  in $\mathrm{L}^2(\Om;\mathrm{L}^{\infty}([0,T];\mH))\cap   \mathrm{L}^2(\Om;\mathrm{L}^2(0,T;\mV))$ having c\`adl\`ag paths
and almost surely in $\mathscr{D}([0,T];\mH)\cap\mathrm{L}^2(0,T;\mV).$  
\end{Pro}

\begin{proof}
Let $\Y_m(\cdot)$ be the solution of \eqref{2F3}. Let us set $\widetilde{\Y}_m=\Y_m-\mathrm{K}_m$. Then, $\widetilde{\Y}_m$ satisfies the following system:
\begin{equation}
\left\{
\begin{aligned}
\frac{\d\widetilde{\Y}_m(t)}{\d t}&=-[\A\widetilde{\Y}_m(t)+\B(\widetilde{\Y}_m(t)+\mathrm{K}_m(t))], \ t\in(0,T),\\
\widetilde{\Y}_m(0)&=\mathrm{P}_m\x.
\end{aligned}
\right.
\end{equation}
Taking inner product with $\widetilde{\Y}_m$, we find 
\begin{align}\label{442}
\frac{1}{2}\frac{\d}{\d t}\|\widetilde{\Y}_m(t)\|^2+\|\widetilde{\Y}_m(t)\|_{\frac{1}{2}}^2=-\langle \B(\widetilde{\Y}_m(t)+\mathrm{K}_m(t)), \widetilde{\Y}_m(t)\rangle.
\end{align}
Using H\"older's, Ladyzhenskaya and Young's inequalities, we estimate $\langle \B(\widetilde{\Y}_m+\mathrm{K}_m), \widetilde{\Y}_m\rangle$ as 
\begin{align}\label{443}
\langle \B(\widetilde{\Y}_m+\mathrm{K}_m), \widetilde{\Y}_m\rangle&= \langle \B(\widetilde{\Y}_m,\mathrm{K}_m), \widetilde{\Y}_m\rangle+ \langle \B(\mathrm{K}_m,\mathrm{K}_m), \widetilde{\Y}_m\rangle\nonumber\\&\leq \|\widetilde{\Y}_m\|_{\mathbb{L}^4}\|\widetilde{\Y}_m\|_{\frac{1}{2}}\|\mathrm{K}_m\|_{\mathbb{L}^4}+\|\mathrm{K}_m\|_{\mathbb{L}^4}^2\|\widetilde{\Y}_m\|_{\frac{1}{2}}\nonumber\\&\leq 2^{1/4}\|\widetilde{\Y}_m\|^{1/2}\|\widetilde{\Y}_m\|_{\frac{1}{2}}^{3/2}\|\mathrm{K}_m\|_{\mathbb{L}^4}+\frac{1}{4}\|\widetilde{\Y}_m\|_{\frac{1}{2}}^2+\|\mathrm{K}_m\|_{\mathbb{L}^4}^4\nonumber\\&\leq \frac{1}{2}\|\widetilde{\Y}_m\|_{\frac{1}{2}}^2+\frac{27}{2}\|\mathrm{K}_m\|_{\mathbb{L}^4}^4\|\widetilde{\Y}_m\|^2+\|\mathrm{K}_m\|_{\mathbb{L}^4}^4.
\end{align}
Integrating \eqref{442} from $0$ to $t$ and then using \eqref{443}, we get 
\begin{align}
\|\widetilde{\Y}_m(t)\|^2+\int_0^t\|\widetilde{\Y}_m(s)\|_{\frac{1}{2}}^2\d s\leq \|\x\|^2+\int_0^t\|\mathrm{K}_m(s)\|_{\mathbb{L}^4}^4\d s+ 27\int_0^t\|\mathrm{K}_m(s)\|_{\mathbb{L}^4}^4\|\widetilde{\Y}_m(s)\|^2\d s.
\end{align}
An application of Gronwall's inequality and \eqref{4.40} yields 
\begin{align}
\|\widetilde{\Y}_m(t)\|^2+\int_0^t\|\widetilde{\Y}_m(s)\|_{\frac{1}{2}}^2\d s&\leq \left(\|\x\|^2+\int_0^t\|\mathrm{K}_m(s)\|_{\mathbb{L}^4}^4\d s\right)\exp\left( 27\int_0^t\|\mathrm{K}_m(s)\|_{\mathbb{L}^4}^4\d s\right)\nonumber\\&\leq (1+\|\x\|^2) \exp\left( C\sup_{t\in[0,T]}\|\mathrm{K}(t)\|^2\int_0^T\|\mathrm{K}(t)\|_{\frac{1}{2}}^2\d t\right)<+\infty,
\end{align}
$\mathbb{P}$-a.s.	This implies that for each $\omega\in\Omega$, $\widetilde{\Y}_m$ is a bounded sequence in $\mathrm{L}^{\infty}(0,T;\mH)$ and $\mathrm{L}^2(0,T;\mV)$. By a standard argument (based on compactness and uniqueness of limits), $\widetilde{\Y}_m$ converges almost surely to $\widetilde{\Y}$ in $\mathrm{L}^2(0,T;\mH)$, which satisfies:
\begin{equation}\label{446}
\left\{
\begin{aligned}
\frac{\d\widetilde{\Y}(t)}{\d t}&=-[\A\widetilde{\Y}(t)+\B(\widetilde{\Y}(t)+\mathrm{K}(t))], \ t\in(0,T), \\
\widetilde{\Y}(0)&=\x.
\end{aligned}
\right.
\end{equation}
Let us set $\Y=\widetilde{\Y}+\mathrm{K}$, so that $\Y$ satisfies: 
\begin{equation}
\left\{
\begin{aligned} 
\d\Y(t)&=-[\A\Y(t)+\B(\Y(t))]\d t+\A^{-\frac{\va}{2}}\d\W(t)+ \int_{\mZ}\G(t,z) \wi\uppi(\d t,\d z), \   t\in (0,T),  \\
\Y(0)&=\x, \ \x\in \mH.
\end{aligned}
\right.\end{equation}

The existence and uniqueness of solution of the  equation \eqref{2F3} is given in Theorem \ref{eu1}. We only need to prove the strong convergence of $\Y_m(\cdot)$ to $\Y(\cdot)$ in the given topology. Using the It\^o formula and taking expectation, we get
\begin{align} \label{c1}
\hspace{-.12in}\mE\left[\|\Y_m(t)\|^2+2\int_0^t\|\Y_m(s)\|^2_{\frac{1}{2}}\d s\right]
&=\|x_m\|^2 +t\sum_{j=1}^m\lambda_j^{-\var} +\int_0^t\int_{\mZ_m}\|\G_m(t,z)\|^2\mu(\d z) \d t\\&\leq \|x\|^2 +T\tr(\A^{-\var}) +\int_0^T\iZ\|\G(t,z)\|^2\mu(\d z) \d t. \no
\end{align}
So, $\{\Y_m\}$ is bounded in $\mathrm{L}^2(\Om;\mH)\cap   \mathrm{L}^2(\Om;\mathrm{L}^2(0,T;\mV)).$  By the Banach-Alaoglu theorem, there exists a subsequence of $\{\Y_m\}$ which we denote again by $\{\Y_m\}$ such that 
\begin{equation} \label{cc2}
\left\{
\begin{aligned}
&\Y_m(t) \rightharpoonup \Y(t), \ \ \mbox{in} \ \ \mathrm{L}^2(\Om;\mH), \\
&\Y_m \rightharpoonup \Y, \ \ \mbox{in} \ \ \mathrm{L}^2(\Om;\mathrm{L}^2(0,T;\mV)),  \ \mbox{as}  \ m\to\infty. 
\end{aligned}
\right.
\end{equation} 
Applying the infinite dimensional It\^o formula (\cite{Me}) to $\|\Y(\cdot)\|^2,$ we get
\begin{align} \label{c3}
\mE\left[\|\Y(t)\|^2+2\int_0^t\|\Y(s)\|^2_{\frac{1}{2}}\d s\right]
=\|x\|^2 +t\tr(\A^{-\var}) +\int_0^t\iZ\|\G(s,z)\|^2\mu(\d z) \d s.
\end{align}  
From the equalities of (\ref{c1}) and (\ref{c3}), we obtain 
\begin{align*}
&\left|\E\left[\|\Y_m(t)\|^2-\|\Y(t)\|^2+2\int_0^t\left(\|\Y_m(s)\|_{\frac{1}{2}}^2-\|\Y(s)\|_{\frac{1}{2}}^2\right)\d s\right]\right|\no\\&\leq \|x_m-x\|\left(\|x_m\|+\|x\|\right)+t\Tr[(\mathrm{I}-\mathrm{P}_m)\A^{-\var}]\no\\&\quad+\left|\int_0^t\int_{\mZ_m}\|\G_m(t,z)\|^2\mu(\d z)\d t-\int_0^t\iZ\|\G_m(t,z)\|^2\mu(\d z)\d t\right|\no\\&\quad+\left|\int_0^t\iZ\|\G_m(t,z)\|^2\mu(\d z)\d t-\int_0^t\iZ\|\G(t,z)\|^2\mu(\d z)\d t\right|\no\\&\leq 2\|x\| \|x_m-x\|+t\|\mathrm{I}-\mathrm{P}_m\|_{\mathcal{L}(\mH)}\Tr(\A^{-\var})+\int_0^t\int_{{\mathcal{Z}\backslash\mathcal{Z}_m}}\|\mathrm{\G}_m(s,z)\|^2\mu(\d z)\d s\nonumber\\&\quad +\int_0^t\iZ\left(\|\G_m(t,z)\|+\|\G(t,z)\|\right)\|\G_m(t,z)-\G(t,z)\|\mu(\d z)\d t\no\\&\leq 2\|x\| \|x_m-x\|+t\|\mathrm{I}-\mathrm{P}_m\|_{\mathcal{L}(\mH)}\Tr(\A^{-\var})+\int_0^t\int_{{\mathcal{Z}\backslash\mathcal{Z}_m}}\|\mathcal{\G}(s,z)\|^2\mu(\d z)\d s\nonumber\\&\quad+2\|\mathrm{I}-\mathrm{P}_m\|_{\mathcal{L}(\mH)}\int_0^t\iZ\|\mathrm{\G}(s,z)\|^2\mu(\d z)\d s\no\\& \to 0\text{ as }m\to\infty.
\end{align*}
It leads to 
\begin{align} \label{c4}
\mE\left[\|\Y_m(t)\|^2+2\int_0^t\|\Y_m(s)\|^2_{\frac{1}{2}}\d s\right]
\to \mE\left[\|\Y(t)\|^2+2\int_0^t\|\Y(s)\|^2_{\frac{1}{2}}\d s\right], \ \  \mbox{as} \ \ m\to \infty. 
\end{align} 
Clearly, by the weak convergence of $\{\Y_m\}$ established in \eqref{cc2} together with \eqref{c4}, we get  the strong convergence of $\{\Y_m\}$ to $\Y$ in  $\mathrm{L}^2(\Om;\mH)$ and $\mathrm{L}^2(\Om;\mathrm{L}^2(0,T;\mV))$ respectively. Since $\Y_m\in\mathrm{L}^{p}(\Omega;\mathrm{L}^{\infty}(0,T;\mH)\cap\mathrm{L}^2(0,T;\mV))$, for all $p\geq 2$, we  also get the convergence $\Y_m\to\Y\in\mathrm{L}^{p}(\Omega;\mathrm{L}^{\infty}(0,T;\mH)\cap\mathrm{L}^2(0,T;\mV))$. 
Moreover, applying $\mathrm{P}_m$ on \eqref{2p19} and taking difference with \eqref{2F3}, we get
\begin{equation}\d(\Y_m(t)-\mathrm{P}_m\Y(t))=-\A\big(\Y_m(t)-\mathrm{P}_m\Y(t)\big)\d t + \big(\mathrm{P}_m\B(\Y_m(t))-\mathrm{P}_m\B(\Y(t))\big)\d t.
\end{equation}
Taking inner product with $(\Y_m-\mathrm{P}_m\Y)$, integrating by parts and applying Young's inequality,  we obtain
\begin{align}\label{3.59}
&\frac{1}{2}\|\Y_m(t)-\mathrm{P}_m\Y(t)\|^2+\int_0^t\|\Y_m(s)-\mathrm{P}_m\Y(s)\|^2_{\frac{1}{2}}\d s\no\\
&=\i0t \la  \mathrm{P}_m\B(\Y_m(s))-\mathrm{P}_m\B(\Y(s)),\Y_m(s)-\mathrm{P}_m\Y(s)\ra \d s.
\end{align}
We estimate $\la  \mathrm{P}_m\B(\Y_m)-\mathrm{P}_m\B(\Y),\Y_m-\mathrm{P}_m\Y\ra$ as
\begin{align}\label{3.60}
&\la  \mathrm{P}_m\B(\Y_m)-\mathrm{P}_m\B(\Y),\Y_m-\mathrm{P}_m\Y\ra\nonumber\\&= \la  \B(\Y_m,\Y_m-\mathrm{P}_m\Y)+\B(\Y_m-\Y,\mathrm{P}_m\Y)+\B(\Y,\mathrm{P}_m\Y-\Y),\Y_m-\mathrm{P}_m\Y\ra\nonumber\\&=  \la  \B(\Y_m-\Y,\mathrm{P}_m\Y),\Y_m-\mathrm{P}_m\Y\ra-\la \B(\Y,\Y-\mathrm{P}_m\Y),\Y_m-\mathrm{P}_m\Y\ra.
\end{align}
Let us estimate the first term from the right hand side of the equality \eqref{3.60} using H\"older's, Ladyzhenskaya and Young's inequalities as
\begin{align}\label{3.61}
&	|\la  \B(\Y_m-\Y,\mathrm{P}_m\Y),\Y_m-\mathrm{P}_m\Y\ra|\nonumber\\&=|\la  \B(\Y_m-\Y,\Y_m-\mathrm{P}_m\Y),\mathrm{P}_m\Y\ra|\nonumber\\&\leq \|\Y_m-\Y\|_{\mL^4}\|\Y_m-\mathrm{P}_m\Y\|_{\frac{1}{2}}\|\mathrm{P}_m\Y\|_{\mL^4}\nonumber\\& \leq \frac{1}{4}\|\Y_m-\mathrm{P}_m\Y\|_{\frac{1}{2}}^2+\|\Y_m-\Y\|_{\mL^4}^2\|\mathrm{P}_m\Y\|_{\mL^4}^2 \nonumber\\&\leq \frac{1}{4}\|\Y_m-\mathrm{P}_m\Y\|_{\frac{1}{2}}^2+\sqrt{2}\|\Y_m-\Y\|\|\Y_m-\Y\|_{\frac{1}{2}}\|\Y\|\|\Y\|_{\frac{1}{2}}\nonumber\\&\leq \frac{1}{4}\|\Y_m-\mathrm{P}_m\Y\|_{\frac{1}{2}}^2+\frac{1}{2}\|\Y_m-\Y\|_{\frac{1}{2}}^2+\|\Y_m-\Y\|^2\|\Y\|^2\|\Y\|_{\frac{1}{2}}^2.
\end{align}
Similarly, we have 
\begin{align}\label{3.62}
&	|\la \B(\Y,\Y-\mathrm{P}_m\Y),\Y_m-\mathrm{P}_m\Y\ra|\nonumber\\&\leq  \frac{1}{4}\|\Y_m-\mathrm{P}_m\Y\|_{\frac{1}{2}}^2+\frac{1}{2}\|\Y-\mathrm{P}_m\Y\|_{\frac{1}{2}}^2+\|\Y-\mathrm{P}_m\Y\|^2\|\Y\|^2\|\Y\|_{\frac{1}{2}}^2.
\end{align}
Combining \eqref{3.61} and \eqref{3.62}, substituting it in \eqref{3.59},  taking supremum over time and then taking expectation, we find  
\begin{align}\label{3.63}
&\mE\left[\sup_{t\in[0,T]}\|\Y_m(t)-\mathrm{P}_m\Y(t)\|^2\right]+\mE\left[\int_0^T\|\Y_m(t)-\mathrm{P}_m\Y(t)\|^2_{\frac{1}{2}}\d t\right] \nonumber\\&\leq \mE\left[\int_0^T\|\Y_m(t)-\Y(t)\|^2_{\frac{1}{2}}\d t\right]+ \mE\left[\int_0^T\|\Y(t)-\mathrm{P}_m\Y(t)\|^2_{\frac{1}{2}}\d t\right]\nonumber\\&\quad +2\mE\left[\int_0^T\|\Y_m(t)-\Y(t)\|^2\|\Y(t)\|^2\|\Y(t)\|_{\frac{1}{2}}^2\d t\right]\nonumber\\&\quad+2\mE\left[\int_0^T\|\Y(t)-\mathrm{P}_m\Y(t)\|^2\|\Y(t)\|^2\|\Y(t)\|_{\frac{1}{2}}^2\d t\right]=:\sum_{i=1}^4I_i.
\end{align}
Note that $I_1\to 0$ as $m\to\infty$ using the fact that $\Y_m\to\Y\in \mathrm{L}^2(\Omega;\mathrm{L}^2(0,T;\mV))$. For the term $I_2$, we have 
\begin{align*}
I_2\leq \|\mathrm{I}-\mathrm{P}_m\|_{\mathcal{L}(\mV,\mV)}\mE\left[\int_0^T\|\Y(t)\|_{\frac{1}{2}}^2\d t\right]\to 0, \ \text{ as } \ m\to\infty,
\end{align*}
since $\Y\in \mathrm{L}^2(\Omega;\mathrm{L}^2(0,T;\mV))$. Using H\"older's inequality, we estimate $I_3$ as 
\begin{align*}
I_4&\leq \mE\left[\sup_{t\in[0,T]}\|\Y_m(t)-\Y(t)\|^2\sup_{t\in[0,T]}\|\Y(t)\|^2\int_0^T\|\Y(t)\|_{\frac{1}{2}}^2\d t\right]\nonumber\\&\leq \left\{\mE\left[\sup_{t\in[0,T]}\|\Y_m(t)-\Y(t)\|^8\right]\right\}^{1/4} \left\{\mE\left[\sup_{t\in[0,T]}\|\Y(t)\|^2\right]\right\}^{1/4} \left\{\mE\left[\left(\int_0^T\|\Y(t)\|_{\frac{1}{2}}^2\d t\right)^2\right]\right\}^{1/2}\nonumber\\& \to 0  \ \text{ as } \ m\to\infty,
\end{align*}
since $\Y_m\to\Y\in\mathrm{L}^p(\Omega;\mathrm{L}^2(0,T;\mH))$,  and $\Y\in\mathrm{L}^{2p}(\Omega;\mathrm{L}^2(0,T;\mH))\cap\mathrm{L}^p(\Omega;\mathrm{L}^2(0,T;\mV))$, for all $p\geq 2$. Thus from \eqref{3.63}, we infer that 
\begin{align*}
\mE\left[\sup_{t\in[0,T]}\|\Y_m(t)-\mathrm{P}_m\Y(t)\|^2\right]+\mE\left[\int_0^T\|\Y_m(t)-\mathrm{P}_m\Y(t)\|^2_{\frac{1}{2}}\d t\right]\to 0, \ \text{ as } \ m\to\infty. 
\end{align*}
Therefore, for any fixed  $\varepsilon>0,$ we get by the Chebyshev inequality that
\begin{align*}
\mP\left(\sup_{t\in[0,T]}\|\Y_m(t)-\mathrm{P}_m\Y(t)\|\geq \varepsilon\right)&\leq \frac{1}{\varepsilon^2}\mE\left(\sup_{t\in[0,T]}\|\Y_m(t)-\mathrm{P}_m\Y(t)\|^2\right) \to 0 \  \mbox{as} \ \ m\to \infty
\end{align*}
and by the Markov inequality that 
\begin{align}\label{449}
\mP\left(\int_0^T\|\Y_m(t)-\mathrm{P}_m\Y(t)\|_1^2\d t\geq \varepsilon\right)\leq \frac{1}{\varepsilon}\mE\left(\int_0^T\|\Y_m(t)-\mathrm{P}_m\Y(t)\|_1^2\d t\right) \to 0,
\end{align}
as $m\to\infty$. Moreover
\begin{align}\label{4.50}
&	\mP\left(\sup_{t\in[0,T]}\|\Y_m(t)-\Y(t)\|\geq \varepsilon\right)\nonumber\\& \leq 	\mP\left(\sup_{t\in[0,T]}\|\Y_m(t)-\mathrm{P}_m\Y(t)\|\geq \frac{\varepsilon}{2}\right)+	\mP\left(\sup_{t\in[0,T]}\|\mathrm{P}_m\Y(t)-\Y(t)\|\geq \frac{\varepsilon}{2}\right)\to 0,
\end{align}
as $m\to\infty$.  Similarly, arguing for \eqref{449} as in \eqref{4.50}, we get that  there exists a subsequence of $\Y_m$ (denoted by $\Y_m$ for notational simplification, see Theorem 17.3, \cite{J}) such that 
$\Y_m \to \Y$ almost surely in $\mathscr{D}([0,T];\mH)\cap\mathrm{L}^2(0,T;\mV).$ Since the solution $\Y(\cdot)$ is unique, the entire sequence $\Y_m(\cdot)$ converges almost surely. This completes  the proof.
\end{proof}

\begin{Pro}\label{prop4.4}
Suppose that Assumption \ref{ass2.1} is satisfied. 	Then, $\Y\in\mathscr{D}([0,T];\mathcal{D}(\A^{\alpha}))$, $\mathbb{P}$-a.s., for any $\alpha\in(0,\frac{1}{2})$ and satisfies the estimate: 
\begin{align}\label{a}
\sup_{s\in[0,T]}\|\Y(t)\|_{\alpha}\leq C\left(\|\x\|_{\alpha},\sup_{t\in[0,T]}\|\mathrm{K}(t)\|_{\alpha},\int_0^T\|\mathrm{K}(t)\|_{\frac{1}{2}}^2\d t\right)<+\infty, \ \mathbb{P}\text{-a.s.}
\end{align}
Moreover,  $\Y_m(\cdot),$ the solution of  \eqref{2F3} converges to the unique solution  $\Y(\cdot)$  almost surely in $\mathscr{D}([0,T];\mathcal{D}(\A^{\alpha}))\cap\mathrm{L}^2(0,T;\mathcal{D}(\A^{\alpha+\frac{1}{2}})).$ 
\end{Pro}

\begin{proof} 
Let us now take inner product with $\A^{2\alpha}\widetilde{\Y}(\cdot)$ to the first equation in \eqref{446} to obtain 
\begin{align}\label{4.61}
\frac{1}{2}\frac{\d}{\d t}\|\widetilde{\Y}(t)\|_{\alpha}^2+\|\widetilde{\Y}(t)\|_{\alpha+\frac{1}{2}}^2= -\la\B(\widetilde{\Y}(t)+\mathrm{K}(t)),\A^{2\alpha}\widetilde{\Y}(t)\ra.
\end{align}
Using \eqref{ne}, we estimate $\la\B(\widetilde{\Y}+\mathrm{K}),\A^{2\alpha}\widetilde{\Y}\ra$ as 
\begin{align}\label{4.62}
&	\la\B(\widetilde{\Y}+\mathrm{K}),\A^{2\alpha}\widetilde{\Y}\ra\no\\&\leq   
\la\B(\widetilde{\Y}, \widetilde{\Y}),\A^{2\alpha}\widetilde{\Y}\ra
+	\la\B(\widetilde{\Y},\mathrm{K}),\A^{2\alpha}\widetilde{\Y}\ra
+	\la\B(\mathrm{K}, \widetilde{\Y}),\A^{2\alpha}\widetilde{\Y}\ra
+\la\B(\mathrm{K}, \mathrm{K}),\A^{2\alpha}\widetilde{\Y}\ra\no\\&\leq C \|\widetilde{\Y}\|_{\alpha}\|\widetilde{\Y}\|_{\frac{1}{2}}\|\widetilde{\Y}\|_{\alpha+\frac{1}{2}}+C  \|\widetilde{\Y}\|_{\alpha}\|\mathrm{K}\|_{\frac{1}{2}}\|\widetilde{\Y}\|_{\alpha+\frac{1}{2}} +C\|\mathrm{K}\|_{\alpha}\|\widetilde{\Y}\|_{\frac{1}{2}}\|\widetilde{\Y}\|_{\alpha+\frac{1}{2}}+C
\|\mathrm{K}\|_{\alpha}\|\mathrm{K}\|_{\frac{1}{2}}\|\widetilde{\Y}\|_{\alpha+\frac{1}{2}} \nonumber\\&\leq \frac{1}{2} \|\widetilde{\Y}\|_{\alpha+\frac{1}{2}}^2+C\left(\|\widetilde{\Y}\|_{\frac{1}{2}}^2+\|\mathrm{K}\|_{\frac{1}{2}}^2\right)\|\widetilde{\Y}\|_{\alpha}^2+C\|\mathrm{K}\|_{\alpha}^2\|\widetilde{\Y}\|_{\frac{1}{2}}^2+C\|\mathrm{K}\|_{\alpha}^2\|\mathrm{K}\|_{\frac{1}{2}}^2.
\end{align}
Using \eqref{4.62} in \eqref{4.61} and then 	applying Gronwall's inequality, we arrive at
\begin{align}\label{4.63}
&	\|\widetilde{\Y}(t)\|_{\alpha}^2+\int_0^t\|\widetilde{\Y}(s)\|_{\alpha+\frac{1}{2}}^2\d s\nonumber\\&\leq \left(  \|\x\|_{\alpha}^2+C\sup_{s\in[0,t]}\|\mathrm{K}(s)\|_{\alpha}^2\int_0^t\|\widetilde{\Y}(s)\|_{\frac{1}{2}}^2\d s+C\sup_{s\in[0,t]}\|\mathrm{K}(s)\|_{\alpha}^2\int_0^t\|\mathrm{K}(s)\|_{\frac{1}{2}}^2\d s \right) \nonumber\\&\qquad \times\exp\left(C\int_0^t\left(\|\widetilde{\Y}(s)\|_{\frac{1}{2}}^2+\|\mathrm{K}(s)\|_{\frac{1}{2}}^2\right)\d s\right)<+\infty,
\end{align}
$\mathbb{P}$-a.s., since $\widetilde{\Y}\in\mathrm{L}^2(0,T;\mV)$, 	$\mathbb{P}$-a.s and \eqref{4.40}. The estimate \eqref{a} follows from \eqref{4.63}. 

Taking inner product with $\A^{2\alpha}\Y_m(\cdot)$ to the equation\eqref{3.59}, we obtain
\begin{align}\label{4.55}
&\frac{1}{2}\|\Y_m(t)-\mathrm{P}_m\Y(t)\|^2_{\alpha}+\int_0^t\|\Y_m(s)-\mathrm{P}_m\Y(s)\|^2_{\alpha+\frac{1}{2}}\d s\no\\
&=\i0t\la\mathrm{P}_m\B(\Y_m(s))-\mathrm{P}_m\B(\Y(s)),\A^{2\alpha}(\Y_m(s)-\mathrm{P}_m\Y(s))\ra \d s.
\end{align}
We estimate $\la\mathrm{P}_m\B(\Y_m)-\mathrm{P}_m\B(\Y),\A^{2\alpha}(\Y_m-\mathrm{P}_m\Y)\ra$ as
\begin{align}\label{4.56}
&\la\mathrm{P}_m\B(\Y_m)-\mathrm{P}_m\B(\Y),\A^{2\alpha}(\Y_m-\mathrm{P}_m\Y)\ra\nonumber\\&=\la\B(\Y_m,\Y_m-\mathrm{P}_m\Y),\A^{2\alpha}(\Y_m-\mathrm{P}_m\Y) \ra+\la\B(\Y_m-\mathrm{P}_m\Y,\mathrm{P}_m\Y),\A^{2\alpha}(\Y_m-\mathrm{P}_m\Y) \ra\nonumber\\&\quad + \la\B(\mathrm{P}_m\Y-\Y,\mathrm{P}_m\Y),\A^{2\alpha}(\Y_m-\mathrm{P}_m\Y) \ra+\la\B(\Y,\mathrm{P}_m\Y-\Y),\A^{2\alpha}(\Y_m-\mathrm{P}_m\Y) \ra\nonumber\\&=:\sum_{i=1}^4I_i.
\end{align}
We estimate $I_i$, for $i=1,\ldots,4$ using the trilinear estimate \eqref{ne} and interpolation inequality \eqref{II} as follows:
\begin{align*}
I_1&\leq C \|\Y_m\|_{\alpha}\|\Y_m-\mathrm{P}_m\Y\|_{\frac{1}{2}}\|\Y_m-\mathrm{P}_m\Y\|_{\alpha+\frac{1}{2}}\leq C\|\Y_m\|_{\alpha}\|\Y_m-\mathrm{P}_m\Y\|^{2\alpha}_{\alpha}\|\Y_m-\mathrm{P}_m\Y\|_{\alpha+\frac{1}{2}}^{2(1-\alpha)}\nonumber\\&\leq \frac{1}{8}\|\Y_m-\mathrm{P}_m\Y\|_{\alpha+\frac{1}{2}}^{2}+C\|\Y_m\|_{\alpha}^{\frac{1}{\alpha}}\|\Y_m-\mathrm{P}_m\Y\|^{2}_{\alpha}, \nonumber\\I_2&\leq C\|\Y_m-\mathrm{P}_m\Y\|_{\alpha}\|\mathrm{P}_m\Y\|_{\frac{1}{2}}\|\Y_m-\mathrm{P}_m\Y\|_{\alpha+\frac{1}{2}} \leq \frac{1}{8}\|\Y_m-\mathrm{P}_m\Y\|_{\alpha+\frac{1}{2}}^2+C\|\Y\|_{\frac{1}{2}}^2\|\Y_m-\mathrm{P}_m\Y\|_{\alpha}^2,\nonumber\\ 
I_3&\leq \|\mathrm{P}_m\Y-\Y\|_{\alpha}^2\|\mathrm{P}_m\Y\|_{\frac{1}{2}}\|\Y_m-\mathrm{P}_m\Y\|_{\alpha+\frac{1}{2}} \nonumber
\leq \frac{1}{8}\|\Y_m-\mathrm{P}_m\Y\|_{\alpha+\frac{1}{2}}^2+C\|(\mathrm{P}_m-\mathrm{I})\Y\|_{\alpha}^2\|\Y\|_{\frac{1}{2}}^2,\nonumber
\end{align*}
and 
\begin{align*}
I_4\leq C\|\Y\|_{\alpha}\|\mathrm{P}_m\Y-\Y\|_{\frac{1}{2}}\|\Y_m-\mathrm{P}_m\Y\|_{\alpha+\frac{1}{2}} \nonumber\leq \frac{1}{8}\|\Y_m-\mathrm{P}_m\Y\|_{\alpha+\frac{1}{2}}^2+C\|\Y\|_{\alpha}^2\|\mathrm{P}_m\Y-\Y\|_{\frac{1}{2}}^2.
\end{align*}
Thus, from \eqref{4.55}, it is immediate that 
\begin{align}\label{457}
&\|\Y_m(t)-\mathrm{P}_m\Y(t)\|^2_{\alpha}+\int_0^t\|\Y_m(s)-\mathrm{P}_m\Y(s)\|^2_{\alpha+\frac{1}{2}}\d s\no\\&\leq C\int_0^t\|\Y_m(s)\|_{\alpha}^{\frac{1}{\alpha}}\|\Y_m(s)-\mathrm{P}_m\Y(s)\|^{2}_{\alpha}\d s+C\int_0^t\|\Y(s)\|_{\frac{1}{2}}^2\|\Y_m(s)-\mathrm{P}_m\Y(s)\|_{\alpha}^2\d s \nonumber\\&\quad+C\int_0^t\|(\mathrm{P}_m-\mathrm{I})\Y(s)\|_{\alpha}^2\|\Y(s)\|_{\frac{1}{2}}^2\d s+C\int_0^t\|\Y(s)\|_{\alpha}^2\|\mathrm{P}_m\Y(s)-\Y(s)\|_{\frac{1}{2}}^2\d s.
\end{align}
Applying Gronwall's inequality, we arrive at
\begin{align}\label{4.69}
&\|\Y_m(t)-\mathrm{P}_m\Y(t)\|^2_{\alpha}+\int_0^t\|\Y_m(s)-\mathrm{P}_m\Y(s)\|^2_{\alpha+\frac{1}{2}}\d s\no\\&\leq  \bigg\{C\|(\mathrm{P}_m-\mathrm{I})\|_{\mathcal{L}(\mathcal{D}(\A^{\alpha});(\mathcal{D}(\A^{\alpha}))}^2\sup_{s\in[0,t]}\|\Y(s)\|_{\alpha}^2\int_0^t\|\Y(s)\|_{\frac{1}{2}}^2\d s\nonumber\\&\quad +C\|(\mathrm{P}_m-\mathrm{I})\|_{\mathcal{L}(\mV;\mV)}^2\sup_{s\in[0,t]}\|\Y(s)\|_{\alpha}^2\int_0^t\|\Y(s)\|_{\frac{1}{2}}^2\d s\bigg\}\nonumber\\&\qquad\times\exp\left(Ct\sup_{s\in[0,t]}\|\Y_m(s)\|_{\alpha}^{\frac{1}{\alpha}}+C\int_0^t\|\Y(s)\|_{\frac{1}{2}}^2\d s\right)\to 0,
\end{align}
$\mathbb{P}$-a.s., as $m\to\infty$, since  $\Y\in\mathscr{D}([0,T];\mathcal{D}(\A^{\alpha}))\cap\mathrm{L}^2(0,T;\mV),$ $\mathbb{P}$-a.s. From this, we further deduce that $\Y_m(\cdot),$ converges to   $\Y(\cdot)$  almost surely in $\mathscr{D}([0,T];\mathcal{D}(\A^{\alpha}))\cap\mathrm{L}^2(0,T;\mathcal{D}(\A^{\alpha+\frac{1}{2}}))$.
\end{proof}

\section{Proof of Theorem \ref{main1}}\label{sec5}
We first prove the existence of smooth solution for the transformed HJB equation \eqref{315}.
\begin{proof}[Proof of Proposition \ref{main0}]
We prove the convergence of \eqref{2p21} in several steps. 

\noindent\textbf{Step (1):} \emph{Pointwise convergence and continuity of the sequence $(\w_m)_{m\in\mathbb{N}}$.} To attain this aim, we appeal to  Arz\'ela-Ascoli theorem. For any $ l \in \mathbb{N}$, we define the compact sets as follows: 
\begin{align*}
&K_l:=\{\x\in \mathrm{P}_l \mD (\A^{\alpha}):\|\x\|_{\alpha}\leq l\}, \\
&\w_m^l=\w_m|_{K_l\times[0,T]} \ \ \mbox{and} \ \ \ \mathrm{D}_{\x}\w_m^l=\mathrm{D}_{\x}\w_m|_{K_l\times[0,T]} .
\end{align*}
From \eqref{2p21},  we have 
\begin{align}\label{3.39}
&\sup_{(\x,t)\in K_l\times[0,T]}|\mathrm{D}_t\w_m^l(t,\x)| \nonumber\\&\leq \sup_{(\x,t)\in K_l\times[0,T]}\left[ |\mcL_m \w_m^l(t,\x)|+2\ga \|\x\|_{\frac{1}{2}}^2|\w_m^l(t,\x)|\right]\nonumber\\&\quad +\sup_{(\x,t)\in K_l\times[0,T]}\left[|\widetilde{\mathrm{F}}_m(\x,\w_m^l(t,\x),\mathrm{D}_\x \w_m^l(t,\x))|+ |\wi g(\x)|\right]=\sum_{i=1}^4J_i,
\end{align}
where $\wi g(\x)=e^{-\ga\|\x\|^2}\|\text{curl} \ \x\|^2$, and $\mcL_m$ and $\wi \F_m$ are defined in \eqref{g12} and \eqref{ft}, respectively.  The terms in $J_1$ can be estimated as follows: 
\begin{align}\label{3.40} 
J_1&\leq \sup_{(\x,t)\in K_l\times[0,T]}\bigg[ \frac{1}{2}\tr(\A^{-\va} \mathrm{D}_\x^2\w_m^l(t,\x))+|\la\A\x+\B(\x),\mathrm{D}_\x\w_m^l(t,\x)\ra| \no\\
&\quad+\left|\int_{\mathcal{Z}_m}\big[\w_m^l(t,\x+\G_m(t,z))-\w_m^l(t,\x)-\la \G_m(t,z),\mathrm{D}_\x\w_m^l(t,\x)\ra\big]\mu(\d z)\right|\bigg]=\sum_{i=5}^7J_i.
\end{align}
It is clear from \eqref{3p35} in Proposition \eqref{prop3.2} that 
\begin{align} \l{3.41}
\sup_{t \in [0,T]}\|\w_m^l(t,\cdot)\|_{\alpha,d,2} \leq \sup_{t\in[0, T]}\|\w_m(t,\cdot)\|_{\alpha,d,2}\leq C(\alpha_1,\alpha).
\end{align} It leads to the following estimate
\begin{align}
J_5&\leq  \sup_{(\x,t)\in K_l\times[0,T]}\frac{1}{2}\tr(\A^{-\va} \mathrm{D}_\x^2\w_m^l(t,\x)) \no\\&\leq \frac{1}{2}\text{Tr}(\A^{-(\va-\alpha)})\sup_{(\x,t)\in K_l\times[0,T]}|\A^{-\alpha}\mathrm{D}_{\x}^2\w_m^l(t,\x)|_{\mathcal{L}^2(\mH^2;\mR)}\nonumber\\&\leq  \frac{1}{2}\tr(\A^{-(\va-\alpha)})\sup_{(\x,t)\in K_l\times[0,T]}\sup_{\|\mathbf{h}\|_{\alpha}=1}|\langle \mathrm{D}_{\x}^2\w_m^l(t,\x)\mathbf{h},\mathbf{h}\rangle|\nonumber\\&\leq \frac{1}{2}\tr(\A^{-(\va-\alpha)})\sup_{(\x,t)\in K_l\times[0,T]}(1+\|\x\|_{\alpha})^d\|\w_m^l\|_{\alpha,d,2}\leq C(l),
\end{align}
since $\tr(\A^{-(\va-\alpha)})<+\infty$, for $\va>1+\alpha.$ 
Note that for any $\x\in K_l,$ we have 
\begin{equation*}\left\{\begin{aligned}\|\A^{1+\alpha} \x\|^2&=\sum_{i=1}^r\lambda_i^{2(1+\alpha)}\left\|(\x,e_i)\right\|^2\leq \lambda_l^{2(1+\alpha)} l\|\x\|^2\leq \lambda_l^{2(1+\alpha)} l\|\x\|_\alpha^2,\\
\|\A^{\alpha}\B_m(\x)\|^2&\leq \|\mathrm{P}_m(\x\cdot\nabla)\x\|_{\alpha}\leq \|(\x\cdot\nabla)\x\|_{\alpha+1} \leq C\|\x\|_{\alpha+1}\|\nabla\x\|_{\alpha+1}\leq C\|\x\|_{\alpha+\frac{3}{2}}^2\\&\leq C\lambda_l^{2(\alpha+\frac{3}{2})}l\|\x\|^2\leq C\lambda_l^{2\alpha+3}l\|\x\|^2_{\alpha}.
\end{aligned}\right.\end{equation*}
Using the above nonlinear estimates, we get
\begin{align}
J_6&\leq C\sup_{(\x,t)\in K_l\times[0,T]}\left[\|\A^{1+\alpha}\x\|+\|\A^{\alpha}\B(\x)\|\right]\|\A^{-\alpha}\mathrm{D}_{\x}\w_m^l(t,\x)\|
\nonumber\\&\leq C(l,\lambda_l)\sup_{(\x,t)\in K_l\times[0,T]}\sup_{\|\mathbf{h}\|_{\alpha}=1}|\langle \mathrm{D}_{\x}\w_m^l(t,\x),\mathbf{h}\rangle|\leq C(l,\lambda_l).
\end{align}
Finally, we estimate $J_7$ using Taylor's formula for some $0<\theta<1$ as 
\begin{align}
J_7&=\sup_{(\x,t)\in K_l\times[0,T]}\frac{1}{2}\left|\int_{\mathcal{Z}_m}\la \mathrm{D}_\x^2\w_m^l(t,\x+\theta\G_m(t,z))\G_m(t,z),\G_m(t,z)\ra\mu(\d z)\right|\nonumber\\&\quad \leq \frac{1}{2} \sup_{(\x,t)\in K_l\times[0,T]}\int_{\mathcal{Z}_m}\|\A^{-\alpha} \mathrm{D}_\x^2\w_m^l(t,\x+\theta\G_m(t,z))\G_m(t,z)\|\|\G_m(t,z)\|_{\alpha}\mu(\d z)\nonumber\\&\quad \leq \frac{1}{2} \sup_{(\x,t)\in K_l\times[0,T]}\int_{\mathcal{Z}_m}\|\A^{-\alpha} \mathrm{D}_\x^2\w_m^l(t,\x+\theta\G_m(t,z))\|_{\mathcal{L}^2(\mH^2;\mR)}\|\G_m(t,z)\|\|\G_m(t,z)\|_{\alpha}\mu(\d z)\nonumber\\&\quad \leq \frac{1}{2} \sup_{(\x,t)\in K_l\times[0,T]}\int_{\mathcal{Z}_m}(1+\|\x\|_{\alpha}+\|\G_m\|_{\alpha})^d\frac{1}{(1+\|\x+\theta\G_m\|_{\alpha})^d}\nonumber\\&\qquad\qquad\times\|\A^{-\alpha} \mathrm{D}_\x^2\w_m^l(t,\x+\theta\G_m(t,z))\|_{\mathcal{L}^2(\mH^2;\mR)}\|\G_m(t,z)\|_{\alpha}^2\mu(\d z)\nonumber\\&\quad \leq C \sup_{(\x,t)\in K_l\times[0,T]}\sup_{z\in\mathcal{Z}}\frac{1}{(1+\|\y(z)\|_{\alpha})^d}\|\A^{-\alpha} \mathrm{D}_\x^2\w_m^l(t,\y(z))\|_{\mathcal{L}^2(\mH^2;\mR)}\nonumber\\&\qquad\qquad  \times\int_{\mathcal{Z}}(1+\|\x\|_{\alpha})^d(1+\|\G(t,z)\|_{\alpha})^d\|\G(t,z)\|_{\alpha}^2\mu(\d z)\nonumber\\&\leq C(1+l)^d\sup_{t\in[0,T]}\left\{\|\w(t)\|_{\alpha,d,2}\int_{\mathcal{Z}}(1+\|\G(t,z)\|_{\alpha})^d\|\G(t,z)\|_{\alpha}^2\mu(\d z)\right\}\leq C(l).
\end{align}
By the calculations similar to the estimation of the terms in \eqref{3.11}, we have 
\begin{align}\label{3.44}
J_4&\leq C(R,l,\uptheta)\sup_{t\in[0,T]}\left\{\left(1+\int_{\mathcal{Z}}e^{2\uptheta\|\G\|^2}(1+\|\G(t,z)\|_{\alpha})^{d}\mu(\d z)\nonumber\right.\right.\\
&\left.\left.\quad\quad+\int_{\mathcal{Z}}\|\G(t,z)\|\mu(\d z)\right)\|\w_m^l(t)\|_{\alpha,d,1}\right\}\nonumber\\
&\leq C(R,l,\uptheta).
\end{align}
Combining \eqref{3.40}-\eqref{3.44} and substituting it in  \eqref{3.39}, we arrive at 
\begin{align}\label{3.45}
\sup_{(\x,t)\in K_l\times[0,T]}|\mathrm{D}_t\w_m^l(t,\x)|\leq C(l).
\end{align}

In order to apply Arzel\`a-Ascoli's theorem, we proceed as follows.  For any $\x,\y\in K_l$, $s,t\in[0,T]$, we obtain 
\begin{align}\label{3.46}
|\w_m^l(t,\x)-\w_m^l(s,\y)|&\leq 	|\w_m^l(t,\x)-\w_m^l(t,\y)|+	|\w_m^l(t,\y)-\w_m^l(s,\y)|\nonumber\\&=  \left|\int_0^1\frac{\d}{\d r}\w_m^l(t,r\x+(1-r)\y)\d r\right|+\left|\int_s^t\frac{\d}{\d r}\w_m^l(r,\y)\d r\right|\nonumber\\&\leq \sup_{(t,\z)\in[0,T]\times K_l}\left[\|\mathrm{D}_{\z}\w_m^l(t,\z)\|_{-\alpha}\|\x-\y\|_{\alpha}+|\mathrm{D}_t\w_m^l(t,\z)| |t-s|\right]\nonumber\\&\leq C\sup_{(t,\x)\in[0,T]\times K_l}\left[(1+\|\x\|_{\alpha})^d\|\w_m^l(t)\|_{\alpha,d,1}\|\x-\y\|_{\alpha}+\left|\mathrm{D}_t\w_m^l(t)\right||t-s|\right]\nonumber\\&\leq C(l)\left[\|\x-\y\|_{\alpha}+|t-s|\right],
\end{align}
where we used \eqref{3.41} and \eqref{3.45}.  It shows that $\{\w_m^l\}_{m\in\mathbb{N}}$ is equicontinuous on the compact set $[0,T]\times K_l$. By estimate \eqref{3.41} and Arzel\`a-Ascoli's theorem, there exists a function $ \w^l\in\C([0,T]\times K_l;\mR)$ and a subsequence $\{\w^l_{m_k}\}_{k\in\mathbb{N}}$ of $\{\w_m^l\}_{m\in\mathbb{N}}$  such that 
$$\w^l_{m_k}\to \w^l \ \text{ in } \  \C([0,T]\times K_l;\mR), \ \text{ as }\ k\to\infty.$$  

Moreover, by the estimate \eqref{3.41}, we have that for any $\|\x\|_\alpha,\|\y\|_\alpha \leq r$:
\begin{align}\label{3.461}
|\w_{m_k}^l(t,\x)-\w_{m_k}^l(t,\y)| \leq C(r)\|\x-\y\|_{\alpha}.
\end{align}
Letting $k,l\to\infty$, we get that exists a function $\w(t,\x)\in\mathrm{C}^{\alpha,d,2}$ such that  $\w$ is continuous on $[0,T]\times \Big(\bigcup\limits_{l\in\mathbb{N}}	 \mathrm{P}_l \mD (\A^{\alpha})\Big),$ so that it can be extended uniquely to $[0,T]\times \mD (\A^{\alpha}).$ Thus, we have
\begin{align}\label{360}
\w_{m_k}(t,\x)\to \w(t,\x), \ \text{ as } \ k\to\infty \ \text{ in }\ \mR,
\end{align}
for all $(t,\x)\in[0,T]\times\mD(\A^{\alpha})$. 

\noindent\textbf{Step (2):} \emph{Convergence and continuity of the derivative sequence  $(\mathrm{D}_{\x}\w_m)_{m\in\mathbb{N}}$.}	Using interpolation inequality for space of continuous functions, we have 
\begin{align}
&|\mathrm{D}_{\x}\w_m^l(t,\x)-\mathrm{D}_{\x}\w_m^l(s,\x)|\no\\&\leq \|\w_m^l(t)-\w_m^l(s)\|_{\C^1(K_l)}\no\\&\leq C \|\w_m^l(t)-\w_m^l(s)\|_{\C^2(K_l)}^{1/2} \|\w_m^l(t)-\w_m^l(s)\|_{\C^0(K_l)}^{1/2}\no\\&\leq C\left( \|\w_m^l(t)\|_{\C^2(K_l)}^{1/2}+\|\w_m^l(s)\|_{\C^2(K_l)}^{1/2}\right) \left\|\int_s^t\frac{\d}{\d r}\w_m^l(r)\d r\right\|_{\C^0(K_l)}^{1/2}\nonumber\\&\leq C\sup_{(t,\x)\in[0,T]\times K_l}\left[(1+\|\x\|_{\alpha})^d\|\w_m^l(t)\|_{\alpha,d,2}\left|\mathrm{D}_t\w_m^l(t)\right|^{1/2}\right]|t-s|^{1/2}\nonumber\\&\leq C(l)|t-s|^{1/2},
\end{align}
where we used Proposition \eqref{3.41}  and \eqref{3.45}. Moreover, for any $\x,\y\in\mathrm{K}_l$ and $t\in[0,T]$, we also have 
\begin{align}\label{3.48}
\|\mathrm{D}_{\x}\w_m^l(t,\x)-\mathrm{D}_{\x}\w_m^l(t,\y)\|_{-\alpha}&= \left\|\A^{-\alpha}\int_0^1\frac{\d}{\d r}\mathrm{D}_{\x}\w_m^l(t,r\x+(1-r)\y)\d r\right\|\nonumber\\&\leq \left\|\int_0^1\A^{-\alpha}\mathrm{D}_{\x}^2\w_m^l(t,r\x+(1-r)\y)(\x-\y)\d r\right\|\nonumber\\&\leq \|\A^{-\alpha}\mathrm{D}_{\x}^2\w_m^l(t,\z)\|_{\mathcal{L}(\mH^2;\mR)}\|\x-\y\|\nonumber\\&\leq C(l)\|\x-\y\|_{\alpha}.
\end{align}
It shows that $\{\mathrm{D}_{\x}\w_m^l\}_{m\in\mathbb{N}}$ is equicontinuous on the compact set $[0,T]\times K_l$. Once again using Proposition \ref{prop3.2} and Arzel\`a-Ascoli's theorem, there exists a function $W^l\in\C([0,T]\times K_l;\mD(\A^{-\alpha}))$ and a subsequence $\{\mathrm{D}_{\x}\w^l_{m_k}\}_{k\in\mathbb{N}}$ of $\{\mathrm{D}_{\x}\w_m^l\}_{m\in\mathbb{N}}$  such that 
\begin{align}\label{3.49}\mathrm{D}_{\x}\w^l_{m_k}\to W^l \ \text{ in } \  \C([0,T]\times K_l;\mD(\A^{-\alpha})), \ \text{ as }\ k\to\infty.\end{align}  
We define $W(t,\x)$ in $\mathcal{D}(\A^{\alpha})$ by $\mathrm{P}_l W(t,\x)=W^l(t,\x), $  for all $l\in \mathbb{N}$, so that $W$ is defined on $[0,T]\times \Big(\bigcup\limits_{l\in\mathbb{N}}	 \mathrm{P}_l \mD (\A^{\alpha})\Big)$. 
Besides,  for any $\|\x\|_\alpha,\|\y\|_\alpha \leq r,$ we also get 
$$
\|\mathrm{D}_{\x}\w_m^l(t,\x)-\mathrm{D}_{\x}\w_m^l(t,\y)\|_{-\alpha}\leq C(r)\|\x-\y\|_{\alpha},
$$
and so the function $W$ is also continuously  extended to $ \mD (\A^{\alpha})\times[0,T].$
Now we show that $\mathrm{D}_{\x}\w=W$. For any $m,n\in\mathbb{N}$,  $\|\x\|_{\alpha}\leq r$, $t\in[0,T]$, $\mathbf{h},\mathbf{h}_j\in\mD(\A^{\alpha})$ such that $\mathbf{h}_j\to\mathbf{h}$ in $\mD(\A^{\alpha})$ for some fixed $\alpha$,  we find 
\begin{align}\label{3.50}
&|\la\mathrm{D}_{\x}\w_m(t,\x)-\mathrm{D}_{\x}\w_n(t,\x), \mathbf{h}\ra|\nonumber\\& \leq |\la\mathrm{D}_{\x}\w_m(t,\x), \mathbf{h}-\mathbf{h}_j\ra|+|\la\mathrm{D}_{\x}\w_m(t,\x)-\mathrm{D}_{\x}\w_n(t,\x), \mathbf{h}_j\ra|+|\la\mathrm{D}_{\x}\w_n(t,\x), \mathbf{h}-\mathbf{h}_j\ra|\nonumber\\&\leq \|\A^{-\alpha}\mathrm{D}_{\x}\w_m(t,\x)\|\|\mathbf{h}-\mathbf{h}_j\|_{\alpha}+\|\A^{-\alpha}(\mathrm{D}_{\x}\w_m(t,\x)-\mathrm{D}_{\x}\w_n(t,\x))\|\|\mathbf{h}_j\|_{\alpha}\nonumber\\&\quad +\|\A^{-\alpha}\mathrm{D}_{\x}\w_n(t,\x)\|\|\mathbf{h}-\mathbf{h}_j\|_{\alpha}
\nonumber\\&\leq C(r)\|\mathbf{h}-\mathbf{h}_j\|_{\alpha}+\|\mathrm{D}_{\x}\w_m(t,\x)-\mathrm{D}_{\x}\w_n(t,\x)\|_{-\alpha}\|\mathbf{h}_j\|_{\alpha}\to 0,
\end{align}
using \eqref{3.49}. Hence $\la\mathrm{D}_{\x}\w_m(t,\x)-\mathrm{D}_{\x}\w_n(t,\x), \mathbf{h}\ra$ is a Cauchy sequence in $\mathbb{R}$ and converging to $\la W, \mathbf{h}\ra$. Hence, it coincides with the Gateaux derivative $\la\mathrm{D}_{\x}\w(t,\x),\mathbf{h}\ra$. It shows that 
\begin{align}\label{3.64}
\mathrm{D}_{\x}\w_m(t,\x)\to \mathrm{D}_{\x}\w(t,\x), \ \text{ as } \ k\to\infty \ \text{ in }\ \mathcal{D}(\A^{-\alpha}),
\end{align}
for all $(t,\x)\in[0,T]\times\mD(\A^{\alpha})$. 

Moreover, for any $r>0$ $\|\x\|_{\alpha},\|\y\|_{\alpha}\leq r$, we have 
\begin{align*}
&\|\mathrm{D}_{\x}\w(t,\x)-\mathrm{D}_{\x}\w(t,\y)\|_{-\alpha}\nonumber\\&\leq \|\mathrm{D}_{\x}\w(t,\x)-\mathrm{D}_{\x}\w_{m_k}(t,\x)\|_{-\alpha}+ \|\mathrm{D}_{\x}\w_{m_k}(t,\x)-\mathrm{D}_{\x}\w_{m_k}(t,\y)\|_{-\alpha}\nonumber\\&\quad+\|\mathrm{D}_{\x}\w_{m_k}(t,\y)-\mathrm{D}_{\x}\w(t,\y)\|_{-\alpha}\nonumber\\&\leq C(r)\|\x-\y\|_{\alpha}+\varepsilon,
\end{align*}
where we used \eqref{3.48}-\eqref{3.50}. Since $\varepsilon>0$ is arbitrary, we have 
\begin{align}\label{3.65}
&\|\mathrm{D}_{\x}\w(t,\x)-\mathrm{D}_{\x}\w(t,\y)\|_{-\alpha}\leq  C(r)\|\x-\y\|_{\alpha}.
\end{align}

\noindent\textbf{Step (3).} \emph{Passing limit in the mild form \eqref{2p23} of the HJB equation \eqref{2p21}.}
Now we show that $\w$ is a mild solution of \eqref{349}. For any $k\in\mathbb{N}$, we have 
\begin{align}  \label{3.51}
\w_{m_k}(t,\x)=\mathrm{T}^{m_k}_t\wi f(\x)+\int_0^t\mathrm{T}^{m_k}_{t-s}\widetilde{\mathrm{F}}_{m_k}(\x,\w_{m_k},\mathrm{D}_\x\w_{m_k}(s,\x))\d s+\int_0^t\mathrm{T}^{m_k}_{t-s}\wi g(\x)\d s,
\end{align}
for any $\x\in\mathrm{P}_{m_k}\mH$, $t\in[0,T]$.

By the definition of the semigroup $\mathrm{T}_t^m$, we have 
\begin{align}
(\mathrm{T}^{m_k}_t\widetilde{f})(\x)&=\mE\big[e^{-2\ga\mathcal{Y}_{m_k}(t)}\widetilde{f}(\Y_{m_k}(t,\x))\big] \nonumber\\&= \mE\left[e^{-2\ga\int_0^t\|\Y_{m_k}(r)\|_{\frac{1}{2}}^2\d r}e^{-\ga\|\Y_{m_k}(t,\x)\|^2}\|\Y_{m_k}(t,\x)\|^2 \right] \nonumber\\& \leq \mE\left[\|\Y_{m_k}(t,\x)\|^2 \right]\leq C(T)(1+\|\x\|^{2}),
\end{align}
using Lemma \ref{L3}. Since $\Y_{m_k} \to \Y$ almost surely in $\mathscr{D}([0,T];\mH)\cap\mathrm{L}^2(0,T;\mV),$ by continuity, $e^{-2\ga\int_0^t\|\Y_{m_k}(r)\|_{\frac{1}{2}}^2\d r}\to e^{-2\ga\int_0^t\|\Y(r)\|_{\frac{1}{2}}^2\d r}$ almost surely as $k\to\infty$. 
Using dominated convergence theorem, we have 
\begin{align}\label{3.54}
&|(\mathrm{T}^{m_k}_t\widetilde{f})(\x)-(\mathrm{T}_t\widetilde{f})(\x)| \nonumber\\& =\left|\mE\left[e^{-2\ga\int_0^t\|\Y_{m_k}(r)\|_{\frac{1}{2}}^2\d r}e^{-\ga\|\Y_{m_k}(t,\x)\|^2}\|\Y_{m_k}(t,\x)\|^2 \right]  - \mE\left[e^{-2\ga\int_0^t\|\Y(r)\|_{\frac{1}{2}}^2\d r}e^{-\ga\|\Y(t,\x)\|^2}\|\Y(t,\x)\|^2 \right] \right|\nonumber\\&\to 0, \ \text{ as }\ k\to\infty, 
\end{align}
Hence, we get  $(\mathrm{T}^{m_k}_t\widetilde{f})(\x)\to (\mathrm{T}_t\widetilde{f})(\x)$.  Note that 
\begin{align}
\int_0^t\mathrm{T}^{m_k}_{t-s}\wi g(\x)\d s&=\int_0^t\mE\left[e^{-2\ga\int_0^{t-s}\|\Y_{m_k}(r)\|_{\frac{1}{2}}^2\d r}e^{-\ga\|\Y_{m_k}(t-s,\x)\|^2}\|\Y_{m_k}(t-s,\x)\|_{\frac{1}{2}}^2\right]\d s\nonumber\\&\leq \mathbb{E}\left[\int_0^t\|\Y_{m_k}(s,\x)\|_{\frac{1}{2}}^2\right]\d s\leq C(T)(1+\|\x\|^{2}),
\end{align}
once again using Lemma \ref{L3}. 
Since $\Y_{m_k} \to \Y$ almost surely in $\mathscr{D}([0,T];\mH)\cap\mathrm{L}^2(0,T;\mV),$ we have  \begin{align}\label{3.71}\int_0^t\mathrm{T}^{m_k}_{t-s}\wi g(\x)\d s\to\int_0^t\mathrm{T}_{t-s}\wi g(\x)\d s, \ \text{ as } \ k\to\infty.\end{align}

We estimate the middle term of \eqref{2p23} as follows.
\begin{align}\label{3.72}
&\int_0^t\mathrm{T}^{m_k}_{t-s}\widetilde{\mathrm{F}}_{m_k}(\x,\w_{m_k},\mathrm{D}_\x\w_{m_k}(s,\x))\d s\nonumber\\& = \int_0^t  \mE\left[e^{-2\ga\int_0^{t-s}\|\Y_{m_k}(r)\|_{\frac{1}{2}}^2\d r}\widetilde{\mathrm{F}}_{m_k}(\Y_{m_k}(s),\w_{m_k}(s,\Y_{m_k}(s)),\mathrm{D}_\x\w_{m_k}(s,\Y_{m_k}(s)))\right] \d s\nonumber\\&= \sum_{i=1}^4I^{m_k}_i,
\end{align}
where 
\begin{align*}
I_1^{m_k}&= 2\th \mE\left[\int_0^te^{-2\ga\int_0^{t-s}\|\Y_{m_k}(r)\|_{\frac{1}{2}}^2\d r}\la\A^{-\va}\Y_{m_k}(s),\mathrm{D}_\x\w_{m_k}(s,\Y_{m_k}(s))\ra\d s\right], \\ I_2^{m_k}&= \mE\left[\int_0^te^{-2\ga\int_0^{t-s}\|\Y_{m_k}(r)\|_{\frac{1}{2}}^2\d r}(4
\ga^2\|\A^{-\frac{\va}{2}}\Y_{m_k}(s)\|^2+2\ga\tr(\A^{-\va}))\w_{m_k}(s,\Y_{m_k}(s))\d s\right]\\ I_3^{m_k}&= \mE\left[\int_0^te^{-2\ga\int_0^{t-s}\|\Y_{m_k}(r)\|_{\frac{1}{2}}^2\d r} e^{-\ga\|\Y_{m_k}(s)\|^2}\right.\nonumber\\&\quad\quad\left.\times\mathrm{F}(e^{\ga\|\Y_{m_k}(s)\|^2}(\mathbf{K}^*\mathrm{D}_\x\w_{m_k}(s,\Y_{m_k}(s))+2\ga\mathbf{K}^*\w_{m_k}(s,\Y_{m_k}(s))\Y_{m_k}(s)))\d s\right],\\  I_4^{m_k}&=\mE\left[\int_0^te^{-2\ga\int_0^{t-s}\|\Y_{m_k}(r)\|_{\frac{1}{2}}^2\d r} \right.\nonumber\\&\quad\quad \times  \int_{\mathcal{Z}_{m_k}}\left(e^{\ga \|\Y_{m_k}(s)+\G_{m_k}\|^2-\ga\|\Y_{m_k}(s)\|^2}-1\right)\w_{m_k}(s,\Y_{m_k}(s)+\G_{m_k}(s,z))\mu(\d z)\d s\bigg]\nonumber\\ I_5^{m_k}&=-2\th\mE\left[\int_0^te^{-2\ga\int_0^{t-s}\|\Y_{m_k}(r)\|_{\frac{1}{2}}^2\d r}   \int_{\mathcal{Z}_{m_k}}\la \G_{m_k}(t,z),\w_{m_k}(s.\Y_{m_k}(s))\Y_{m_k}(s)\ra\mu(\d z)\d s\right].
\end{align*}
For simplicity, the limit integrals of $I_i^{m_k},$ $i=1,\ldots,4$ will be denoted as $I_i$, $i=1,\ldots,4$, where $I_i$ stands for the integral in $I_i^{m_k}$ without the subscript $m_k$. Now we prove the convergences of integrals $I_i^{m_k}$, $i=1,\ldots,4$. The boundedness of $I_1^{m_k}$ can be easily obtained using Proposition \ref{prop3.2} as follows:
\begin{align}
|I_1^{m_k}|&\leq 2\th \mE\left[\int_0^te^{-2\ga\int_0^{t-s}\|\Y_{m_k}(r)\|_{\frac{1}{2}}^2\d r}\|\A^{\alpha-\va}\Y_{m_k}(s)\|\|\A^{-\alpha}\mathrm{D}_\x\w_{m_k}(s,\Y_{m_k}(s))\|\d s\right]\no\\&\leq 2\uptheta \|\A^{-\varepsilon}\|_{\mathcal{L}(\mH)}\|\w_{m_k}\|_{\alpha,d,1}\int_0^t\mE\left[e^{-2\ga\int_0^{t-s}\|\Y_{m_k}(r)\|_{\frac{1}{2}}^2\d r}(1+\|\A^{\alpha}\Y_{m_k}(s)\|)^{d+1}\right]\d s\no\\&\leq C(d,T)(1+\|\A^{\alpha}\x\|)^{1+d}.
\end{align}
Let us consider 
\begin{align}\label{3.73}
&\mE\left[e^{-2\ga\mathcal{Y}_{m_k}(\tau)}\la\A^{-\va}\Y_{m_k}(\tau),\mathrm{D}_\x\w_{m_k}(\tau,\Y_{m_k}(\tau))\ra- e^{-2\ga\mathcal{Y}_{m_k}(\tau)}\la\A^{-\va}\Y(\tau),\mathrm{D}_\x\w(\tau,\Y(\tau))\ra\right]\nonumber\\& = \mE\left[ \left(e^{-\ga\mathcal{Y}_{m_k}(\tau)}-e^{-\ga\mathcal{Y}(\tau)}\right)e^{-\ga\mathcal{Y}_{m_k}(\tau)}\la\A^{-\va}\Y_{m_k}(\tau),\mathrm{D}_\x\w_{m_k}(\tau,\Y_{m_k}(\tau))\ra\right]\nonumber\\&\quad +  \mE\left[ e^{-\ga\mathcal{Y}(\tau)}e^{-\ga\mathcal{Y}_{m_k}(\tau)}\la\A^{-\va}(\Y_{m_k}(\tau)-\Y(\tau)),\mathrm{D}_\x\w_{m_k}(\tau,\Y_{m_k}(\tau))\ra\right]\nonumber\\&\quad+  \mE\left[ e^{-\ga\mathcal{Y}(\tau)}e^{-\ga\mathcal{Y}_{m_k}(\tau)}\la\A^{-\va}\Y(\tau),\mathrm{D}_\x\w_{m_k}(\tau,\Y_{m_k}(\tau))-\mathrm{D}_\x\w(\tau,\Y(\tau))\ra\right]\nonumber\\&\quad+  \mE\left[( e^{-\ga\mathcal{Y}_{m_k}(\tau)}-e^{-\ga\mathcal{Y}(\tau)})e^{-\ga\mathcal{Y}(\tau)}\la\A^{-\va}\Y(\tau),\mathrm{D}_\x\w(\tau,\Y(\tau))\ra\right] =:\sum_{i=1}^4J_i. 
\end{align}
We estimate terms $J_i$, $i=1,\ldots,4$ as follows. 
\begin{align}\label{3.74}
J_1&\leq\left\{ \mE\left(e^{-\ga\mathcal{Y}_{m_k}(\tau)}-e^{-\ga\mathcal{Y}(\tau)}\right)^2 \right\}^{1/2}\nonumber\\&\quad\times\left\{\mE\left(e^{-2\ga\mathcal{Y}_{m_k}(\tau)}\|\A^{-\va}\|_{\mathcal{L}(\mH)}^2\|\A^{\alpha}\Y_{m_k}(\tau)\|^2\|\A^{-\alpha}\mathrm{D}_\x\w_{m_k}(\tau,\Y_{m_k}(\tau))\|^2\right)\right\}^{1/2}\nonumber\\&\leq \|\A^{-\va}\|_{\mathcal{L}(\mH)}\|\w_{m_k}\|_{\alpha,d,1}  \left\{ \mE\left(e^{-\ga\mathcal{Y}_{m_k}(\tau)}-e^{-\ga\mathcal{Y}(\tau)}\right)^2 \right\}^{1/2} \left\{\mE\left(e^{-2\ga\mathcal{Y}_{m_k}(\tau)}(1+\|\A^{\alpha}\Y_{m_k}\|)^{2(d+1)}\right)\right\}^{1/2}.
\end{align}
The final term in \eqref{3.74} is bounded by Lemma \ref{lem2.4}. By the almost sure  convergence of $\Y_{m_k} \to \Y$  in $\mathscr{D}([0,T];\mH)\cap\mathrm{L}^2(0,T;\mV),$ $\Y_{m_k}(t)$ and dominated convergence theorem, we have  $\mE\left(e^{-\ga\mathcal{Y}_{m_k}(\tau)}-e^{-\ga\mathcal{Y}(\tau)}\right)^2\to 0$ as $k\to\infty$. 
Next, we show the convergence of $J_2$. 
\begin{align}\label{3.75}
J_2&\leq \|\A^{-\va}\|_{\mathcal{L}(\mH)}  \mE\left[ e^{-\ga\mathcal{Y}(\tau)}e^{-\ga\mathcal{Y}_{m_k}(\tau)}\|\A^{\alpha}(\Y_{m_k}(\tau)-\Y(\tau))\|\|\A^{-\alpha}\mathrm{D}_\x\w_{m_k}(\tau,\Y_{m_k}(\tau))\|\right]\nonumber\\&\leq \|\A^{-\va}\|_{\mathcal{L}(\mH)} \|\w_{m_k}\|_{\alpha,d,1}  \mE\left[ e^{-\ga\mathcal{Y}(\tau)}e^{-\ga\mathcal{Y}_{m_k}(\tau)}\|\A^{\alpha}(\Y_{m_k}(\tau)-\Y(\tau))\|(1+\|\A^{\alpha}\Y_{m_k}\|)^d\right]\nonumber\\&\leq  \|\A^{-\va}\|_{\mathcal{L}(\mH)} \|\w_{m_k}\|_{\alpha,d,1} \left\{ \mE\left[ e^{-\ga\mathcal{Y}(\tau)}e^{-\ga\mathcal{Y}_{m_k}(\tau)}\|\A^{\alpha}(\Y_{m_k}(\tau)-\Y(\tau))\|^2\right]\right\}^{1/2} \nonumber\\&\quad\times\left\{ \mE\left[ e^{-\ga\mathcal{Y}(\tau)}e^{-\ga\mathcal{Y}_{m_k}(\tau)}(1+\|\A^{\alpha}\Y_{m_k}\|)^{2d}\right]\right\}^{1/2}.
\end{align}
The final term in \eqref{3.75} is bounded. Since $\Y_{m_k} \to \Y$ almost surely in $\mathrm{L}^2(0,T;\mV),$ we have $\Y_{m_k} (\tau)\to \Y(\tau)$ in $\mV$ and hence in $\mathcal{D}(\A^{\alpha}),$ $\alpha<\frac{1}{2}$,  for a. e. $\tau\in[0,T]$. By dominated convergence theorem, $\mE\left[ e^{-\ga\mathcal{Y}(\tau)}e^{-\ga\mathcal{Y}_{m_k}(\tau)}\|\A^{\alpha}(\Y_{m_k}(\tau)-\Y(\tau))\|^2\right]\to  0$ as $k\to\infty$.  The term $J_4\to 0$ as $k\to\infty$, by an argument similar to $J_1$. Finally, we estimate $J_3$ as 
\begin{align}\label{3.76}
J_3&=  \mE\left[ e^{-\ga\mathcal{Y}(\tau)}e^{-\ga\mathcal{Y}_{m_k}(\tau)}\la\A^{-\va}\Y(\tau),\mathrm{D}_\x\w_{m_k}(\tau,\Y_{m_k}(\tau))-\mathrm{D}_\x\w_{m_k}(\tau,\Y(\tau))\ra\right]\no\\&\quad+  \mE\left[ e^{-\ga\mathcal{Y}(\tau)}e^{-\ga\mathcal{Y}_{m_k}(\tau)}\la\A^{-\va}\Y(\tau),\mathrm{D}_\x\w_{m_k}(\tau,\Y(\tau))-\mathrm{D}_\x\w(\tau,\Y(\tau))\ra\right]\nonumber\\&\leq  \|\A^{-\va}\|_{\mathcal{L}(\mH)}   \mE\left[ e^{-\ga\mathcal{Y}(\tau)}e^{-\ga\mathcal{Y}_{m_k}(\tau)}\|\A^{\alpha}\Y(\tau)\|\|\A^{-\alpha}(\mathrm{D}_\x\w_{m_k}(\tau,\Y_{m_k}(\tau))-\mathrm{D}_\x\w_{m_k}(\tau,\Y(\tau)))\|\right]\nonumber\\&\quad + \|\A^{-\va}\|_{\mathcal{L}(\mH)}  \mE\left[ e^{-\ga\mathcal{Y}(\tau)}e^{-\ga\mathcal{Y}_{m_k}(\tau)}\|\A^{\alpha}\Y(\tau)\|\|\A^{-\alpha}(\mathrm{D}_\x\w_{m_k}(\tau,\Y(\tau))-\mathrm{D}_\x\w(\tau,\Y(\tau)))\|\right]\nonumber\\&\leq C \|\A^{-\va}\|_{\mathcal{L}(\mH)} \left\{  \mE\left[ e^{-\ga\mathcal{Y}(\tau)}e^{-\ga\mathcal{Y}_{m_k}(\tau)}\|\A^{\alpha}\Y(\tau)\|^2\right]\right\}^{1/2}\Big(\left\{ \mE\left[ e^{-\ga\mathcal{Y}(\tau)}e^{-\ga\mathcal{Y}_{m_k}(\tau)}\|\Y_{m_k}-\Y\|_{\alpha}^2\right]\right\}^{1/2} \nonumber\\&\quad+ \left\{  \mE\left[ e^{-\ga\mathcal{Y}(\tau)}e^{-\ga\mathcal{Y}_{m_k}(\tau)}\|\A^{-\alpha}(\mathrm{D}_\x\w_{m_k}(\tau,\Y(\tau))-\mathrm{D}_\x\w(\tau,\Y(\tau)))\|^2\right]\right\}^{1/2}\Big).
\end{align}
By the arguments similar to the convergence of $J_2$, the first term in the final inequality of \eqref{3.76} tends to zero as $k\to\infty$. By an application of dominated convergence theorem and \eqref{3.64}, the last term in \eqref{3.76} also tends to zero a.s. and a.e. $\tau\in[0,T]$ as $k\to\infty$. Combining \eqref{3.73}-\eqref{3.76}, we arrive at $I_1^{m_k}\to I_1$ as $k\to\infty$. 

Clearly, the terms in $I_2^{m_k}$ is bounded by using energy estimates and Proposition \ref{prop3.2}. Once again by using dominated convergence theorem, one can see that $I_2^{m_k}$ converges to $I_2$. Indeed, for the second term, one can use the fact that $\Y_{m_k}\to\Y$, a.s. and a.e. $t\in[0,T]$ in $\mathcal{D}(\A^{\alpha})$,  the continuity of $\w_{m_k}$ on $[0,T]\times\mathcal{D}(\A^{\alpha})$ and \eqref{360}. 

By the definition of $\mathrm{F}(\cdot)$ in \eqref{2p15} and Assumption \ref{ass2.1}$-(H_2)$, we have the following bound for $I_3^{m_k}$. 
\begin{align}
|I_3^{m_k}|&\leq R \mE\left[\int_0^te^{-2\ga\mathcal{Y}_{m_k}(t-s)}( \|\mathbf{K}^*\mathrm{D}_\x\w_{m_k}(s,\Y_{m_k}(s))\|+2\ga\|\mathbf{K}^*\w_{m_k}(s,\Y_{m_k}(s))\Y_{m_k}(s)\|)\d s\right]\no\\&\leq C(R, \ga)\mE\left[\int_0^te^{-2\ga\mathcal{Y}_{m_k}(t-s)}( \|\A^{-\alpha}\mathrm{D}_\x\w_{m_k}(s,\Y_{m_k}(s))\|\no\right.\\
&\left.\quad\quad+|\w_{m_k}(s,\Y_{m_k}(s))|\|\A^{-\alpha}\Y_{m_k}(s)\|)\d s\right]\no\\
&\leq C(R, \ga)\|\w_{m_k}\|_{\alpha,d,1}\mE\left[\int_0^te^{-2\ga\mathcal{Y}_{m_k}(t-s)}( (1+\|\A^{\alpha}\Y_{m_k}(s))^d+\|\A^{-\alpha}\|_{\mathcal{L}(\mH)}\|\Y_{m_k}(s)\|)\d s\right]<+\infty.
\end{align}
Therefore, by dominated convergence theorem, one can show that the integral $I_3^{m_k}$ converges to $I_3$ as $k\to\infty$, by applying the similar arguments used for $I_1^{m_k}$. 

Let us now show the boundedness of the integral $I_4^{m_k}$. 
\begin{align}
|I_4^{m_k}|&=\mE\left[\int_0^t  \int_{\mathcal{Z}_{m_k}}e^{\ga \|\G_{m_k}(s,z)\|^2+2\ga\la\Y_{m_k}(s),\G_{m_k}(s,z)\ra}|\w_{m_k}(s,\Y_{m_k}(s)+\G_{m_k}(s,z))|\mu(\d z)\d s\right]\nonumber\\&\leq \|\w_{m_k}\|_{0,0,0}\int_0^t\mE\left[ e^{\ga\|\Y(s)\|^2}\right] \int_{\mathcal{Z}}e^{2\ga \|\G(s,z)\|^2}\mu(\d z)\d s\leq C,
\end{align}
where we used Assumption \ref{ass2.1} and Proposition 3.3. Note that 
\begin{align}
&\int_{\mathcal{Z}_{m_k}}(e^{\ga \|\Y_{m_k}(s)+\G_{m_k}(s,z)\|^2-\ga\|\Y_{m_k}(s)\|^2}-1)\w_{m_k}(s,\Y_{m_k}(s)+\G_{m_k}(s,z))\mu(\d z) \nonumber\\& =   \int_{\mathcal{Z}}(e^{\ga \|\Y_{m_k}(s)+\G_{m_k}(s,z)\|^2-\ga\|\Y_{m_k}(s)\|^2}-1)\w_{m_k}(s,\Y_{m_k}(s)+\G_{m_k}(s,z))\mu(\d z)\nonumber\\&\quad+  \int_{\mathcal{Z}_{m_k}^c}(e^{\ga \|\Y_{m_k}(s)+\G_{m_k}(s,z)\|^2-\ga\|\Y_{m_k}(s)\|^2}-1)\w_{m_k}(s,\Y_{m_k}(s)+\G_{m_k}(s,z))\mu(\d z)\nonumber\\&=:J_5^{m_k}+J_6^{m_k}.
\end{align}
Since, for almost all $t\in[0,T]$, $\Y_{m_k}\to\Y$, $\mathbb{P}$-a.s. and $\G_{m_k}\to\G$, a.e. $z\in\mZ$ in $\mathcal{D}(\A^{\alpha})$,  the continuity of $\w_{m_k}$ on $[0,T]\times\mathcal{D}(\A^{\alpha})$ and the convergence  \eqref{360} show that $\w_{m_k}(s,\Y_{m_k}(s)+\G_{m_k}(s,z))$  converges to $ \w(s,\Y(s)+\G(s,z))$, $\mathbb{P}$-a.s. and a,e. $t\in[0,T]$ in $\mathbb{R}$. Therefore, by dominated convergence theorem, 
\begin{align}
J_5^{m_k}\to  \int_{\mathcal{Z}}(e^{\ga \|\Y(s)+\G(s,z)\|^2-\ga\|\Y(s)\|^2}-1)\w(s,\Y(s)+\G(s,z))\mu(\d z),
\end{align}
as $k\to\infty$, $\mathbb{P}$-a.s. and a.e. $t\in[0,T]$. Now, we consider 
\begin{align}
|J_6^{m_k}|&\leq  \int_{\mathcal{Z}_{m_k}^c}e^{\ga \|\G_{m_k}(s,z)\|^2+2\ga\la\Y_{m_k}(s),\G_{m_k}(s,z) \ra}|\w_{m_k}(s,\Y_{m_k}(s)+\G_{m_k}(s,z))|\mu(\d z)\nonumber\\&\leq \|\w_{m_k}\|_{0,0,0} e^{\ga\|\Y(s)\|^2} \int_{\mathcal{Z}_{m_k}^c}e^{2\ga \|\G(s,z)\|^2}\mu(\d z).
\end{align}
Therefore $I^4_{m_k}\to I_4$, $\mathbb{P}$-a.s., as $k\to\infty$, since $\mE\left[ e^{\ga\|\Y(s)\|^2}\right] \leq C, $ and $\int_0^t\int_{\mathcal{Z}_{m_k}^c}e^{2\ga \|\G(s,z)\|^2}\mu(\d z)\d s\to 0$, as $k\to\infty$. Using similar arguments, one can easily show that $I_5^{m_k}\to I_5$, as $k\to\infty$. Passing limit as $k\to\infty$ in  \eqref{3.72} and  combining \eqref{3.54}, \eqref{3.71}, we obtain that $\w$ as a mild solution of \eqref{349}. This completes the proof of Proposition \ref{main0}. 
\end{proof}
Now, the proof of Theorem \ref{main1} is a direct consequence of the Proposition \ref{main0}. Indeed, by  the transformation $\v(t,\x)=e^{\th \|\x\|^2}\w(t,\x),$  one can see that $\v(\cdot,\cdot)$ is a mild solution of the HJB equation \eqref{2p16}.

\begin{Rem}\label{rem5.1}
Since $v_m(t,\x)=e^{\th\|\x\|^2}w_m(t,\x)$, by \eqref{360} and \eqref{3.64}, we get that as $k\to\infty$
\begin{align}\label{534}
&\v_{m}(t,\x)\to \v(t,\x), \ \text{ in }\ \mR,\
\mathrm{D}_{\x}\w_m(t,\x)\to \mathrm{D}_{\x}\w(t,\x), \ \text{ in }\ \mathcal{D}(\A^{-\alpha}),
\end{align}
for all $(t,\x)\in[0,T]\times\mD(\A^{\alpha})$ and $\|\x\|_{\alpha}\leq r$. 

Since $v(t,\x)=e^{\th\|\x\|^2}w(t,\x)$, 
by the continuity estimates for the transformed HJB equation \eqref{315} proved in step 1 and \eqref{3.65}, one can easily show that the solution of the HJB equation \eqref{2p16} is locally Lipschitz and satisfies the following estimates: 
\begin{align}
|v(t,\x)-\v(t,\y)|&\leq C(r)\|\x-\y\|_{\alpha}, \label{5.34}\\ 
\|\mathrm{D}_{\x}\v(t,\x)-\mathrm{D}_{\x}\v(t,\y)\|_{-\alpha}&\leq  C(r)\|\x-\y\|_{\alpha},\label{5.35}
\end{align}
for any $\|\x\|_{\alpha},\|\y\|_{\alpha}\leq r$, $r>0$. 
\end{Rem}

\section{Proof of Theorem  \ref{thm3.3}}\label{sec6}

In order to prove Theorem \ref{thm3.3}, we need the following results on the approximated cost functional. 

\begin{Lem}\label{lem6.1}
The approximated solution  $\X_m(\cdot)$ of (\ref{2F1}) converges almost surely to the solution $\X(\cdot)$ of (\ref{2p12}) in $\mathscr{D}([0,T];\mathcal{D}(\A^{\alpha}))\cap\mathrm{L}^2(0,T;\mathcal{D}(\A^{\alpha+\frac{1}{2}}))$. Moreover, the approximated cost functional $\mathcal{J}_m(0,T;\x,\U_m)$ defined in \eqref{2F2} converges  to $\mathcal{J}(0,T;\x,\U)$ defined in \eqref{2p13b}.
\end{Lem}
\begin{proof}
It is clear from (\ref{c1}) and the rest of the arguments in Proposition \ref{prop3.3} that in order  to prove the almost sure convergence of $\X_m(\cdot)$ in the given topology, it is sufficient to prove that 
\begin{align}\label{3.127}
\mE\left[\int_0^T\la\U_m(t),\K^*\X_m(t)\ra\d t\right]\to\mE\left[\int_0^T\la\U(t),\K^*\X(t)\ra\d t\right], \text{ as } m\to\infty.
\end{align}
Since $\X_m\rightharpoonup\X$ in $\mathrm{L}^2(\Omega;\mathrm{L}^2(0,T;\mV))$ and $\U_m\to\U$ in $\mathrm{L}^2(\Omega;\mathrm{L}^2(0,T;\mH))$, the convergence in (\ref{3.127}) follows. Therefore, $\X_m(\cdot)$ converges almost surely to $\X(\cdot)$  in $\mathrm{L}^2(0,T;\mV)\cap\mathscr{D}([0,T];\mH)$. 

To prove $\X\in\mathscr{D}([0,T];\mathcal{D}(\A^{\alpha}))$, $\mathbb{P}$-a.s., we note that in view of the identity \eqref{4.61}, it is enough to estimate the following: 
\begin{align}
|\la\K\U,\A^{2\alpha}\X\ra|\leq \|\U\|\|\K^*\A^{2\alpha}\X\|\leq C_{\K}\|\U\|\|\X\|_{\alpha}\leq \frac{1}{2}\|\X\|_{\alpha+\frac{1}{2}}+C_{\K}\|\U\|^2. 
\end{align}
Using this in \eqref{4.61} and noting that $\int_0^T\|\U(t)\|^2\d t<+\infty$, $\mathbb{P}$-a.s., we get an estimate similar to \eqref{4.63}.

Now,  in view of \eqref{4.55}, we note that 
\begin{align}
&	|\la\K(\U_m-\mathrm{P}_m\U),\A^{2\alpha}(\X_m-\mathrm{P}_m\X)\ra|\nonumber\\& \leq \frac{1}{2}\|\X_m-\mathrm{P}_m\X\|_{\alpha+\frac{1}{2}}^2+C_{\K}\left(\|\U_m-\U\|^2+\|\mathrm{P}_m-\I\|^2\|\U\|^2\right). 
\end{align}
Arguing as in \eqref{4.69} and using the fact that $\U_m\to \U$ in $\mathrm{L}^2(\Omega;\mathrm{L}^2(0,T;\mH))$, one can conclude   the almost sure convergence  in $\mathscr{D}([0,T];\mathcal{D}(\A^{\alpha}))\cap\mathrm{L}^2(0,T;\mathcal{D}(\A^{\alpha+\frac{1}{2}}))$,

Since $\|\U_m(\cdot,t)\|\leq R$, $\mathbb{P}$-a.s., $t\in[0,T]$,  the definition of $\mathcal{J}_m$ in \eqref{2F2} and the energy estimate in \eqref{2p31}, it is clear that 
\begin{align}
\mathcal{J}_m(0,T;\x,\U_m)\leq C(T,R)(1+\|\x\|^2).
\end{align}
By using the respective convergences of $\X_m$, $\U_m$ and dominated convergence theorem, one can prove that $\mathcal{J}_m(0,T;\x,\U_m)$ converges to $\mathcal{J}(0,T;\x,\U)$.
\end{proof}

Using the regularity of solutions of the HJB equation \eqref{2p16} proved in Theorem \ref{main1}, we derive an identity satisfied by the cost functional in terms this solution. 

\begin{Lem}
For $m\in\mathbb{N}$, let $\X_m(\cdot)$ be the solution of (\ref{2F1}) and $\U_m\in \mathrm{L}^2(\Omega;\mathrm{L}^2(0,T;\mathrm{P}_m\mH))$. For any $\x\in\mathrm{P}_m\mH$, the following identity holds for the approximated cost functional:
\begin{align}\label{65}
\mathcal{J}_m(0,T;\x,\U_m)&=v_m(T,\x)+\frac{1}{2}\mE\left[\int_0^T\Big(\|\U_m(t)+\K^*\mathrm{D}_{\x}v_m(T-t,\X_m(t))\|^2
\nonumber\right.\\
&\left.\quad\quad-\chi(\|\K^*\mathrm{D}_{\x} v_m(T-t,\X_m(t))\|-R)\Big)\d t\right], 
\end{align}
where $\chi(\cdot)$ is the function defined in Theorem \ref{thm3.3}. 

Moreover, $\mathcal{J}_m(0,T;\x,\U_m)$ converges to  $\mathcal{J}(0,T;\x,\U)$ defined in \eqref{3.122a}.
\end{Lem}
\begin{proof}
Let us apply the finite dimensional It\^o formula to the process $v_m(t-s,\X_m(s))$ to obtain 
\begin{align}\label{3.124}
\d v_m(t-s,\X_m(s))&=\Big[-\mathrm{D}_sv_m(t-s,\X_m(s))+\mathscr{L}_xv_m(t-s,\X_m(s))\no\\&\qquad +\la\K^*\mathrm{D}_{\x}v_m(t-s,\X_m(s)),\U_m(s)\ra\Big]\d s+\d\M^m_{s}\no\\&= \Big[-\mathrm{F}_m(\K^* \mathrm{D}_\x \v_m(t-s,\X_m(s)))- g_m(\X_m(s))\no\\ &\qquad+\la\K^*\mathrm{D}_{\x}v_m(t-s,\X_m(s)),\U_m(s)\ra\Big]\d s +\d\M^m_{s},
\end{align}
where 
\begin{align*}
\d\M^m_{s}&=\la\mathrm{D}_{\x}v_m(t-s,\X_m(s)),\A^{-\frac{\var}{2}}\d\W_m(s)\ra \no\\&\quad+\int_{{\mZ}_m}\left[v_m(t-s,\X_m(s)+\G_m(s,z))-v_m(t-s,\X_m(s))\right]\wi{\uppi}(\d s,\d z).
\end{align*}
Setting $t=T$, integrating from $0$ to $T$, and taking expectation, we get 
\begin{align}\label{3.125}
\mE\Big[v_m(0,\X_m(T))\Big]&=v_m(T,\x)+\mE\left[\int_0^T\la\K^*\mathrm{D}_{\x}v_m(T-s,\X_m(s)),\U_m(s)\ra\d s\right]\no\\&\quad -\mE\left[\int_0^T\Big(\mathrm{F}_m(\K^* \mathrm{D}_\x \v_m(T-s,\X_m(s)))+ g_m(\X_m(s))\Big)\d s\right].
\end{align} 
Using $v_m(0,\X_m(T))=\|X_m(T)\|^2$   and rearranging the terms in (\ref{3.125}), we arrive at 
\begin{align}\label{3.125}
\mE\left[\|X_m(T)\|^2 +\int_0^Tg_m(\X_m(t))dt\right]&=v_m(T,\x)+\mE\left[\int_0^T\la\K^*\mathrm{D}_{\x}v_m(T-s,\X_m(s)),\U_m(s)\ra\d s\right]\no\\&\quad -\mE\left[\int_0^T\mathrm{F}_m(\K^* \mathrm{D}_\x \v_m(T-s,\X_m(s)))\d s\right].
\end{align} 
Adding $\frac{1}{2}\|\U_m(t)\|^2$ on both sides and using the definition of approximated cost functional $\mathcal{J}_m,$  one can obtain
\begin{align}\label{3.126}
\mathcal{J}_m(0,T;x,\U_m)&= \mE\left[\int_0^T\left(g_m(\X_m(t))+\frac{1}{2}\|\U_m(t)\|^2\right)\d t+\|\X_m(T)\|^2\right]\no\\&=v_m(T,\x)+\frac{1}{2}\mE\left[\int_0^T\|\U_m(t)+\K^*\mathrm{D}_{\x}v_m(T-t,\X_m(t))\|^2\d t\right]\\&\quad -\mE\left[\int_0^T\Big(\F_m(\K^*\mathrm{D}_{\x}v_m(T-t,\X_m(t)))+\frac{1}{2}\|\K^*\mathrm{D}_{\x}v_m(T-t,\X_m(t))\|^2\Big)\d t\right].\no
\end{align}
Using the definition of $\F_m(\cdot)$, the final integral in (\ref{3.126}) can be written as 
\begin{align*}
\frac{1}{2}\mE\left[\int_0^T\chi(\|\K^*\mathrm{D}_{\x} v_m(T-t,\X(t))\|-R)\d t\right],
\end{align*}
which leads to the identity \eqref{65}. 

To prove the convergence of the cost functional, we proceed as follows: 
\begin{align}
&\|\U_m(t)+\K^*\mathrm{D}_{\x}v_m(T-t,\X_m(t))\|\leq \|\U(t)\|+\|\K^*\mathrm{D}_{\x}v_m(T-t,\X_m(t))\|<+\infty,
\end{align}
$\mathbb{P}$-a.s., since 	by Lemma \ref{lem6.1}, we have 
\begin{align}
&	\|\K^*\mathrm{D}_{\x}v_m(T-t,\X_m(t))\|\nonumber\\&\leq C_{\K}\|\mathrm{D}_{\x}v_m(T-t,\X_m(t))\|_{-\alpha} \nonumber\\&\leq C(\K,\th)e^{\th\|\X_m(t)\|^2}\Big(\|\X_m(t)\|_{-\alpha}|w_m(T-t,\X_m(t))| +\|\mathrm{D}_{\x}w_m(T-t,\X_m(t))\|_{-\alpha}\Big)\nonumber\\& \leq C(\K,\th) \exp\left({\th\sup\limits_{t\in[0,T]}\|\X(t)\|^2}\right)\left[\sup_{t\in[0,T]}\|\X(t)\|\|w\|_{0,0,0}+\sup_{t\in[0,T]}(1+\|\X(t)\|_{\alpha})^{d}\|\w\|_{\alpha,d,1}\right]\no\\&<+\infty,\  \mathbb{P}\text{-a.s.}
\end{align}
It is clear that the second integrand $\chi(\|\K^*\mathrm{D}_{\x} v_m(T-t,\X_m(t))\|-R)$ in \eqref{65} is also $\mathbb{P}$-a.s. bounded. We further note that 
\begin{align}\label{610}
&	\|\K^*\mathrm{D}_{\x}v_m(T-t,\X_m(t))-\K^*\mathrm{D}_{\x}v(T-t,\X(t))\| \\&\leq 	\|\mathrm{D}_{\x}v_m(T-t,\X_m(t))-\mathrm{D}_{\x}v_m(T-t,\X(t))\| _{-\alpha}+ \|\mathrm{D}_{\x}v_m(T-t,\X(t))-\mathrm{D}_{\x}v(T-t,\X(t))\| _{-\alpha}.\no
\end{align}
By using \eqref{5.35} and the almost sure convergence of $\X_m\to\X$ in $\mathcal{D}(\A^{\alpha})$, a.e. $t\in[0,T]$, the first term in \eqref{610} goes to zero as $m\to\infty$ and so does the second term by \eqref{534}. Therefore, passing $m\to\infty$ in \eqref{65} and applying dominated convergence theorem, we conclude that $\mathcal{J}_m(0,T;\x,\U_m)$ converges to  $\mathcal{J}(0,T;\x,\U)$ defined in \eqref{3.122a}. 
\end{proof}

\begin{proof}[Proof of Theorem \ref{thm3.3}] 
It can be easily shown that the following approximated closed loop equation  has a unique solution $\wi\X_m(\cdot)$:
\begin{equation} 
\left\{\begin{aligned}
\d\wi \X_m(t)&=[-\A\wi \X_m(t)+\B(\wi \X_m(t))+{\mathcal G}(\K^*\mathrm{D}_{\x}v(T-t,\wi \X_m(t)))]\d t+\A^{-\frac{\var}{2}}\d\W_m(t)\\
&\quad+ \int_{\mZ_m}\G_m(t,z) \wi{\uppi}(\d t,\d z), \;t\in (0,T), \\
\wi  \X_m(0)&=\mathrm{P}_m\x.
\end{aligned}
\right.
\end{equation}
From the definition of $\mathcal{G}$ in (\ref{2p17}), we know that $$\|\mathcal{G}(\p)\|\leq R,\ \text{ for all }\ \p\in\mH.$$ Arguing similarly as in Proposition \ref{prop3.3}, one can obtain that $\wi\X_m(\cdot)$ is almost surely bounded in $\mathscr{D}([0,T];\mH)\cap\mathrm{L}^2(0,T;\mV).$  Then, there exists a subsequence converging to the solution $\wi\X(\cdot)$ of the closed loop equation (\ref{2p18}). If we show that the closed loop equation (\ref{2p18}) has at most one solution, the above convergence along the subsequence imply that the entire sequence converges.

Now, we prove the uniqueness of the closed loop equation (\ref{2p18}). Let $\X_1(\cdot)$ and $\X_2(\cdot)$ be two solution of (\ref{2p18}) and $\wi\X(\cdot)=\wi\X_1(\cdot)-\wi\X_2(\cdot)$. 	Then, we have 
\begin{equation} \label{3.131}
\frac{\d}{\d t}\wi \X(t)=-\A\wi \X(t)+\B(\wi \X_1(t))-\B(\wi \X_2(t))+{\mathcal G}(\K^*\mathrm{D}_{\x}v(T-t,\wi \X_1(t)))-{\mathcal G}(\K^*\mathrm{D}_{\x}v(T-t,\wi \X_2(t))).
\end{equation} 
Taking inner product with $\A^{2\alpha}\wi\X(\cdot)$ in (\ref{3.131}), we get
\begin{align}\label{3.132}
\frac{1}{2}\frac{\d}{\d t}\|\wi\X(t)\|^2_{\alpha}+\|\wi\X(t)\|_{\alpha+\frac{1}{2}}^2&=\la\B(\wi \X_1(t))-\B(\wi \X_2(t)),\A^{2\alpha}\wi\X(t)\ra\no\\&\quad+\la{\mathcal G}(\K^*\mathrm{D}_{\x}v(T-t,\wi \X_1(t)))-{\mathcal G}(\K^*\mathrm{D}_{\x}v(T-t,\wi \X_2(t))),\A^{2\alpha}\wi\X(t)\ra\no\\&:=I_1+I_2.
\end{align}
The term $I_1$ can be written as 
\begin{align}
I_1=\la\B(\wi \X_1-\wi \X_2,\wi \X_1),\A^{2\alpha}\wi\X\ra+\la\B(\wi \X_2,\wi \X_1-\wi \X_2),\A^{2\alpha}\wi\X\ra.
\end{align}
Using the trilinear estimate \eqref{ne} and Young's inequality, we estimate $\la\B(\wi \X_1-\wi \X_2,\wi \X_1),\A^{2\alpha}\wi\X\ra$ as 
\begin{align}
|\la\B(\wi \X_1-\wi \X_2,\wi \X_1),\A^{2\alpha}\wi\X\ra|\leq C \|\wi\X\|_{\alpha}\|\wi\X_1\|_{\frac{1}{2}}\|\wi\X\|_{\alpha+\frac{1}{2}}\leq \frac{1}{4}\|\wi\X\|_{\alpha+\frac{1}{2}}^2+C\|\wi\X_1\|_{\frac{1}{2}}^2\|\wi\X\|_{\alpha}^2.
\end{align}
Once again using the trilinear estimate \eqref{ne}  and interpolation inequality \eqref{II}, we estimate $\la\B(\wi \X_2,\wi \X_1-\wi \X_2),\A^{2\alpha}\wi\X\ra$ as 
\begin{align}
|\la\B(\wi \X_2,\wi \X_1-\wi \X_2),\A^{2\alpha}\wi\X\ra|&\leq \|\wi\X_2\|_{\alpha}\|\wi\X\|_{\frac{1}{2}}\|\wi\X\|_{\alpha+\frac{1}{2}}\leq \|\wi\X_2\|_{\alpha}\|\wi\X\|_{\alpha}^{2\alpha}\|\wi\X\|_{\alpha+\frac{1}{2}}^{2(1-\alpha)}\nonumber\\&\leq \frac{1}{4}\|\wi\X\|_{\alpha+\frac{1}{2}}+C\|\wi\X_2\|_{\alpha}^{\frac{1}{\alpha}}\|\wi\X\|_{\alpha}^2. 
\end{align}
Thus it is immediate that 
\begin{align}
|I_1|\leq \frac{1}{2}\|\wi\X\|_{\alpha+\frac{1}{2}}+C\left(\|\wi\X_1\|_{\frac{1}{2}}^2+\|\wi\X_2\|_{\alpha}^{\frac{1}{\alpha}}\right)\|\wi\X\|_{\alpha}^2.
\end{align}
From the definition of $\mathcal{G}(\cdot)$ and Assumption \ref{ass2.1}-($H_4$), it is clear that
\begin{align}\label{3.134}
|I_2|&\leq  \|{\mathcal G}(\K^*\mathrm{D}_{\x}v(T-t,\wi \X_1))-{\mathcal G}(\K^*\mathrm{D}_{\x}v(T-t,\wi \X_2))\|_{\alpha-\frac{1}{2}} \|\wi\X\|_{\alpha+\frac{1}{2}}\no\\&\leq  \|\K^*\mathrm{D}_{\x}v(T-t,\wi \X_1)-\K^*\mathrm{D}_{\x}v(T-t,\wi \X_2)\|_{\alpha-\frac{1}{2}}\|\wi\X\|_{\alpha+\frac{1}{2}}\no\\&\leq  C_{\K} \|\mathrm{D}_{\x}v(T-t,\wi \X_1)-\mathrm{D}_{\x}v(T-t,\wi \X_2)\|_{-\alpha}\|\wi\X\|_{\alpha+\frac{1}{2}}\no\\&\leq C_{\K}(\|\X_1\|_{\alpha},\|\X_2\|_{\alpha}) \|\wi\X\|_{\alpha} \|\wi\X\|_{\alpha+\frac{1}{2}}\no\\&\leq  \frac{1}{2}\|\wi\X\|_{\alpha+\frac{1}{2}}^2+C_{\K}(\|\X_1\|_{\alpha},\|\X_2\|_{\alpha})  \|\wi\X\|_{\alpha}^2.
\end{align}
for $\|\K^*\mathrm{D}_{\x}v(T-t,\wi \X_1)\|\leq R$ and $\|\K^*\mathrm{D}_{\x}v(T-t,\wi \X_2)\|\leq R$.  For $\|\K^*\mathrm{D}_{\x}v(T-t,\wi \X_1)\|>  R$ and $\|\K^*\mathrm{D}_{\x}v(T-t,\wi \X_2)\|> R$, we have 
\begin{align}\label{3.134a}
|I_2|&\leq \left\|-R\frac{\K^*\mathrm{D}_{\x}v(T-t,\wi \X_1)}{\|\K^*\mathrm{D}_{\x}v(T-t,\wi \X_1)\|}+R\frac{\K^*\mathrm{D}_{\x}v(T-t,\wi \X_2)}{\|\K^*\mathrm{D}_{\x}v(T-t,\wi \X_2)\|}\right\|_{\alpha-\frac{1}{2}}\|\wi\X\|_{\alpha+\frac{1}{2}}\no\\&\leq \frac{R}{\|\K^*\mathrm{D}_{\x}v(T-t,\wi \X_1)\|\|\K^*\mathrm{D}_{\x}v(T-t,\wi \X_2)\|}\nonumber\\&\qquad\times\left\|(\|\K^*\mathrm{D}_{\x}v(T-t,\wi \X_1)\|-\|\K^*\mathrm{D}_{\x}v(T-t,\wi \X_2)\|)\K^*\mathrm{D}_{\x}v(T-t,\wi \X_1)\right.\no\\&\left.\quad +\|\K^*\mathrm{D}_{\x}v(T-t,\wi \X_1)\|\left(\K^*\mathrm{D}_{\x}v(T-t,\wi \X_2)-\K^*\mathrm{D}_{\x}v(T-t,\wi \X_1)\right)\right\|_{\alpha-\frac{1}{2}}\|\wi\X\|_{\alpha+\frac{1}{2}}\no\\&\leq \frac{2R}{\|\K^*\mathrm{D}_{\x}v(T-t,\wi \X_2)\|}\|\mathrm{D}_{\x}v(T-t,\wi \X_1)-\mathrm{D}_{\x}v(T-t,\wi \X_2)\|_{-\alpha}\|\wi\X\|_{\alpha+\frac{1}{2}}\no\\&\leq  \frac{1}{2}\|\wi\X\|_{\alpha+\frac{1}{2}}^2+C_{\K}(\|\X_1\|_{\alpha},\|\X_2\|_{\alpha})  \|\wi\X\|_{\alpha}^2.
\end{align}
using \eqref{3.134}. Now for $\|\K^*\mathrm{D}_{\x}v(T-t,\wi \X_1)\|>  R$ and $\|\K^*\mathrm{D}_{\x}v(T-t,\wi \X_2)\|\leq  R$, we obtain 
\begin{align}\label{3.134b}
|I_2|&\leq \left\|-R\frac{\K^*\mathrm{D}_{\x}v(T-t,\wi \X_1)}{\|\K^*\mathrm{D}_{\x}v(T-t,\wi \X_1)\|}+\K^*\mathrm{D}_{\x}v(T-t,\wi \X_2)\right\|_{\al-\f2}\|\wi\X\|_{\al+\f2}\no\\&\leq \frac{1}{\|\K^*\mathrm{D}_{\x}v(T-t,\wi \X_1)\|}\left\| \|\K^*\mathrm{D}_{\x}v(T-t,\wi \X_1)\|\left(\K^*\mathrm{D}_{\x}v(T-t,\wi \X_2)-\K^*\mathrm{D}_{\x}v(T-t,\wi \X_1)\right)\right.\no\\&\left.\quad +\K^*\mathrm{D}_{\x}v(T-t,\wi \X_1)\left(\|\K^*\mathrm{D}_{\x}v(T-t,\wi \X_1)\|-R\right)\right\|_{\al-\f2} \|\wi\X\|_{\al+\f2}\no\\&\leq \left(\|\K^*\mathrm{D}_{\x}v(T-t,\wi \X_1)-\K^*\mathrm{D}_{\x}v(T-t,\wi \X_2)\|_{\al-\f2}+\|\K^*\mathrm{D}_{\x}v(T-t,\wi \X_1)\|-R\right)\|\wi\X\|_{\al+\f2}\no\\&\leq  \frac{1}{2}\|\wi\X\|_{\alpha+\frac{1}{2}}^2+C_{\K}(\|\X_1\|_{\alpha},\|\X_2\|_{\alpha})  \|\wi\X\|_{\alpha}^2.
\end{align}
using \eqref{3.134}. An estimate similar to \eqref{3.134b} can be obtained for $\|\K^*\mathrm{D}_{\x}v(T-t,\wi \X_1)\|\leq  R$ and $\|\K^*\mathrm{D}_{\x}v(T-t,\wi \X_2)\|>  R$. Now the equation (\ref{3.132}) can be estimated by the Gr\"onwall inequality as follows: 
\begin{align*}
\|\wi\X(t)\|^2_{\alpha}\leq \|\wi\X(0)\|^2_{\alpha}\exp\left(\int_0^tC\left(\|\wi \X_1(s)\|_{\alpha},\|\wi \X_2(s)\|_{\alpha}\right)\d t\right),
\end{align*}
for any $t\in[0,T)$.  Thus, if $\wi\X_1(0)=\wi\X_2(0)=\x$, then $\wi\X_1(t)=\wi\X_2(t)$ for all $t\in[0,T)$, $\mP$-a.s. If there is a jump at $t=T$, we can extend the analysis for some time $\wi T>T$ and can conclude the uniqueness for all $t\in[0,\wi T]$. Otherwise the uniqueness follows immediately. 

Moreover, $\wi\U(\cdot)$ is an optimal control. Indeed, from (\ref{3.122a}), it is clear that $$v(T,\x)\leq \mathcal{J}(0,T;\x,\U),\text{ for all }\U\in\mathrm{L}^2(\Omega;\mathrm{L}^2(0,T;\mH)).$$ Since $\wi\X(\cdot)$ is the solution corresponding to the control $\wi\U(t)={\mathcal G}( \mathrm{D}_{\x}v(T-t,\wi \X(t)))$, from (\ref{3.122a}) we also have $v(T,\x)=\mathcal{J}(0,T;\x,\wi\U)$. From the above inequality it is clear that $\wi\U(\cdot)$ is an optimal control and $(\wi\X,\wi\U)$ is an optimal pair.
\end{proof}

\medskip\noindent
{\bf Acknowledgments:}  M. T. Mohan would  like to thank the Department of Science and Technology (DST), India for Innovation in Science Pursuit for Inspired Research (INSPIRE) Faculty Award (IFA17-MA110) and Indian Institute of Technology Roorkee, for providing stimulating scientific environment and resources.

\begin{appendix}
\renewcommand{\thesection}{\Alph{section}}
\numberwithin{equation}{section}
\section{Dynamic Programming Principle} \label{DPP}
\setcounter{section}{1}
\setcounter{equation}{0}
\renewcommand{\thesection}{\Alph{section}} 

We briefly give  the key steps in deriving the stochastic Dynamic Programming Principle and  HJB equation.  Let us define 
\begin{align*}
\mathfrak{L}(\X,\U)=\|\text{curl } \X\|^2+\frac{1}{2}\|\U\|^2, \ \   \Psi(\X(T))=\|\X(T)\|^2, \ \
\F(\X,\U)=-\A\X+\B(\X)+\K\U.
\end{align*} 
Let $T>0$ be given. Then for  any $t\in[0,T),$ consider the problem of minimizing the cost functional 
\begin{align} \label{a1}
{\mathcal J}(t,T;\x,\U)= \mE_t\left[\int_t^T\mathfrak{L}(\X(s),\U(s))\d s+ \Psi(\X(T)) \right],
\end{align}
over all controls $\U\in \mU^{t,T}_R$  and  $\X(\cdot)$ satisfying   the following state equation  
\begin{equation} \label{a2}
\left\{
\begin{aligned}
\d  \X(s)&=\F(\X(s),\U(s))\d s+\A^{-\frac{\var}{2}}\d\W(s) +\int_{\mZ}\G(s,z)\wi{\uppi}(\d s,\d z),  \ s\in [t,T], \\
\X(t)&=\x, \ \x\in \mH.
\end{aligned} 
\right.
\end{equation}
Here $\mE_t[\X(s)]=\mE[\X(s) | \X(t)=\x]$, for $s\geq t$  and the admissible control set is defined as
\begin{equation*}
\mU^{t,T}_{R}=\Big\{\U\in \mathrm{L}^2(\Omega, \mathrm{L}^2(t,T;\mH)):\|\U(\cdot, s)\|\leq R,\mP\text{-a.s.,} \ \mbox{and}  \  \U \ \mbox{is adapted to } \mF_{t,s}\Big\},
\end{equation*}
where $\mF_{t,s}$ is the $\sigma$-algebra  generated by the paths of $\W$ and random measures $\uppi$ upto time $s,$ i.e.,  $\sigma\{\W(r); t\leq r\leq s\}$ and $\sigma\{\uppi(S); S\in \mB([t,s]\times\mZ)\}.$      The value function of the above control problem is defined as  
\begin{equation*}
\left\{
\begin{aligned}
\mathscr{V}(t,\x)&=\inf_{\U\in\mU^{t,T}_R} {\mathcal J}(t,T;\x,\U),   \text{ for all }t\in[0,T), \ \x\in \mH, \\ 
\mathscr{V}(T,\x)&=\|\x\|^2.
\end{aligned} 
\right.
\end{equation*}
Following the dynamic programming strategy,  any admissible control  $\U\in\mU^{t,T}_R$ is the combination of  controls in $ \mU^{t,\tau}_R$ and $ \mU^{\tau,T}_R$ for any $0\leq t< \tau< T.$ More precisely,  suppose the processes $\U_1(s)$ and $\U_2(s)$ be the restriction of the control process $\U(s)$ to the time intervals $[t,\tau] $ and $[\tau,T] $  respectively, i.e.,
\begin{equation*}
\U(s)= (\U_1\oplus\U_2)(s)
=\left\{
\begin{array}{ll}
\U_t(s),& s\in[t,\tau], \\
\U_\tau(s),& s\in[\tau, T]. \\
\end{array}\right.
\end{equation*}
Accordingly, the  admissible control set is written as $\mU^{t,T}_R=\mU^{t,\tau}_R\oplus \mU^{\tau,T}_R.$ Note also that the controls $\U_1(s)$ and $ \U_2(s)$ are adapted to $\mF_{t,s}$ and $\mF_{\tau,s}$ respectively. 

The system state $\X(\cdot)$ is determined by \eqref{a2} with $\U(s)=(\U_1\oplus\U_2)(s)\in \mU^{t,T}_R.$  We decompose the system state as $\X(s)=(\X_1\oplus  \X_2)(s),$ where $\X_1$ and $\X_2$ satisfy 
\begin{equation}\label{a4}
\left\{
\begin{aligned}
\d  \X_1(s)&=\F(\X_1(s),\U_1(s))\d s+\A^{-\frac{\var}{2}}\d\W(s) +\int_{\mZ}\G(s,z)\wi{\uppi}(\d s,\d z),  \ s\in [t,\tau], \\
\X_1(t)&=\X(t)=\x\in \mH, 
\end{aligned}
\right.
\end{equation}
and 
\begin{equation} \label{a5}
\left\{
\begin{aligned}
\d  \X_2(s)&=\F(\X_2(s),\U_2(s))\d s+\A^{-\frac{\var}{2}}\d\W(s) +\int_{\mZ}\G(s,z)\wi{\uppi}(\d s,\d z),  \ s\in [\tau,T],\\
\X_2(\tau)&=\X_1(\tau)=\X(\tau).
\end{aligned}
\right.
\end{equation} 
By the tower property of conditional expectation 
$$\mE_t\Big[\mE_t\big(\mathfrak{L}(\X_2(s),\U_2(s))\ |  \mF_{t,\tau}\big)\Big] =\mE_t\big[\mathfrak{L}(\X_2(s),\U_2(s))\big], \tau\leq s\leq T, $$   we write
\begin{align*} 
\mathscr{V}(t,\x)&= \inf_{\U\in\mU^{t,T}_R }\mE_t\left\{\int_t^\tau\mathfrak{L}(\X(s),\U(s))\d s + \int_\tau^T\mathfrak{L}(\X(s),\U(s))\d s+\Psi(\X(T))\right\}  \\ 
&= \inf_{\substack{\U_1\in\mU^{t,\tau}_R, \U_2\in\mU^{\tau,T}_R \\  \X_2(\tau)=\X_1(\tau) }} \mE_t\left\{ \int_t^\tau\mathfrak{L}(\X_1(s),\U_1(s))\d s \right.\\
&\left. \quad+ \mE_t\left[ \int_\tau^T\mathfrak{L}(\X_2(s),\U_2(s))\d s+\Psi(\X_2(T))\big|\mF_{t,\tau}\right]\right \}. 
\end{align*}
Using the Markovian property of the process $\X(\cdot)$, one can get 
for $\tau\leq s\leq T,$   $$\mE_t\Big[\mathfrak{L}(\X_2(s),\U_2(s))\ |  \mF_{t,\tau}\Big] =\mE_{\tau}\Big[\mathfrak{L}(\X_2(s),\U_2(s))\ | \ \X_2(\tau)=\X(\tau) \Big]$$ and the same reasoning is true for $\Psi(\X_2(T))$ as well. It leads to  
\begin{align} \label{a21}
{\mathcal J}(\tau,T;\X_2(\tau),\U_2)= \mE_t\left\{\int_\tau^T\mathfrak{L}(\X_2(s),\U_2(s))\d s+\Psi(\X_2(T)) \ \Big | \mF_{t,\tau}\right\}.
\end{align} Besides, for more details on the case of continuous diffusion, one can  refer to \cite{Y} (and also \cite{G}) and that can be modified to this case. Hence
\begin{align}\label{a3}
\mathscr{V}(t,\x)&=  \inf_{\U_1\in\mU^{t,\tau}_R}  \mE_t\left\{\int_t^\tau\mathfrak{L}(\X_1(s),\U_1(s))\d s\right\}\no  \\
&\quad+ \mE_t\left\{\inf_{\substack{\U_2\in\mU^{\tau,T}_R \\ \X_2(\tau)=\X_1(\tau) }}\mE_t\left[\int_\tau^T\mathfrak{L}(\X_2(s),\U_2(s))\d s+\Psi(\X_2(T))\Big| \mF_{t,\tau}\right]\right\} \no \\ 
&= \inf_{\U_1\in\mU^{t,\tau}_R}  \mE_t\left\{\int_t^\tau\mathfrak{L}(\X_1(s),\U_1(s))\d s  +\mathscr{V}(\tau,\X_1(\tau)) \right\}. 
\end{align}
Thus, we have proved the following  DPP {(or \it Bellman's principle of optimality)}
\begin{align} \label{a6}
\mathscr{V}(t,\x)=\inf_{\U\in\mU^{t,\tau}_R}  \mE\left\{\int_t^\tau\mathfrak{L}(\X(s),\U(s))\d s  +\mathscr{V}(\tau,\X(\tau)) \big| \X(t)=\x\right\}.
\end{align}

\subsection{Dynamic Programming Equation} 

Rewrite \eqref{a6} as follows
\begin{align}\label{a7}
\inf_{\U\in\mU^{t,\tau}_R}  \mE\left\{\int_t^\tau\mathfrak{L}(\X(s),\U(s))ds  +\mathscr{V}(\tau,\X(\tau)) - \mathscr{V}(t,\x) \big| \X(t)=\x\right\}=0.
\end{align}
Now the HJB equation can be formally obtained by applying the  It\^o formula for $\mathscr{V}(\tau,\X(\tau)) - \mathscr{V}(t,\x).$ 

\begin{Rem}[It\^o Formula]
	For any $t\geq 0$ and   $\Phi\in \C_b^{1,2}([0,T];\mathcal{D}(\mathscr{A}^{\U})),$ the following relation holds (in fact,  $\Phi$ could be a function such that whose Gate\^aux derivatives $\mathrm{D}_{\x}\Phi$ and $\mathrm{D}^2_\x\Phi$ are H\"older continuous with  exponent $\delta=1$, see, \cite{Me}): 
	\begin{align} \label{a8}
	\Phi(\tau,\X(\tau))-\Phi(t,\x)=\int_t^\tau\big[\mathrm{D}_s\Phi(s,\X(s))+\mA^{\U}\Phi(s,\X(s))\big]\d s + \M_\tau,  \ \mP{\text -a.s.},
	\end{align}
	where $\mA^{\U}\Phi$ is the second order partial integrodifferential operator
	\begin{align*}
	\mA^{\U}\Phi(s,\X(s))& = \la \mathrm{D}_{\x}\Phi(s,\X(s)), \F(\X(s),\U(s))\ra +\frac{1}{2} \tr(\A^{-\var} \mathrm{D}^2_\x\Phi(s,\X(s)) \\
	&\quad+\iZ\big[\Phi(s,\X(s)+\G(s,z))-\Phi(s,\X(s))-\la \G(s,z),\mathrm{D}_{\x}\Phi(s,\X(s))\ra\big]\mu(\d z)
	\end{align*}
	and $\M_{\tau}$ is the martingale given by
	\begin{align*}
	\M_\tau=&\int_t^\tau\la  \mathrm{D}_{\x}\Phi(s,\X(s)),\A^{-\frac{\var}{2}}\d\W(s)\ra \\
	&+ \int_t^\tau\iZ \big[\Phi(s,\X(s-)+\G(s,z))-\Phi(s,\X(s-))\big]\wi{\uppi}(\d s,\d z).
	\end{align*}
\end{Rem}
Taking  $\Phi(\tau,\X(\tau))=\mathscr{V}(\tau,\X(\tau))$ (of course by assuming required smoothness and boundedness on $\mathscr{V}$)  and plug this back into \eqref{a7}, and use the fact that martingale $\M_{\tau}$ has a zero mean to obtain the following
\begin{equation}\label{a9}
\inf_{\U\in\mU^{t,\tau}_R}  \mE\left\{\int_t^\tau\big[\mathfrak{L}(\X(s),\U(s))+\mathrm{D}_s\mathscr{V}(s,\X(s))+\mA^{\U}\mathscr{V}(s,\X(s))\big]\d s \big| \X(t)=\x\right\}=0.
\end{equation}
Finally, take $\tau=t+h, h>0$ and divide by $h.$  Passing $h\to 0$ and using the conditional expectation, we formally obtain the HJB equation 
\begin{equation}\label{a10}
\left\{
\begin{aligned}
&\mathrm{D}_t\mathscr{V}(t,\x)+\inf_{\U\in\mathfrak{U}} \big\{\mA^{\U}\mathscr{V}(t,\x) +\mathfrak{L}(\x,\U(t)) \big\}=0, \ \ t\in[0,T), \\
&\mathscr{V}(T,\x)=\|\x\|^2,  \ \x\in \mH.
\end{aligned}
\right.
\end{equation}
Moreover, setting $v(t,\x)=\mathscr{V}(T-t,\x),$ one can obtain the following initial value problem
\begin{equation}\label{a11}
\left\{
\begin{aligned} 
\mathrm{D}_tv(t,\x)=&  \mathscr{H}(\x,t,\mathrm{D}_{\x}v,\mathrm{D}^2_\x v), \ \  \ t\in (0,T), \\
v(0,\x)=& \|\x\|^2, \ \x\in \mH, \no
\end{aligned} 
\right.\end{equation}
where $\mathscr{H}$ is given by
\begin{align*}
\mathscr{H}(\x,t,\mathrm{D}_{\x}v,\mathrm{D}^2_\x v) =&\frac{1}{2}\tr(\A^{-\var}\mathrm{D}_{\x}^2v)+\la - \A \x+\B(\x),\mathrm{D}_{\x}v\ra  \\
&+\iZ\big[v(t,\x+\G(t,z))-v(t,\x)-\la \G(t,z),\mathrm{D}_{\x}v\ra\big]\mu(\d z) \\
&+\|\text{curl } \x\|^2+\inf_{\U\in\mathfrak{U}}\left\{\la \U(t),\mathrm{D}_{\x}v\ra+\frac{1}{2}\|\U(t)\|^2\right\}.
\end{align*}  

\end{appendix}


\begin{thebibliography}{30} 
\bibitem{AA} P. Agarwal, U. Manna and M. Debopriya, 	Stochastic control of tidal dynamics equation with L\'evy noise, \emph{ Appl. Math. Optim.,} 79(2019), 327–396.
\bibitem{A} D. Applebaum, {\it L\'evy Processes and Stochastic Calculus,} Cambridge University Press, 2$^{nd}$ Edition, Cambridge, 2009. 
\bibitem{Ba} M. Bardi and I. Capuzzo Dolcetta, {\it Optimal Control and Viscosity Solutions of Hamilton-Jacobi-Bellman Equations} , Birkh\"auser, Berlin, 1997. 
\bibitem{Be}   D.P. Bertsekas  and J. N. Tsitsiklis, Neuro-Dynamic Programming, Athena Scientific, 1996.
\bibitem {BB} B. Birnir, \emph{The Kolmogorov-Obukhov Theory of Turbulence: A Mathematical Theory of Turbulence}, Springer, 2013.  
\bibitem{B}  Z. Brze\'zniak, E. Hausenblas, J. Zhu,  2D stochastic Navier-Stokes equations driven by jump noise, {\it Nonlinear Analysis,} 79 (2013), 122-139. 
\bibitem{BZ}  Z.  Brze\'zniak, W. Liu and, J. Zhu, Strong solutions for SPDE with locally monotone coefficients driven by L\'evy noise, \emph{Non. Anal. Real World Appl.}, 17 (2014), 283-310.
\bibitem{D} G. Da Prato and A. Debussche,  Control of the stochastic Burgers model of turbulence, {\it SIAM J. Control. Optim.,} 37(1999), 1123-1149.
\bibitem{Da} G. Da Prato and A. Debussche,  Dynamic programming for the stochastic Burgers equations, {\it Annali di Mat. Pura ed Appl.,} IV (2000), 143-174.  
\bibitem{Da1} G. Da Prato and A.Debussche,  Dynamic programming for the stochastic Navier-Stokes equations, {\it Math. Model. Nume. Anal.,} 34(2000), 459-475.  
\bi {DZ} G. Da Prato  and J. Zabczyk, {\it Stochastic Equations in Infinite Dimensions,}  Cambridge University Press, 2$^{nd}$ Edition, Cambridge, 2014. 
\bibitem{Do} Z. Dong and T. G. Xu, One dimensional stochastic Burgers equation briven by L\'evy processes, {\it J. Fun. Anal.,} 243 (2007), 631-678. 
\bibitem{DD} Z. Dong and Y. Xie,  Global solutions of stochastic 2D Navier-Stokes equations with L\'evy noise, \emph{Sci.  China Series A: Math.},
52 (2009), 1497-1524.
\bibitem{DX} Z. Dong and Y. Xie, Ergodicity of stochastic 2D Navier-Stokes equation with L\'evy noise, {\it J. Diff. Equa.,} 251 (2011), 196-222.
\bibitem{E} K. D. Elworthy and X. M. Li, Formulae for the derivative of heat semigroups, {\it J. Fun. Anal.,}  125 (1994), 252-286. 
\bibitem{Fab} G. Fabri, F. Gozzi, and A. Swiech {\it Stochastic Optimal Control in Infinite Dimension: Dynamic Programming and HJB Equations,} Springer-Verlag, New York, 2018.
\bibitem{FS}   H. O. Fattorini and S.S. Sritharan, Existence of optimal controls for viscous flow problems, {\it Proc. R. Soc. Lond. A},   439 (1992), 81-102.
\bibitem{FS1}   H. O. Fattorini and S.S. Sritharan, Optimal chattering controls for viscous flow, Non. Anal. Theory Meth. Appli., 25 (1995), 763-197. 
\bibitem{PFS} B. P. W. Fernando and S. S. Sritharan, Nonlinear filtering of stochastic Navier-Stokes equation with It\^o-L\'evy noise, {\it Stoch. Anal. Appl.}, 31 (2013), 381-426. 
\bibitem{Fo} C. Foias, O. Manley, R. Rosa and R. Temam, \emph{ Navier-Stokes Equations and Turbulence}, Cambridge University Press, 2004. 
\bibitem{Fu} A.V. Fursikov, \emph{Optimal Control of Distributed Systems, Theory and Applications,} Translations of Mathematical Monographs, Vol. 187, AMS, 1999.
\bibitem{G} F.Gozzi,  S.S. Sritharan and A.Swiech,  Bellman equations associated to the optimal feedback control of stochastic Navier-Stokes equations,  {\it Comm.  Pure Appl. Math.},  Vol. LVIII(2005), 0001-0030.   
\bibitem{I} A. Ichikawa, Some inequalities for martingales and stochastic convolutions, {\it Stoch. Anal. Appl.,} 4 (1986), 329-339.
\bibitem{AAI} A. A. Ilyin, 	On the spectrum of the stokes operator, \emph{Functional Analysis and Its Applications}, 43 (2009), 254--263.
\bibitem{J} J. Jacod and P. Protter, {\it Probability Essentials,} Springer, 2003.
\bibitem{L} O. A. Ladyzhenskaya, {\it The Mathematical Theory of Viscous Incompressible Flow,}  Gordon and Breach, New York, 1969.
\bibitem{MM} M. T. Mohan and S. S. Sritharan, Ergodic control of stochastic Navier-Stokes equation with L\'evy noise, {\it Comm. Stoch. Anal.,}	{\bf 10} (2016) 389-404.
\bibitem{MSS} M. T. Mohan, K. Sakthivel and S. S. Sritharan, Ergodicity for the 3D Stochastic Navier-Stokes
Equations Perturbed by L\'{e}vy Noise, {\it Mathematische Nachrichten},  292 (2019), 1056-1088.
\bibitem{Me} M. M\'{e}tivier,  {\it Stochastic Partial Differential Equations in Infinite Dimensional Spaces},
Quaderni, Scuola Normale Superiore, Pisa, 1988.
\bibitem{M} M.Metivier, {\it Semimartingales: A Course in Stochastic Process}, DeGruyer, Berlin, 1982.
\bibitem{P} S.Peszat and J.Zabczyk, {\it Stochastic Partial Differential Equations with L\'evy Noise,} Cambridge University Press, Cambridge, 2007.
\bibitem{Sa} K. Sakthivel and S.S. Sritharan,  Martingale solutions for stochastic Navier-Stokes equations driven by L\'evy noise, {\it Evol. Equa.  Control. Theory},  1(2012), 355-392.  
\bibitem{Sh} H. Sohr, \emph{The Navier-Stokes Equations:
	An Elementary Functional Analytic Approach}, Springer,  2001.
\bibitem{Sr0} S. S. Sritharan, Editor, {\it Optimal control of viscous flow}, SIAM, Philadelphia, 1998.
\bibitem{Sr} S. S. Sritharan, An introduction to deterministic and stochastic control of viscous
flow, {\it Optimal control of viscous flow}, SIAM, Philadelphia, 1998, 1-42.
\bibitem{Sr1} S. S. Sritharan, Dynamic programming of the Navier-Stokes equations, {\it Syst. Cont. Letters}, 16(1991), 299-307.
\bibitem{SSS} S. S. Sritharan,   Deterministic and stochastic control of Navier-Stokes equation with linear, monotone, and hyperviscosities,   
{\it Appl. Math. Optim.,} 41(2000), 255-308.
\bibitem{T} R.  Temam, {\it Navier-Stokes Equations and Nonlinear Functional Analysis,} SIAM, Philadelphia, 1995. 
\bibitem{Y} J. Yong and X.Y. Zhou, {\it Stochastic Controls: Hamiltonian systems and HJB Equations,} Springer-Verlag, New York, 1999.
\end{thebibliography}
\end{document}